\documentclass[11pt]{article}
\usepackage[margin=1in]{geometry}
\usepackage[T1]{fontenc}
\usepackage{lmodern,amsmath,amssymb,amsthm,mathtools,microtype}
\usepackage{xcolor}
\usepackage[authoryear,round]{natbib}
\setcitestyle{aysep={,}}

\makeatletter
\renewcommand\NAT@nmfmt[1]{{\NAT@up#1\unskip}}
\makeatother
\usepackage[colorlinks=true,linkcolor=black,citecolor=black,urlcolor=black]{hyperref}
\hypersetup{pdftitle={A Dimension-Free Bound on the Poincare Constant of Isotropic Log-Concave Measures},pdfauthor={Krishnakumar Balasubramanian and Shiva Kasiviswanathan}}
\newcommand{\E}{\mathbb E}
\newcommand{\R}{\mathbb R}
\newcommand{\HS}{\mathrm{HS}}
\DeclareMathOperator{\Tr}{Tr}
\DeclareMathOperator{\Var}{Var}
\DeclareMathOperator{\Cov}{Cov}
\DeclareMathOperator{\Sym}{Sym}
\newtheorem{theorem}{Theorem}[section]
\newtheorem{lemma}[theorem]{Lemma}
\newtheorem{proposition}[theorem]{Proposition}

\theoremstyle{definition}\newtheorem{definition}[theorem]{Definition}
\theoremstyle{remark}\newtheorem{remark}[theorem]{Remark}
\numberwithin{equation}{section}
\title{A Dimension-Free Bound on the Poincar\'e Constant\protect\\ of Isotropic Log-Concave Measures}

\author{
\begin{tabular}{cc}
Krishnakumar Balasubramanian$^{1,2}$ &
Shiva Kasiviswanathan$^{1}$ \\
\texttt{kbala@ucdavis.edu} &
\texttt{kasivisw@gmail.com}
\end{tabular}
\\[1em]
$^1$Amazon\thanks{This work does not relate to the authors' positions at Amazon.}\\
$^2$Department of Statistics, University of California, Davis
}

\date{}
\begin{document}
\maketitle

\begin{abstract}
The Kannan--Lov\'asz--Simonovits (KLS) conjecture asserts that isotropic log-concave probability measures have Poincar\'e constants bounded by a universal constant, independently of dimension. In this work, we provide a deterministic variational proof that the Poincar\'e constant $C_P(\mu)\le25$, where $\mu$ is any isotropic log-concave probability measure. The proof centers on mean-zero integration of compatible symmetric tensor fields. Starting from constant tensors, we obtain factorially normalized Appell polynomials by iterating this procedure. Using moment estimates and a quadratic variance inequality, we establish geometric bounds on these iterated primitives through a \emph{joint} variational argument over the measure and test function. At a maximizing pair, stationarity in the test function controls the weighted sum of the two highest-order cumulant terms arising from the measure variation. To extend these bounds to arbitrary compatible fields, we use a weighted Helmholtz--Hodge decomposition of weighted divergence, the formal adjoint of differentiation. Controlling the curl in the decomposition yields a curvature comparison for the integration operators, uniformly in tensor order. Combining this comparison with the Appell bounds gives a bound on the full integration operator and hence the stated Poincar\'e inequality. The estimate is independent of the auxiliary positive curvature, allowing approximation to general isotropic log-concave measures.
\end{abstract}

\section{Introduction}\label{sec:introduction}

Consider the class of isotropic log-concave probability measures on
$\R^n$, denoted by $\mathsf C_n$. Each $\mu\in\mathsf C_n$ has a density
$Z^{-1}e^{-V}$ for a convex potential $V$, possibly taking the value
$+\infty$, and satisfies $\E_\mu X=0$ and $\Cov_\mu X= \E_\mu X X^\top = I$.
The Poincar\'e constant is given by
\[
 C_P(\mu)=\sup_f
       \frac{\Var_\mu f}{\int|\nabla f|^2\,d\mu},
\]
where the supremum is over locally Lipschitz $f\in L^2(\mu)$
with nonzero finite gradient energy. The Kannan--Lov\'asz--Simonovits (KLS) conjecture \citep{KLSOriginal}
asks whether
\[
 \sup_{\mu\in\mathsf C_n}C_P(\mu)<C
\]
for a universal constant $C<\infty$, independent of $n$.

In its original geometric formulation, the KLS conjecture asks whether hyperplane
cuts approximate the optimal isoperimetric ratio up to a universal
factor. Its connection to the formulation above is through the
Cheeger isoperimetric constant $h_\mu$; see, for example, \citet[Equations 1.1 and 1.2]{KLSOriginal} or \citet[Section~1.1]{Milman}. Indeed, writing
\[
 \psi_n=\sup_{\mu\in\mathsf C_n}h_\mu^{-1},
 \qquad\text{we have}\qquad
 \psi_n^2\asymp\sup_{\mu\in\mathsf C_n}C_P(\mu).
\]
Thus, KLS equivalently asks whether $\psi_n$ is bounded by a
universal constant, independently of dimension. Here and below,
comparison constants are universal. This equivalence, and its
relation to concentration, are discussed by~\citet[Theorem~1.5]{Milman}.

KLS bounds have consequences for both probability and computation.
They control the conductance and mixing of geometric random walks
used to sample log-concave distributions \citep{LeeVempalaLocalization},
and enter the complexity of rounding and volume algorithms for convex
bodies \citep{JiaLaddhaLeeVempala}. A uniform spectral gap also yields
uniform exponential concentration of Lipschitz test functions
\citep{Milman}. For isotropic $X\sim\mu$, the test function $f(x)=|x|^2$
gives $\Var(|X|^2)\le4nC_P(\mu)$. This connects KLS to the thin-shell
problem and, through thin-shell concentration, to the central limit
phenomenon for convex bodies \citep{KlartagCLT}.

\subsection{Related Work and Progress}
The KLS conjecture, formulated in 1995, has been a central problem in high-dimensional convex geometry for more than three decades.
The seminal deterministic localization argument of \cite{KLSOriginal} proved that $\psi_n\lesssim\sqrt n$.
Bobkov's comparison between isoperimetry and fluctuations of the
Euclidean norm \citep{BobkovIsoperimetry}, combined with the thin-shell
estimate of \cite{GuedonMilman}, improved this to
$\psi_n\lesssim n^{5/12}$. \cite{EldanLocalization} introduced stochastic localization and
related thin-shell estimates to the spectral gap up to logarithmic
factors. Together with \cite{GuedonMilman}, this yielded
$\psi_n\lesssim n^{1/3}\sqrt{\log n}$.

Subsequent general KLS bounds developed this stochastic approach. \cite{LeeVempalaLocalization} obtained
$\psi_n\lesssim n^{1/4}$, and in a landmark work, \cite{ChenAlmostConstant} established
the first subpolynomial bound, $\psi_n=n^{o(1)}$. \cite{KlartagLehecPolylog} proved
$\psi_n\lesssim(\log n)^5$, giving the first polylogarithmic bound. \cite{JambulapatiLeeVempala} improved the
exponent to $3.2226$. \citet[Theorems~1.2--1.3]{KlartagLogarithmic} then obtained
$\psi_n\lesssim\sqrt{\log n}$ by combining localization with an improved
Lichnerowicz inequality for log-concave measures.

Progress on the thin-shell problem supplied further stronger moment estimates.
Building on the covariance estimates by \cite{GuanSlicing}, \cite{KlartagLehecThinShell} proved a dimension-free thin-shell
bound using a parallel coupling associated with stochastic localization. \citet[Theorem~1.1]{ChenKlartag} subsequently established
the sharp inequality $\Var(|X|^2)\le8n$ for isotropic log-concave $X$ by an elliptic method.
Their proof analyzes the Monge--Amp\`ere equation for moment measures,
using the construction of  \cite{CorderoKlartagMoments} and the associated weighted geometry
developed by \cite{KlartagMomentGeometry}.

The same elliptic method also controls directional third moments.
Define
\[
 \kappa_n=\sup_X\sup_{|\theta|=1}
 \bigl\|\mathbb E[(\theta\cdot X)X\otimes X]\bigr\|_{\HS},
\]
where $X$ ranges over isotropic log-concave vectors in $\R^n$. \citet[Section~2 and Corollary~2.7, Equation~(47)]{BizeulMM}
explains how an anisotropic modification of the argument by \cite{ChenKlartag}, gives $\kappa_n^2\le16$. \citet[Theorem~1.2]{Letwin} improved the corresponding sharp
Poincar\'e inequality for every quadratic form:
\[
 \Var(X^\top T X)\le8\Tr(T^2)
 \qquad\text{for every symmetric matrix $T$.}
\]
This improves the third-moment bound to $\kappa_n^2\le8$
\citep[Proof of Theorem~1.1]{Letwin}.

\cite{Letwin} further combined these moment estimates with stochastic localization.
The third moments govern covariance evolution, while a growing quadratic term in the
potential gives the localized measures positive curvature. Covariance control and Klartag's refined
Lichnerowicz inequality
\cite[Theorem~1.3 and Corollary~3.2]{KlartagLogarithmic}
transfer spectral estimates from the localized measures to the
original law. In particular, a dimension-free bound on
$\kappa_n$ yields
$\psi_n\lesssim(\log n)^{1/4}$, or equivalently
$C_P(\mu)\lesssim\sqrt{\log n}$ for isotropic log-concave $\mu$
\citep[Theorem~1.1 and its proof]{Letwin}.

During the final stages of preparation of this manuscript, a recent work by \citet[Version 1, Theorem~7.1]{SongZhangV1} combined higher-degree polynomial estimates with an iteration of curvature bounds to obtain
\[
 \psi_n\le C\,4^{\log^*(n+2)}.
\]
Here, $\log^*x$ is the least number of successive natural logarithms
needed to reduce $x$ to at most one\footnote{We note here that our earlier argument based on a variational Stein-kernel formulation~\citep{CFP2019} with a structured tensor recovery from partial symmetries bound, yielded a bound for the reciprocal Cheeger constant with a fourth-root dependence on the iterated logarithm of the dimension.}. Their argument combines higher-degree polynomial variance estimates
and a spectral curvature comparison with stochastic localization. Concurrent to our work,~\cite{BizeulKlartagLehec} and \citet[Version 2]{SongZhangV2} obtain dimension-free estimates on the Poincar\'e constants of isotropic log-concave measures. Finally, we also remark that \citet{MikulincerZadik} proved a dimension-free Poincar\'e
inequality for isotropic \emph{unconditional} log-concave measures which form a subclass of isotropic log-concave measures invariant under
arbitrary coordinate sign changes.



\subsection{Our Result and Proof Ingredients}

We now state our main result and elaborate on the key proof ingredients.

\begin{theorem}\label{thm:main}
For every $n\ge1$,
\begin{equation}\label{eq:main-bound}
 \sup_{\mu\in\mathsf C_n}\ \sup_f
 \frac{\Var_\mu f}{\int|\nabla f|^2\,d\mu}
 \le25,
\end{equation}
where the inner supremum is over locally Lipschitz $f\in L^2(\mu)$
with nonzero finite gradient energy. 
\end{theorem}

The Poincar\'e--Cheeger equivalence for log-concave measures
\cite[Theorem~1.5]{Milman} also gives the isoperimetric formulation of KLS from Theorem~\ref{thm:main}.

Our proof proceeds as follows. We start with a regular isotropic law $\mu$, with
$D^2V\succeq aI$ for some $a>0$; the final bound will be independent
of $a$. All field norms below are in $L^2(\mu)$, and operator norms
are induced by these norms. Section~\ref{sec:notation} specifies
the regularity assumptions and the precise spaces.

\textit{Scalar integration problem:} For a smooth test function $f$, first let
$b=\E_\mu\nabla f$ and $v=\nabla f-b$. Let operator $J_0$ act on mean-zero gradient fields by assigning to each field its unique mean-zero primitive. Then, we have
\[
 J_0v=f-\E_\mu f-b\cdot x,\qquad
 \|\nabla f\|_2^2=|b|^2+\|v\|_2^2.
\]
Due to isotropy, we have $\|b\cdot x\|_2=|b|$. Furthermore, by the triangle inequality
and Cauchy--Schwarz, we have
\[
 \sqrt{\Var_\mu f}
 \le |b|+\|J_0\|\|v\|_2
 \le\sqrt{1+\|J_0\|^2}\,\|\nabla f\|_2.
\]
Consequently,
\begin{equation}\label{eq:intro-scalar-reduction}
 C_P(\mu)\le1+\|J_0\|^2.
\end{equation}
The main difficulty is hence to bound $J_0$ uniformly. Note that Brascamp--Lieb inequality \citep{carlen-cordero-lieb} provides a preliminary bound $\|J_0\|^2\le a^{-1}$, which deteriorates as $a$ tends to zero. Note that $J_0$ acts on arbitrary mean-zero compatible vector fields. We divide the subsequent argument into two steps: bounding repeated integration starting from constant tensors, and transferring these bounds to general compatible inputs through an operator comparison.\vspace{0.1in}

\textit{Compatible spaces and repeated integration:}
For $r\ge1$, a square-integrable symmetric tensor field
$U:\R^n\to\Sym^r(\R^n)$ is \emph{compatible} if
\[
 \partial_iU_{jI}=\partial_jU_{iI}
 \quad\text{for every ordered tuple $I$ of length $r-1$},
\]
with derivatives understood in the sense of distributions \citep[Section~6.6]{Liebloss}. At order one, compatibility is the usual curl-free condition.  At order two, a smooth compatible symmetric matrix field is a
Hessian. Successive integration then follows the familiar sequence
$D^2h\mapsto\nabla h\mapsto h$, with integration constants fixed
by centering. 

Let $\mathcal G_r$ denote the mean-zero compatible fields for $r\ge1$, and set $\mathcal G_0=L^2_0(\mu)$. Under our regularity assumptions, Lemma~\ref{full-primitive} establishes the existence of global compatible primitives in $L^2(\mu)$. Since the kernel of differentiation consists of constant tensors, imposing zero mean ensures uniqueness. We thus obtain bounded integration operators
\[
 J_r:\mathcal G_{r+1}\longrightarrow\mathcal G_r,\quad
 \nabla J_rU=U,\quad \E_\mu J_rU=0,\quad
 \|J_r\|\le a^{-1/2},
 \qquad r\ge0.
\]
The curvature comparison discussed below relates $J_r$ to $J_{r+1}$, with an additional contribution from constant inputs. To express and iterate this comparison across all tensor orders, we assemble the maps $J_r$ into a bounded operator $J$ on the Hilbert direct sum:
\[
 \mathcal H=\bigoplus_{r\ge0}\mathcal G_r,\qquad
 (JU)_r=J_rU_{r+1},\qquad
 \|J\|=\sup_{r\ge0}\|J_r\|\le a^{-1/2}.
\]
Since $J_0$ is one block of $J$, the scalar reduction above gives
\begin{equation}\label{eq:intro-graded-reduction}
 C_P(\mu)\le1+\|J_0\|^2\le1+\|J\|^2.
\end{equation}
Thus it suffices to prove $\|J\|^2\le24$.\vspace{0.1in}

\textit{Appell polynomials as explicit primitives:}
For a compatible field $F$ of order $r+1$, the orthogonal split
$F=T+F_0$, where $T=\E_\mu F$ and $F_0=F-T$, separates constants
from mean-zero fields. The constant part has the explicit centered primitive
$L_rT(x)=T\mathbin{\lrcorner}x$, where $\lrcorner$ denotes
contraction of tensor indices (see \eqref{eq:contraction} for the definition). Thus $L_rT+J_rF_0$ is the centered
primitive of $F$. Let $L$ collect the constant-input maps at all
modes; isotropy gives $\|L\|=1$. Starting from a constant and
iterating centered integration necessarily produces polynomials.
In one dimension, suppressing tensor modes, the first integrations are
\[
 1\ \xrightarrow{L}\ x\
 \xrightarrow{J}\ \frac{x^2-1}{2}\
 \xrightarrow{J}\ \frac{x^3-3x-\E_\mu X^3}{6}.
\]
Differentiation and centering uniquely determine this sequence:
$p_0=1$, $p_d'=p_{d-1}$ and $\E_\mu p_d=0$ for $d\ge1$.
These are the factorially normalized Appell polynomials \citep{AnshelevichAppell}.
For the standard Gaussian law, these are the factorially
normalized Hermite polynomials. 
We remark here that in principle, other
decompositions including spectral decompositions or orthogonal
polynomials are possible. The advantage of Appell polynomials is
the exact derivative recursion, which orthogonal polynomials
need not satisfy for a general law.

In $\R^n$, let $M_\mu(z)=\E_\mu e^{z\cdot X}$. The Appell tensors are defined by the classical generating function near $z=0$ \citep{AnshelevichAppell}
\[
 \frac{e^{z\cdot x}}{M_\mu(z)}
 =\sum_{d\ge0}
   \frac{\langle A_d^\mu(x),z^{\otimes d}\rangle}{d!}.
\]
Taking expectation yields $\E_\mu A_d^\mu=0$ for $d\ge1$.
For $T\in\Sym^{r+d}(\R^n)$, the gradient of
$T\mathbin{\lrcorner}A_d^\mu/d!$ is
$T\mathbin{\lrcorner}A_{d-1}^\mu/(d-1)!$, a fully symmetric field of order $r+1$. These centered contractions are therefore compatible. Uniqueness of centered primitives gives
\[
 J_r\cdots J_{r+d-2}L_{r+d-1}T
 =\frac{T\mathbin{\lrcorner}A_d^\mu}{d!},
 \qquad r\ge0,\quad d\ge1,
\]
where the product of $J$'s is empty for $d=1$. Define
\[
 c_d(\mu):=
 \sup_{\substack{T\in\Sym^d(\R^n)\\\|T\|_{\HS}=1}}
 \left\|\frac{\langle T,A_d^\mu\rangle}{d!}\right\|_2.
\]
Applying this bound componentwise controls every tensor output without an additional factor. The scalar output block gives the reverse bound, so $\|J^{d-1}L\|=c_d(\mu)$.
Since $c_d(\mu)$ depends only on moments through order $2d$, this identity converts moment estimates into bounds on repeated centered integration starting from constants. The curvature comparison below transfers these bounds to arbitrary compatible fields.\vspace{0.1in}

\textit{Joint variational argument:} The polynomial norms admit the dual representation
\[
 B_d^\mu(f)=\E_\mu[fA_d^\mu],\qquad
 c_d(\mu)=\frac1{d!}\sup_{\|f\|_2=1}\|B_d^\mu(f)\|_{\HS}.
\]
To study variations of the measure, we use exponential tilts \citep[Section~3, Equation~(3.1)]{KlartagLogarithmic}. For $\theta$ near zero, define the linear tilt
\[
 d\mu_\theta(x)=\frac{e^{\theta\cdot x}}{M_\mu(\theta)}\,d\mu(x).
\]
The Appell generating function then gives
$B_d^\mu(f)=\left.D_\theta^d\E_{\mu_\theta}f\right|_{\theta=0}$; see, for example, \citet[Section~2.3, Equation~(11)]{AnshelevichAppell}. Thus the coefficient tensors describe derivatives of the tilted expectation.


The measure variation produces a cumulant term beyond the
induction hypothesis. Maximizing also over the test function
gives a stationarity identity that controls the combination
of terms in which it appears.

For the induction step, fix $k\ge2$ and suppose the coefficient
bounds are known through degree $k$. Let $h_k$ be the known
upper bound for $\|B_k^\mu(f)\|_{\HS}^2$ when $\mu$ is isotropic
and $\|f\|_2=1$. For such pairs, we consider
\[
 \mathcal F(\mu,f)
 =\frac{\|B_{k+1}^\mu(f)\|_{\HS}^2}
 {\phi\!\left(\|B_k^\mu(f)\|_{\HS}^2/h_k\right)},
\]
where $\phi$ is positive, decreasing, concave, and $C^2$
on a neighborhood of $[0,1]$.
The numerator is the next coefficient energy to be bounded;
the argument of $\phi$ lies in $[0,1]$ by the induction
hypothesis. Concavity provides the sign needed in the measure
variation below. For a general nonzero $f$, we divide both
coefficient energies by $\|f\|_2^2$, making the quotient
invariant under rescaling of the test function.

The tilted law need not remain isotropic. We therefore define the quotient using covariance-normalized energies, obtained by applying $\Sigma^{-1/2}$ to every coefficient index before taking the squared norm, where $\Sigma$ is the covariance of the law under consideration. Tilting may change the quotient's value, but subsequent centering and whitening preserve it when the test function is composed with the inverse affine map.

We maximize the quotient \emph{jointly} over isotropic log-concave laws and nonzero test functions. Because both variables are free, the pairs obtained by centering and whitening remain admissible. Orthogonal projection onto polynomials of degree at most $k+1$ preserves the two coefficient tensors and cannot decrease the quotient. After normalizing $\|f\|_2=1$, compactness in the fixed dimension therefore yields a maximizing pair $(\mu_*,f_*)$.

At this pair, we first keep $f_*$ fixed in the original coordinates and perturb $\mu_*$ by linear tilts and by adding $s|x|^2/2$, $s\ge0$, to its potential \citep[Section~3]{KlartagLogarithmic}. These perturbations preserve log-concavity, and affine normalization produces admissible pairs. The necessary conditions for a maximum therefore give
an inequality containing two cumulant contractions,
of orders $k+1$ and $k+2$. However, estimating the higher-order contraction separately would require a coefficient bound beyond the induction hypothesis.

We therefore retain both contractions and, with $\mu_*$ fixed, vary $f_*$ in the linear direction $x\mapsto B_1^{\mu_*}(f_*)\cdot x$. The stationarity identity~\eqref{eq:barrier-terminal-cancellation} expresses their weighted sum in terms of coefficient energies and the squared norm of the linear coefficient, which isotropy bounds by one. Because both variations are evaluated at the same maximizing pair, the weights in this identity agree with those in the measure inequality. We can therefore control the combined contribution without a separate estimate at the next cumulant order.

The third-moment estimates of Section~\ref{sec:elliptic-inputs} control the covariance terms arising in the measure variation. Isotropy and the quadratic variance inequality of~\citet{Letwin} provide the initial bounds at degrees one and two. Each subsequent induction step reduces to a scalar inequality for the concave barrier. We construct the barriers by inverting decreasing branches of a quadratic polynomial for degree three and a single cubic polynomial for all higher degrees. Writing $c_d=\sup_{\mu\in\mathsf C_n}c_d(\mu)$, we obtain
\begin{equation}\label{eq:intro-coefficients}
 c_1=1,\qquad c_2^2\le2,\qquad
 c_d^2\le B K^{d-3}\quad(d\ge3),
 \qquad B=\frac{121}{4},\quad K=\frac{223}{10}.
\end{equation}

\textit{Curl correction and operator comparison:}
The polynomial estimates described above control $J^mL$. However, recall that the Poincar\'e bound requires control of $J$ on arbitrary mean-zero compatible fields. We derive the required operator comparison from a curvature estimate for the adjoint of the compatible gradient $D_r=\nabla$. On smooth potential fields, integration by parts gives $D_r^*F=P_{\mathcal G_r}\operatorname{div}_V F$. The projection removes any curl in the weighted divergence, and controlling the resulting energy loss is the central geometric step.

Compatibility ensures that all derivatives of $F$ cancel in the curl of $\operatorname{div}_V F$, leaving an expression involving only $F$ and $D^2V$. A weighted Helmholtz--Hodge estimate bounds the energy of the curl correction by the curvature contribution associated with $D^2V-aI$. Consequently, the projected adjoint satisfies
\[
 \|D_r^*F\|_2^2\ge\|\nabla F\|_2^2+a\|F\|_2^2,
 \qquad F=\nabla^{r+1}\psi,\quad \psi\in C_c^\infty(\R^n).
\]
The coefficient of $a$ is independent of the tensor order (Lemma~\ref{full-hodge}). Quadratic duality and density of the projected divergences then convert this differentiation estimate into the operator comparison for centered integration
\begin{equation}\label{eq:intro-operator}
 J^*J\preceq g_a(JJ^*+LL^*),
 \qquad \text{where}\qquad g_a(t)=\frac{t}{1+at}.
\end{equation}
Here $LL^*$ accounts for the constant component in the orthogonal decomposition of compatible inputs.\vspace{0.1in}

\textit{From Appell bounds to the operator norm:}
Section~\ref{sec:completion} combines the Appell bounds
with the positive curvature in \eqref{eq:intro-operator}
to bound $\|J\|$ directly. Suppose $\|J\|^2>24$. We choose a unit vector
$f\in\mathcal H$ with $\|J^*f\|^2>24$ and set
\[
 v_m=(J^*)^mf,\qquad
 u_m=\|v_m\|^2,\qquad
 p_m=\|L^*v_m\|^2.
\]
For nonzero iterates, we write $r_m=u_m/u_{m-1}$.
Testing \eqref{eq:intro-operator} on $v_m$ and using
Cauchy--Schwarz together with the variational bound
\eqref{eq:curvature-vector-bound} gives, after inverting $g_a$,
\[
 r_{m+1}\ge r_m+ar_m^2-\frac{p_m}{u_m}.
\]
The induction in Section~\ref{sec:completion} shows that
all iterates are nonzero and that $r_m\ge t>K$ for $m\ge2$,
with $t$ independent of $m$.
The Appell bounds give $p_m\le BK^{m-2}$, while
$u_m\ge u_1t^{m-1}$. Hence $p_m/u_m$ tends to zero
geometrically.
Since $a>0$ and $r_m\ge t$, the recurrence eventually yields
\[
 r_{m+1}\ge r_m+\frac{at^2}{2}.
\]
The ratios would therefore become unbounded, contradicting
$r_m\le\|J\|^2$.
Consequently,
\[
 \|J\|^2\le24,\qquad C_P(\mu)\le25.
\]
The bound is independent of $a$ and of the constants used
for compactness in a fixed dimension.
Approximation then removes the auxiliary curvature and
regularity assumptions, extending the Poincar\'e bound
to every isotropic log-concave law.

We have not fully optimized the numerical constant $25$,
and further refinements of the argument may yield a smaller bound.
\subsection{Organization}

The rest of the paper is organized as follows.
Section~\ref{sec:preliminaries} introduces the notation and
establishes the required third-moment estimates.
Section~\ref{sec:static-coefficients} proves geometric bounds
on the normalized Appell coefficient norms through joint
maximization over the measure and test function and the
construction of explicit concave barriers.
Section~\ref{sec:full-hodge} proves the curvature estimate
for compatible tensor fields using a weighted
Helmholtz--Hodge decomposition.
Section~\ref{full-integration-section} develops centered
integration, converts the curvature estimate into an operator
comparison, and relates repeated integration of constant
tensors to Appell polynomials. It then combines these estimates
on the Hilbert direct sum.
Section~\ref{sec:completion} proves the operator norm bound
and completes the proof of Theorem~\ref{thm:main} by approximation.
The appendices contain the compactness and approximation
arguments, together with the results on centered primitives,
operator domains, and extension to functions of finite
gradient energy.

\subsection*{AI Usage Statement}

We developed this proof with substantial multi-turn interactions from several
AI models. The project was inspired by the work of
\citet{ChenKlartag,BizeulMM,Letwin}.
Three AI interactions are worth emphasizing:\begin{itemize}
\item We proposed joint maximization over the measure and
the test function. AI helped develop the argument in
Proposition~\ref{prop:one-step-coefficient}, combining measure
variation with stationarity in the test function to control
the combined contribution of the two highest-order cumulant terms.
It also helped with constructing the concave barriers by inverting
quadratic and cubic polynomials, and with proving Theorem~\ref{thm:geometric-coefficients}.

\item Motivated by the Helmholtz--Hodge approach
in~\citet{balas0}, we suggested decomposing weighted divergence
into a compatible component and an orthogonal correction.
AI helped develop the curvature estimate in
Lemma~\ref{full-hodge}, with a coefficient independent
of tensor order.

\item Inspired by the coefficient strategy of~\citet{balas},
we suggested organizing repeated integration through recursive
coefficient estimates. AI proposed using Appell coefficient
norms and helped establish the representation
\eqref{full-Appell-primitives} for repeated centered integration
of constant tensors.

\end{itemize}
We have carefully verified all arguments and take full
responsibility for the content of this work.

\subsection*{Acknowledgment}
We sincerely thank Yuansi Chen for helpful comments on an earlier version of this manuscript.

\section{Preliminaries}\label{sec:preliminaries}

\subsection{Notations and Definitions}\label{sec:notation}

\subsubsection{Tensor Spaces and Hilbert Space Conventions}

In this work, we use two families of tensor spaces:
symmetric tensors and alternating tensors.
For $r\ge0$, let $\mathsf E_r=\Sym^r(\R^n)$ be the space
of symmetric tensors of order $r$, with $\mathsf E_0=\R$.
The tensor space $(\R^n)^{\otimes r}$ and its symmetric
subspace $\mathsf E_r$ carry the ordered Hilbert--Schmidt
inner product
\[
 \langle T,U\rangle_{\HS}
 =\sum_{i_1,\ldots,i_r=1}^n
       T_{i_1\cdots i_r}U_{i_1\cdots i_r},
 \qquad
 \|T\|_{\HS}^2=\langle T,T\rangle_{\HS}.
\]
For $0\le p\le n$, let $\Lambda^p\R^n$ denote the space
of alternating tensors\footnote{A tensor of
order $p$ has components $T_{i_1\cdots i_p}$. It is
\emph{alternating} if exchanging any two index positions
changes the sign of the corresponding component.} of order $p$, with $\Lambda^0\R^n=\R$.
At order two, $\mathsf E_2$ consists of symmetric
$n\times n$ matrices $T$ satisfying $T_{ij}=T_{ji}$,
and $\Lambda^2\R^n$ is identified with skew-symmetric
$n\times n$ matrices $\omega$ satisfying
$\omega_{ij}=-\omega_{ji}$.

The wedge product $\wedge$ is bilinear and associative,
with $e_i\wedge e_j=-e_j\wedge e_i$ and
$e_i\wedge e_i=0$ for the standard basis vectors
$e_1,\ldots,e_n$ of $\R^n$.
For each $i$, wedge multiplication by $e_i$ is the map
\begin{equation}\label{eq:wedgemult}
 \varepsilon_i:\Lambda^p\R^n\to\Lambda^{p+1}\R^n,\qquad
 \varepsilon_i\omega=e_i\wedge\omega.
\end{equation}
On a basis element $e_{j_1}\wedge\cdots\wedge e_{j_p}$
with $j_1<\cdots<j_p$, $\varepsilon_i$ inserts $e_i$
as the first factor. The result is zero if $i$ already
occurs among the indices. Otherwise, we reorder the factors
into increasing index order, changing the sign at each interchange. On $\Lambda^p\R^n$, we use the inner product for which the basis $e_{i_1}\wedge\cdots\wedge e_{i_p}$,
$i_1<\cdots<i_p$, is orthonormal. Thus the squared norm sums over increasing index tuples. Tensor products of these spaces carry the product inner product. We set $\mathsf E_r=\{0\}$ for $r<0$ and $\Lambda^p\R^n=\{0\}$ for $p\notin\{0,\ldots,n\}$.

The map $\Sym:(\R^n)^{\otimes r}\to\mathsf E_r$ is the
orthogonal projection obtained by averaging over permutations
of the tensor indices. Single bars denote scalar absolute value or Euclidean vector
length. For a measurable set $A\subset\R^n$, $|A|$ denotes
its Lebesgue volume. For $R>0$, we write
$B(x,R)=\{y\in\R^n:|y-x|<R\}$ and $B_R=B(0,R)$. For a constant tensor, an unsubscripted norm is the norm
of its stated tensor space, in particular $\|\cdot\|_{\HS}$
on $\mathsf E_r$.
We identify $\mathsf E_r$ with symmetric $r$-linear forms
using the Euclidean inner product.
We use the same identification for derivatives $D_z^r$
of scalar functions. For $r,d\ge0$, $T\in\mathsf E_{r+d}$, and $U\in\mathsf E_d$,
we define contraction by
\begin{equation}\label{eq:contraction}
 T\mathbin{\lrcorner}U\in\mathsf E_r,\qquad
 (T\mathbin{\lrcorner}U)_I
 =\sum_{|J|=d}T_{IJ}U_J,
\end{equation}
where $I$ and $J$ are ordered index tuples of lengths
$r$ and $d$, respectively. For example, given a symmetric matrix $T$ and a vector $U$, contraction
is the usual matrix--vector product
$T\mathbin{\lrcorner}U=TU$. Given a symmetric tensor $T$ of order three and a symmetric
matrix $U$, it gives the vector with components
$(T\mathbin{\lrcorner}U)_i=\sum_{j,k=1}^nT_{ijk}U_{jk}$.

For a finite-dimensional inner-product space $E$, such as
$\mathsf E_r$, $\Lambda^p\R^n$, or a tensor product of these
spaces, we define
\[
 \begin{split}
 L^2(\mu;E)&=\left\{U:\R^n\to E\text{ measurable}:
                   \int\|U(x)\|_E^2\,d\mu(x)<\infty\right\}/\mu\text{-a.e.},\\
 \langle U,W\rangle_{2,\mu}
   &=\int\langle U(x),W(x)\rangle_E\,d\mu(x),\qquad
 \|U\|_{2,\mu}^2=\int\|U(x)\|_E^2\,d\mu(x).
 \end{split}
\]
We abbreviate these to $\langle U,W\rangle_2$ and $\|U\|_2$ when
the law is fixed. We take inner products of fields in $L^2(\mu;E)$ outside
integrals and expectations, and pointwise in $E$ inside
integrals and expectations.
Expectations act componentwise:
$\E_\mu:L^2(\mu;E)\to E$.
We regard each $v\in E$ as the constant field $x\mapsto v$.
Since $\mu$ is a probability measure, this field has
$L^2(\mu;E)$ norm $\|v\|_E$. We write $L^2_0(\mu;E)
 =\{U\in L^2(\mu;E):\mathbb E_\mu U=0\}$ for the subspace of fields with mean zero.

For a closed subspace $\mathcal M$ of a Hilbert space $\mathcal V$,
$P_{\mathcal M}:\mathcal V\to\mathcal M$ is its orthogonal projection,
and $\mathcal M^\perp$ is taken in $\mathcal V$.
For a fixed linear map $A:E\to F$, we also write $A$ for
the operator from $L^2(\mu;E)$ to $L^2(\mu;F)$ defined by
$(AU)(x)=A(U(x))$.
These extensions preserve operator norms, adjoints, and
orthogonal projections.

For a bounded linear map $A:\mathcal U\to\mathcal V$ between Hilbert
spaces, our operator-norm and adjoint conventions are
\[
 \|A\|_{\mathcal U\to\mathcal V}
   =\sup_{u\ne0}\frac{\|Au\|_{\mathcal V}}{\|u\|_{\mathcal U}},
 \qquad
 A^*:\mathcal V\to\mathcal U,\qquad
 \langle Au,v\rangle_{\mathcal V}
   =\langle u,A^*v\rangle_{\mathcal U}.
\]
We omit the source and target subscripts when they have been specified.
For vectors in an abstract Hilbert space, an unsubscripted norm is
the norm of that space.
For matrices, $\|\cdot\|_{\mathrm{op}}$ denotes the Euclidean
operator norm, distinct from $\|\cdot\|_{\HS}$.
An unbounded operator is written
$A:\operatorname{Dom}(A)\subset\mathcal U\to\mathcal V$.
If its domain is dense, then
\[
 \operatorname{Dom}(A^*)=
 \left\{v\in\mathcal V:\ \exists w\in\mathcal U,\
   \langle Au,v\rangle_{\mathcal V}=\langle u,w\rangle_{\mathcal U}
   \text{ for every }u\in\operatorname{Dom}(A)\right\},
 \qquad A^*v=w.
\]
A nonnegative quadratic form $q$ on a Hilbert space
$\mathcal U$ is \emph{closed} if its domain is complete
for the norm $(\|u\|_{\mathcal U}^2+q(u))^{1/2}$. The order $B\preceq C$ is the order of quadratic forms on the same
Hilbert space. For bounded self-adjoint operators it means
$\langle u,Bu\rangle\le\langle u,Cu\rangle$ for every $u$.
For nonnegative unbounded operators it means
$\operatorname{Dom}(C^{1/2})\subset\operatorname{Dom}(B^{1/2})$
and $\|B^{1/2}u\|^2\le\|C^{1/2}u\|^2$ on
$\operatorname{Dom}(C^{1/2})$.
The operator estimates use real Hilbert spaces.

\subsubsection{Probability Laws, Tilts, and Appell Coefficients}
\label{sec:probability-laws}

For a probability law $\mu$ on $\R^n$ and scalar
$f\in L^2(\mu)$, we write $ \Var_\mu f=\|f-\E_\mu f\|_{2,\mu}^2.$ The \emph{affine support} of a law is the affine hull of its support. Gradients on this space are understood
intrinsically. The density and set definitions of
log-concavity agree there by Borell's characterization
\citep{Borell}. For a measurable map $T:\R^n\to\R^n$, the
\emph{pushforward} $T_\#\mu$ is the law of $T(X)$
when $X\sim\mu$.

\begin{definition}[Regular law]\label{def:regular-law}
A probability law $\mu$ with potential $V$ is \emph{regular}
if $V\in C^\infty(\R^n)$,
$aI\preceq D^2V(x)\preceq MI$ for every $x\in\R^n$
and some $0<a\le M<\infty$, and every derivative of $V$
has at most polynomial growth.
The constants $a,M$ may depend on the law.
\end{definition}

For a law with local exponential moments, let $\Omega_\mu
 =\operatorname{int}
   \left\{z\in\R^n:\int e^{z\cdot x}\,d\mu(x)<\infty\right\}.$ Then $0\in\Omega_\mu$. We define the moment generating
function and cumulant tensors by
\[
 M_\mu:\Omega_\mu\to(0,\infty),\qquad
 M_\mu(z)=\int e^{z\cdot x}\,d\mu(x),\qquad
 \kappa_j^\mu=D^j\log M_\mu(0)\in\mathsf E_j
 \quad(j\ge1).
\]
For $\theta$ near zero and $s\ge0$, we define the normalized
tilt
\begin{equation}\label{eq:barrier-static-tilt}
 d\mu_{\theta,s}(x)
 =\frac{e^{\theta\cdot x-s|x|^2/2}\,d\mu(x)}
        {\int e^{\theta\cdot y-s|y|^2/2}\,d\mu(y)}.
\end{equation}
If $\mu$ is log-concave, then so is $\mu_{\theta,s}$.
We use $\E_{\mu_{\theta,s}}$ and
$\|\cdot\|_{2,\mu_{\theta,s}}$ for expectations and norms
under the tilted law. Norms under other measures are
indicated explicitly.

For a neighborhood $\Theta\subset\Omega_\mu$ of zero
and $\varepsilon>0$, we define
\begin{equation}\label{eq:doperator}
 \mathcal D:C_\theta^2C_s^1
       (\Theta\times[0,\varepsilon))\to\R,\qquad
 \mathcal Dh
 =\left.(\partial_s+\tfrac12\Delta_\theta)h\right|_{(0,0)}.
\end{equation}
The indicated derivatives extend continuously to $s=0$,
where $\partial_s$ is a right derivative.
For maps into a finite-dimensional tensor space,
$\mathcal D$ acts componentwise.

The Appell tensors are defined near $z=0$ by the generating
function \citep[Section~2.3, Equation~(11)]{AnshelevichAppell}
\begin{equation}\label{full-Appell-definition}
 \frac{e^{z\cdot x}}{M_\mu(z)}
 =\sum_{d=0}^{\infty}
    \frac{\langle A_d^\mu(x),z^{\otimes d}\rangle_{\HS}}{d!}.
\end{equation}
Each $A_d^\mu:\R^n\to\mathsf E_d$ is a polynomial map
of degree $d$ belonging to $L^2(\mu;\mathsf E_d)$.
For $d\ge0$, we define the coefficient operator by
\[
 \begin{aligned}
 B_d^\mu&:L^2(\mu)\to\mathsf E_d,\qquad
 B_d^\mu(f)
 &=\left.D_z^d\left(
       \frac{\int f(x)e^{z\cdot x}\,d\mu(x)}{M_\mu(z)}
     \right)\right|_{z=0}
 =\int fA_d^\mu\,d\mu.
 \end{aligned}
\]
The integral representation follows by integrating
\eqref{full-Appell-definition} against $f\,d\mu$.
Local exponential moments and Cauchy--Schwarz justify
differentiation for every $f\in L^2(\mu)$.
We also define the map from symmetric tensors to scalar
polynomials by
\[
 P_d^\mu:\mathsf E_d\to L^2(\mu),\qquad
 P_d^\mu[T](x)=\langle T,A_d^\mu(x)\rangle_{\HS}.
\]

Taking expectation in \eqref{full-Appell-definition} gives
$\E_\mu A_d^\mu=0$ for $d\ge1$.
Thus $P_d^\mu$ takes values in $L^2_0(\mu)$ and
$B_d^\mu$ annihilates constants.
At degree zero,
$A_0^\mu=1$, $B_0^\mu(f)=\E_\mu f$, and $P_0^\mu[t]=t$.
The identity $\langle B_d^\mu(f),T\rangle_{\HS}
 =\langle f,P_d^\mu[T]\rangle_{2,\mu}$ shows that $(B_d^\mu)^*=P_d^\mu$.
For $d\ge1$, we therefore define the normalized coefficient
norms by
\begin{align}
 c_d(\mu)=\frac{\|B_d^\mu\|_{L^2(\mu)\to\mathsf E_d}}{d!}
  =\frac{\|P_d^\mu\|_{\mathsf E_d\to L^2(\mu)}}{d!}=\frac1{d!}
   \sup_{\substack{T\in\mathsf E_d\\\|T\|_{\HS}=1}}
       \|P_d^\mu[T]\|_{2,\mu}.
\end{align}
For fixed dimension $n$, we write
$c_d=\sup_{\mu\in\mathsf C_n}c_d(\mu)$.

\subsubsection{Compatible Fields and Differential Operators}\label{sec:compatible-fields}
In our proofs, we use weighted Sobolev spaces with respect to $\mu$ and
unweighted Sobolev spaces with respect to Lebesgue measure.

For a law $\mu$ with a smooth positive density on $\R^n$,
the weighted space $W^{k,2}(\mu;E)$ consists of the $E$-valued
fields whose distributional derivatives through order $k$
belong to $L^2(\mu;E)$, with norm $ \|U\|_{W^{k,2}(\mu;E)}^2
 =\sum_{j=0}^k\|\nabla^jU\|_{2,\mu}^2.$ In the norm of $\nabla^jU$, we sum over all ordered tuples of
derivative indices. The gradient acts as
\[
 \nabla:W^{1,2}(\mu;E)\subset L^2(\mu;E)
           \longrightarrow L^2(\mu;E\otimes\R^n).
\]
The derivatives are taken in Euclidean coordinates.
The measure $\mu$ determines integrability and the adjoints
of the differential operators. We formally use \emph{energy} to refer to the squared norms and quadratic
forms specified in each argument, and \emph{gradient energy} for $\|\nabla U\|_{2,\mu}^2
=\int\|\nabla U(x)\|^2\,d\mu(x)$.

The unweighted spaces $L^2(B_R)$, $H^s$, and $H^s_{\rm loc}$
use Lebesgue measure and are defined componentwise for tensor
fields. For $s\in\R$, we use $\|u\|_{H^s}^2
 =\int_{\R^n}(1+|\xi|^2)^s|\widehat u(\xi)|^2\,d\xi,$ where $\widehat u$ denotes the unitary Fourier transform.
A distribution $u$ belongs to $H^s_{\rm loc}(\R^n)$ if
$\zeta u\in H^s(\R^n)$ for every $\zeta\in C_c^\infty(\R^n)$.
For a nonnegative integer $s$, this means that $u$ and its
weak derivatives through order $s$ are square-integrable
on every compact set.

For a first-order differential expression $A$ acting from
$L^2(\mu;E)$ to $L^2(\mu;F)$, the
\emph{maximal distributional domain} refers to $ \operatorname{Dom}(A)
   =\{u\in L^2(\mu;E):Au\in L^2(\mu;F)
                       \text{ in distributions}\}.$ We use this domain unless a constrained space or smooth test domain
is specified. Thus $\partial_i$ and
$\partial_i^*=-\partial_i+V_i$, where $V_i=\partial_iV$ acts by
multiplication, are closed operators on $L^2(\mu;E)$ with these
domains.
For $d\mu=Z^{-1}e^{-V}\,dx$ with
$V\in C^\infty(\R^n)$, the weighted Laplacian is
\begin{align}\label{eq:weighted-laplacian}
\begin{aligned}
 \mathcal L_V:&\operatorname{Dom}(\mathcal L_V)
                   \subset L^2(\mu)\longrightarrow L^2(\mu),
 \qquad \mathcal L_Vf=-\Delta f+\nabla V\cdot\nabla f,\\
\text{where}~~ &\operatorname{Dom}(\mathcal L_V)
   =\{f\in W^{1,2}(\mu):
            \mathcal L_Vf\in L^2(\mu)
            \text{ in distributions}\}.
\end{aligned}
\end{align}
It is the nonnegative self-adjoint operator associated with the quadratic form
$q(f)=\|\nabla f\|_{2,\mu}^2$, defined on $W^{1,2}(\mu)$.
Its kernel consists of the constant functions.

For the law $d\mu=e^{-V}\,dx$ with
$V\in C^\infty(\R^n)$, we set
$\mathcal X_r=L^2(\mu;\mathsf E_r)$.
For $r\ge1$ and $U\in\mathcal X_r$, we define its divergence
and curl distributionally by
\[
 (\operatorname{div}U)_I=\sum_{j=1}^n\partial_jU_{Ij},
 \qquad
 (\operatorname{curl}_rU)_{ijI}
   =\partial_iU_{jI}-\partial_jU_{iI},
 \qquad |I|=r-1.
\]
The divergence is symmetric in $I$. The curl is symmetric in
$I$ and alternating in $i,j$, so it is an
$\mathsf E_{r-1}$-valued two-form with exterior indices $i<j$.
For a vector field $u$, the tuple $I$ is empty, giving
$\operatorname{div}u=\nabla\cdot u$ and
$(\operatorname{curl}_1u)_{ij}=\partial_i u_j-\partial_j u_i$.
In dimension three, the usual curl vector is $\bigl((\operatorname{curl}_1u)_{23},
       -(\operatorname{curl}_1u)_{13},
       (\operatorname{curl}_1u)_{12}\bigr).$

A field $U\in\mathcal X_r$ is \emph{compatible} if
$\operatorname{curl}_rU=0$.
We denote the compatible subspace by $\mathcal C_r(\mu)$ and set
$\mathcal G_r(\mu)=\mathcal C_r(\mu)\cap L^2_0(\mu;\mathsf E_r)$.
Since $\mu$ is fixed, we write $\mathcal C_r$ and $\mathcal G_r$. At order zero, we set $\mathcal C_0=L^2(\mu)$ and
$\mathcal G_0=L^2_0(\mu)$.
Convergence in $L^2(\mu)$ implies local distributional
convergence and preserves expectations, so these subspaces
are closed.
Constants are compatible, giving the orthogonal decomposition
\begin{equation}\label{full-compatible-split}
 \mathcal C_r=\mathsf E_r\oplus\mathcal G_r,\qquad
 P_{\mathcal G_r}U=P_{\mathcal C_r}U-\E_\mu U
 \quad(U\in\mathcal X_r).
\end{equation}

The weighted divergence is
\[
 \operatorname{div}_V:\operatorname{Dom}(\operatorname{div}_V)
       \subset\mathcal X_{r+1}\longrightarrow\mathcal X_r,
 \qquad r\ge0,
\]
with
\[
 (\operatorname{div}_VF)_I
   =\sum_{j=1}^n\partial_j^*F_{Ij}
   =-(\operatorname{div}F)_I+\sum_{j=1}^nV_jF_{Ij},
 \qquad |I|=r.
\]
Its domain consists of those $F\in\mathcal X_{r+1}$ for which
this summed distributional expression belongs to $\mathcal X_r$.
Equivalently,
$\operatorname{div}_VF=-e^V\operatorname{div}(e^{-V}F)$.
For a vector field $u$, this becomes
$\operatorname{div}_Vu=-\nabla\cdot u+\nabla V\cdot u$.
This sign convention makes $\operatorname{div}_V$ the formal
$L^2(\mu)$ adjoint of the gradient.
For smooth compactly supported fields
$U:\R^n\to\mathsf E_r$ and $F:\R^n\to\mathsf E_{r+1}$,
integration by parts gives $\langle\nabla U,F\rangle_2
   =\langle U,\operatorname{div}_VF\rangle_2.$

The compatible gradient is the restriction
\begin{equation}\label{eq:compatible-gradient}
 D_r:\operatorname{Dom}(D_r)
 =\mathcal G_r\cap W^{1,2}(\mu;\mathsf E_r)
 \subset\mathcal G_r\longrightarrow\mathcal C_{r+1},
 \qquad D_rU=\nabla U.
\end{equation}
Compatibility of $U$ makes $\nabla U$ fully symmetric and
compatible. Its mean need not vanish.
When $\mu$ is regular (Definition~\ref{def:regular-law}),
Lemma~\ref{full-primitive} proves that $D_r$ is closed and
densely defined.
Its Hilbert adjoint has type
\[
 D_r^*:\operatorname{Dom}(D_r^*)
       \subset\mathcal C_{r+1}\longrightarrow\mathcal G_r,
\]
with the adjoint-domain convention given above.
We also use the smooth potential test space
\begin{equation}\label{eq:kmdef}
 \mathcal K_s
   =\{\nabla^s\psi:\psi\in C_c^\infty(\R^n)\}
   \subset\mathcal C_s,\qquad s\ge1.
\end{equation}
These fields need not have zero mean.

\subsection{Control of Third Moments}\label{sec:elliptic-inputs}
We now record the preliminary third moment bound require for our analysis. Note that covariance changes under a linear tilt are governed by the third-moment
tensor. We leverage the moment-map estimate by
\citet[Theorem~1.5 and Lemmas~3.3--3.5]{ChenKlartag}, whose anisotropic
adaptation is recorded by
\citet[Proposition~2.4, Equation~(42)]{BizeulMM}.
We further use the following weighted formulation from \cite{Letwin}, together with the third-moment identity by \cite{ChenKlartag}, to obtain a two-index estimate in Lemma~\ref{lem:sharp-third-pair-slot} below.

\begin{lemma}[Weighted Moment-map Estimate
{\citep[Theorem~2.5]{Letwin}}]
\label{lem:weighted-moment-map-energy}
Let $X$ be isotropic and log-concave, with density $e^{-V}$ on a bounded
open convex set $K$, where $V$ is smooth and convex on an open neighborhood of $\overline K$. Let $\mu$ denote the law of $X$, and let
$\nu(dy)=e^{-\varphi(y)}\,dy$, with $\varphi$ smooth and
convex, be a \emph{moment-measure representation} of $\mu$,
meaning that $(\nabla\varphi)_\#\nu=\mu.$ Set $H=D^2\varphi$. For every symmetric matrix $T$,
\begin{equation}\label{eq:weighted-moment-map-input}
 \int\operatorname{Tr}(THTH)\,d\nu
 \le2\operatorname{Tr}(T^2).
\end{equation}
\end{lemma}

Existence of the representation and its smoothness under these
assumptions are supplied by the moment-measure theorem and its
regularity result, stated in \cite[Lemmas~2.1--2.2]{Letwin}.
In particular, $\nabla\varphi:\R^n\to K$ is a smooth
bijection with a smooth inverse.

The following estimate adapts the Bessel inequality argument of~\citet[Lemma~3.7 and the proof of Theorem~1.2]{ChenKlartag},
using the weighted moment-map bound in
Lemma~\ref{lem:weighted-moment-map-energy}. We include the proof to record
the minor modifications needed for a general symmetric matrix $T$.

\begin{lemma}
\label{lem:sharp-third-pair-slot}
Let $X$ be an isotropic log-concave random vector in $\mathbb R^n$, and set
\[
 (S_k)_{ij}=\mathbb E[X_kX_iX_j],\qquad k=1,\ldots,n.
\]
Then, for every real symmetric matrix $T$,
\begin{equation}\label{eq:sharp-third-pair-slot}
 \sum_{k=1}^n\operatorname{Tr}(TS_kTS_k)
 \le4\operatorname{Tr}(T^2).
\end{equation}
Equivalently, the linear map
\[
 \mathsf E_2\longrightarrow\mathsf E_2,\qquad T\longmapsto\sum_k S_kTS_k,
\]
is self-adjoint for $\langle T,U\rangle_{\HS}=\Tr(TU)$ and is
bounded above by $4I_{\mathsf E_2}$ in quadratic-form order.
\end{lemma}

\begin{proof}
First assume the hypotheses of
Lemma~\ref{lem:weighted-moment-map-energy}, and define the symmetric
matrix field $\tau:K\to\mathsf E_2$ by
\[
 \tau(\nabla\varphi(y))=D^2\varphi(y).
\]
The mean identity in \citet[Equation~(11)]{ChenKlartag} gives
$\mathbb E\tau(X)=I$. Furthermore, the third-moment identity from
\citet[Lemma~3.7]{ChenKlartag}, written in the variable
$X=\nabla\varphi(Y)$ with $Y\sim\nu$, is given by 
\begin{equation}\label{eq:third-tensor-steinkernel-coefficient}
 \mathbb E[X_k(\tau(X)-I)]=\tfrac12S_k.
\end{equation}
The cited weighted estimate hence becomes
$\mathbb E\operatorname{Tr}(T\tau T\tau)\le2\operatorname{Tr}(T^2)$
under this same change of variables.

For $T\succeq0$, we apply Bessel's inequality to each entry of
$T^{1/2}(\tau-I)T^{1/2}$.
By isotropy, $X_1,\ldots,X_n$ are orthonormal in $L^2(\mu)$,
where $\mu$ is the law of $X$.
Using \eqref{eq:third-tensor-steinkernel-coefficient}, we obtain
\begin{equation}\label{eq:third-pair-psd}
\begin{aligned}
 \frac14\sum_k\operatorname{Tr}(TS_kTS_k)\le\mathbb E\operatorname{Tr}\bigl(T(\tau-I)T(\tau-I)\bigr)=\mathbb E\operatorname{Tr}(T\tau T\tau)-\operatorname{Tr}(T^2)
 \le\operatorname{Tr}(T^2).
\end{aligned}
\end{equation}
The equality follows from $\mathbb E\tau=I$, and the final
inequality is \eqref{eq:weighted-moment-map-input}.

For a general symmetric matrix $T$, we set $|T|=(T^2)^{1/2}$
and choose an orthonormal eigenbasis of $T$, with eigenvalues
$t_1,\ldots,t_n$. In this basis,
\[
 \sum_k\operatorname{Tr}(TS_kTS_k)
 =\sum_{k,i,j}t_it_j(S_k)_{ij}^2
 \le\sum_{k,i,j}|t_i||t_j|(S_k)_{ij}^2
 =\sum_k\operatorname{Tr}(|T|S_k|T|S_k).
\]
The eigenvalues of $|T|$ are $|t_1|,\ldots,|t_n|$, so
$|T|\succeq0$. Applying \eqref{eq:third-pair-psd}
with $|T|$ in place of $T$ and multiplying by four gives $\sum_k\operatorname{Tr}(|T|S_k|T|S_k)
 \le4\operatorname{Tr}(|T|^2)
 =4\operatorname{Tr}(T^2)$, hence proving \eqref{eq:sharp-third-pair-slot}.

We now remove the bounded-support and smoothness assumptions of
Lemma~\ref{lem:weighted-moment-map-energy}. Let $\mu$ be an arbitrary isotropic log-concave law.
By~\citet[Lemma~A.1]{Letwin}, there exist isotropic log-concave
laws $\mu_j$ satisfying the bounded-support and smoothness
assumptions of Lemma~\ref{lem:weighted-moment-map-energy}.
Their polynomial moments through degree three converge to
those of $\mu$, so their third-moment matrices $S_k^{(j)}$
converge entrywise to $S_k$. For each fixed symmetric matrix $T$, the left-hand side of \eqref{eq:sharp-third-pair-slot} is polynomial in these entries, while the right-hand side is independent of the law.
Passing to the limit therefore proves
\eqref{eq:sharp-third-pair-slot} for $\mu$.

\end{proof}

\begin{remark}
The estimate $\kappa_n^2\le8$ by
\citet[Proof of Theorem~1.1]{Letwin}, gives the one-index bound $\sum_kS_k^2\preceq8I$.
The constant $4$ in
\eqref{eq:sharp-third-pair-slot} controls the interaction of two indices. We emphasize that keeping these estimates separate improves the quadratic dependence on tensor order later.
\end{remark}

\paragraph{Initial Coefficient Bounds.} For an isotropic law, $A_1(x)=x$ and $A_2(x)=x^{\otimes2}-I$.
Bessel's inequality therefore gives
\begin{equation}\label{eq:barrier-linear-seed}
 \|B_1(f)\|^2\le\|f\|_2^2.
\end{equation}
Equality holds for nonzero linear functions, proving $c_1=1$.
Letwin's quadratic variance inequality
\cite[Theorem~1.2, Equation~(1.2)]{Letwin} gives, for every symmetric
matrix $U$,
\[
 \operatorname{Var}(X^TU X)\le8\|U\|_{\mathrm{HS}}^2.
\]
Consequently
\[
 |\langle B_2(f),U\rangle|
 \le\|f\|_2\sqrt{\operatorname{Var}(X^TU X)}
 \le\sqrt8\,\|f\|_2\|U\|_{\mathrm{HS}},
\]
and hence
\begin{equation}\label{eq:barrier-quadratic-seed}
 \|B_2(f)\|^2\le8\|f\|_2^2.
\end{equation}

\section{Geometric Bounds for Normalized Appell Coefficient Norms}\label{sec:static-coefficients}

In this section, we prove Theorem~\ref{thm:geometric-coefficients}, which gives geometric bounds on the normalized Appell coefficient norms. We first derive the tilt identities and then use the third-moment estimates to obtain differential inequalities for the squared covariance-normalized coefficient norms. We then introduce a quotient to be maximized jointly over the measure and test function and establish the existence of a maximizing pair. The measure variation produces a cumulant term beyond the induction hypothesis. At the same maximizing pair, stationarity in the test function bounds its weighted sum with the corresponding term from the preceding degree. We estimate the other cumulant terms using the coefficient bounds already proved. This reduces the next induction step to verifying an explicit scalar inequality for the concave denominator. Finally, we verify this criterion using concave barriers obtained by inverting suitable decreasing branches of quadratic and cubic polynomials, completing the induction from the linear and quadratic initial bounds.

We emphasize that for fixed $n$, Lemma~\ref{lem:static-compactness} gives weak compactness of $\mathsf C_n$, continuity of polynomial moments under weak convergence, and a local exponential-moment bound uniform over $\mathsf C_n$. The auxiliary constants may depend on $n$, but they do not enter the coefficient estimates below.

\subsection{Differential Identities under Tilting}

Fix an isotropic law $\mu$ and a polynomial $f$. Recall also the normalized tilt $\mu_{\theta,s}$ from \eqref{eq:barrier-static-tilt}, and let $m=m(\theta,s)\in\R^n$ and $\Sigma=\Sigma(\theta,s):\R^n\to\R^n$ denote its mean and positive definite covariance, respectively. In this subsection,
$B_j=B_j^{\mu_{\theta,s}}(f)\in\mathsf E_j$ and
$\kappa_j=\kappa_j^{\mu_{\theta,s}}\in\mathsf E_j$ are computed under the
tilted law. Define
\begin{align}\label{eq:ejdef}
 g=\|f\|_{2,\mu_{\theta,s}}^2,\qquad
 E_j=\langle B_j,(\Sigma^{-1})^{\otimes j}B_j\rangle_{\HS}
     =\|(\Sigma^{-1/2})^{\otimes j}B_j\|_{\HS}^2.
\end{align}
The map $(\Sigma^{-1/2})^{\otimes j}:\mathsf E_j\to\mathsf E_j$
normalizes the coefficient tensor, and $E_j$ is the resulting
scalar energy.

All quantities without
parameters in the following calculations are evaluated at the
isotropic base point. Recall the parameter functional $\mathcal D$ defined in \eqref{eq:doperator}. For scalar functions $a$ and $b$ of $(\theta,s)$, the product rule gives
\begin{equation}\label{eq:barrier-product-rule}
 \mathcal D(ab)=a\mathcal Db+b\mathcal Da
               +\langle\nabla_\theta a,\nabla_\theta b\rangle.
\end{equation}

For a fixed polynomial $F$, put $v(\theta,s)=\mathbb E_{\theta,s}F$
and $Y=X-m$. Differentiating the expression for $v(\theta,s)$ obtained from \eqref{eq:barrier-static-tilt} yields
\[
 \nabla_\theta v=\mathbb E[(F-v)Y],\qquad
 D_\theta^2v=\mathbb E[(F-v)Y^{\otimes2}],\qquad
 \partial_s v=-\tfrac12\operatorname{Cov}(F,|X|^2).
\]
Taking the trace of the second identity and expanding
$|X|^2=|Y|^2+2m\cdot Y+|m|^2$ gives
\begin{equation}\label{eq:barrier-expectation-identity}
 (\partial_s+\tfrac12\Delta_\theta)v
 =-m\cdot\nabla_\theta v.
\end{equation}
In particular, $\mathcal Dg=0$.

Identity~\eqref{eq:barrier-expectation-identity} is the
deterministic form of the martingale relation for tilted
expectations used in \citet[Section~3]{KlartagLogarithmic}. There, this relation
compares variances along a stochastic tilt, and covariance
estimates along the process lead to a spectral-gap bound.
Related deterministic heat-flow identities appear in
\citet[Sections~2 and~4.1]{KlartagPutterman}.
The additional ingredient in our work is the joint maximization of a
covariance-normalized coefficient quotient over the measure
and the test function. At the centered base point,
$\mathcal Dg=0$ simplifies the differentiation of its
denominator. Maximality with respect to the measure yields
\eqref{eq:barrier-maximum-inequality}, while stationarity
in the test function yields
\eqref{eq:barrier-terminal-cancellation}, which controls
the weighted sum of the two highest-order cumulant terms. These two
optimality relations reduce the coefficient induction
to a scalar barrier condition.

An additional linear tilt by $z$ replaces $\theta$ by $\theta+z$.
The definition of the Appell tensors in \eqref{full-Appell-definition} therefore gives
$B_j(f)=D_\theta^j\mathbb E_{\theta,s}f$. Differentiating \eqref{eq:barrier-expectation-identity} $j$ times with respect to $\theta$ and evaluating at the base point $(\theta,s)=(0,0)$ gives
\begin{equation}\label{eq:barrier-coefficient-identities}
 \partial_{\theta_i}B_j=B_{j+1}(\cdot,e_i),
 \qquad
 \mathcal DB_j=-jB_j-R_j.
\end{equation}
Here $R_1=0$, and for $j\ge2$,
\begin{equation}\label{eq:barrier-remainder-definition}
 R_j(v_1,\ldots,v_j)
 =\sum_{a=2}^j\ 
   \sum_{\substack{J\subset\{1,\ldots,j\}\\|J|=a}}
 \left\langle\kappa_{a+1}(v_J,\cdot),
       B_{j-a+1}(v_{J^c},\cdot)\right\rangle.
\end{equation}
Indeed, the terms with no derivative on $m$ vanish because $m=0$;
the $j$ terms with one derivative on $m$ each contribute $-B_j$
because $D_\theta m=\Sigma=I$; the other terms are the displayed
subset sum because $D_\theta^a m=\kappa_{a+1}$.

Writing $S_i=\kappa_3(\cdot,\cdot,e_i)$, the covariance identities are given by
\begin{equation}\label{eq:barrier-covariance-identities}
 \mathcal D\Sigma=-I,
 \qquad \partial_{\theta_i}\Sigma=S_i.
\end{equation}
For the first identity, write
$\Sigma=\mathbb E[XX^T]-mm^T$. At the isotropic base point,
\eqref{eq:barrier-expectation-identity} annihilates the expectations
of each fixed coordinate and each fixed quadratic monomial, while
the product rule contributes
$\sum_i(\partial_i m)(\partial_i m)^T=I$ to
$\mathcal D(mm^T)$. The second identity follows by differentiating
the covariance in a linear-tilt direction. The local exponential moments
from Lemma~\ref{lem:static-compactness} justify every differentiation,
including the polynomial factors
introduced by right differentiation in $s$.

\subsection{Covariance Energy with the Two-slot Bound}

For an isotropic law, applying the quadratic coefficient bound \eqref{eq:barrier-quadratic-seed} to $f(x)=b\cdot x$ gives
\[
 \|\kappa_3(\cdot,\cdot,b)\|_{\mathrm{HS}}^2\le8|b|^2.
\]
By the symmetry of the cumulant tensor, this is equivalent to
\begin{equation}\label{eq:barrier-single-slot}
 \sum_i S_i^2\preceq8I.
\end{equation}
By Lemma~\ref{lem:sharp-third-pair-slot}, the corresponding two-slot estimate is
\begin{equation}\label{eq:barrier-two-slot}
 \sum_i\operatorname{Tr}(U S_i U S_i)\le4\|U\|_{\mathrm{HS}}^2
 \quad\text{for every symmetric matrix }U.
\end{equation}
At order $j$, put
\[
 T=B_j,\quad W_i=B_{j+1}(\cdot,e_i),\quad
 V_i=\sum_{a=1}^j S_i^{(a)}T,\quad
 \mathcal S=\sum_{i,a}\|S_i^{(a)}T\|^2,
\]
where the superscript denotes the tensor slot on which a matrix
acts. More explicitly,
$S_i^{(a)}:(\R^n)^{\otimes j}\to(\R^n)^{\otimes j}$
applies $S_i:\R^n\to\R^n$ in slot $a$ and the identity in every
other slot. An individual $S_i^{(a)}T$ need not be symmetric. However, the sum $\sum_aS_i^{(a)}$ preserves $\mathsf E_j$, so $T,W_i,V_i\in\mathsf E_j$.
Applying the product rule in \eqref{eq:barrier-product-rule} to $E_j$ from ~\eqref{eq:ejdef} and using \eqref{eq:barrier-coefficient-identities} and \eqref{eq:barrier-covariance-identities} at $(\theta,s)=(0,0)$ gives
\begin{align}
 \mathcal DE_j={}&-jE_j-2\langle T,R_j\rangle
 +\sum_i\bigl(\|W_i\|^2-2\langle W_i,V_i\rangle
                         +\tfrac12\|V_i\|^2\bigr)
 +\tfrac12\mathcal S.
 \label{eq:barrier-exact-energy}
\end{align}
To verify \eqref{eq:barrier-exact-energy}, write
$Q(\Sigma)=(\Sigma^{-1})^{\otimes j}$, so that
$E_j=\langle T,Q(\Sigma)T\rangle_{\HS}$.
All derivatives below are evaluated at $(\theta,s)=(0,0)$, where $\Sigma=I$. For a symmetric matrix $S$, the auxiliary perturbation $\Sigma_\varepsilon=I+\varepsilon S$ gives
\begin{align*}
 Q(I+\varepsilon S)
 ={}I-\varepsilon\sum_aS^{(a)}+\frac{\varepsilon^2}{2}
 \left[\left(\sum_aS^{(a)}\right)^2
       +\sum_a(S^{(a)})^2\right]
 +O(\varepsilon^3).
\end{align*}
Let $DQ(I)$ and $D^2Q(I)$ denote the first and second
derivatives of $Q$ with respect to its matrix argument.
The second-order chain rule gives
\begin{align*}
 \mathcal D[Q(\Sigma)]
 =DQ(I)[\mathcal D\Sigma]
  +\frac12\sum_{i=1}^n
    D^2Q(I)
    [\partial_{\theta_i}\Sigma,\partial_{\theta_i}\Sigma].
\end{align*}
Substituting the covariance identities
\eqref{eq:barrier-covariance-identities} into this formula
and using the expansion of $Q(I+\varepsilon S)$ above gives
\begin{align*}
 \bigl(\partial_{\theta_i}Q(\Sigma)\bigr)T
 =-V_i,\quad \text{and}\quad
 \left\langle T,\mathcal D[Q(\Sigma)]T\right\rangle_{\HS}
 =jE_j+\frac12\sum_i\|V_i\|^2+\frac12\mathcal S.
\end{align*}
Note that the term $jE_j$ comes from the linear variation of the covariance factor, and the two squared-norm terms come from its quadratic variation.

Applying \eqref{eq:barrier-product-rule} twice to
$E_j=\langle T,Q(\Sigma)T\rangle_{\HS}$ and using
$Q(I)=I$ gives, at $(\theta,s)=(0,0)$,
\begin{align*}
 \mathcal DE_j
 ={}2\langle\mathcal DT,T\rangle_{\HS}
     +\langle T,\mathcal D[Q(\Sigma)]T\rangle_{\HS}+\sum_i\|\partial_{\theta_i}T\|_{\HS}^2
   +2\sum_i
     \left\langle
       \partial_{\theta_i}T,
       \bigl(\partial_{\theta_i}Q(\Sigma)\bigr)T
     \right\rangle_{\HS}.
\end{align*}
We now substitute $\mathcal DT=-jT-R_j$ and
$\partial_{\theta_i}T=W_i$ from
\eqref{eq:barrier-coefficient-identities}, together with
the displayed formulas for
$\bigl(\partial_{\theta_i}Q(\Sigma)\bigr)T$ and
$\langle T,\mathcal D[Q(\Sigma)]T\rangle_{\HS}$,
into the product rule above and use
$-2jE_j+jE_j=-jE_j$ to obtain
\eqref{eq:barrier-exact-energy}.

The one-slot estimate from \eqref{eq:barrier-single-slot} gives $\mathcal S\le8jE_j$. For two distinct
slots $a,b$, fix the remaining tensor indices. Since $T$ is fully
symmetric, its two free indices form a symmetric matrix, to which
\eqref{eq:barrier-two-slot} applies. Summing over the fixed indices gives
\begin{align*}
 \sum_i\langle S_i^{(a)}T,S_i^{(b)}T\rangle\le4E_j,
 \qquad a\ne b.
\end{align*}
Since $V_i=\sum_{a=1}^jS_i^{(a)}T$, expanding the squared norm yields
\begin{align}
 \sum_i\|V_i\|^2
 &=\sum_{i,a}\|S_i^{(a)}T\|^2
   +\sum_{a\ne b}\sum_i
      \langle S_i^{(a)}T,S_i^{(b)}T\rangle \notag\\
 &\le\mathcal S+4j(j-1)E_j.
 \label{eq:barrier-slot-sum}
\end{align}
Here the first sum equals $\mathcal S$, and the preceding estimate
bounds each of the $j(j-1)$ ordered pairs $a\ne b$ by $4E_j$
after summation over $i$. 

For $0\le\beta<1$, put
$\eta_\beta=(1-\beta)^{-1}-1/2$. Completing the square gives
\[
 \|W\|^2-2\langle W,V\rangle+\tfrac12\|V\|^2
 =\beta\|W\|^2-\eta_\beta\|V\|^2
  +(1-\beta)\left\|W-\frac V{1-\beta}\right\|^2.
\]
Since $\eta_\beta\ge1/2$, the two estimates for the slot sums imply
\begin{align*}
 -\eta_\beta\sum_i\|V_i\|^2+\tfrac12\mathcal S
 \ge-(\eta_\beta-\tfrac12)\mathcal S
                    -4\eta_\beta j(j-1)E_j
 \ge-\bigl[4\eta_\beta j^2+(4\eta_\beta-4)j\bigr]E_j.
\end{align*}
Because $\sum_i\|W_i\|^2=E_{j+1}$ at the base point, we obtain
\begin{equation}\label{eq:barrier-adjustable-energy}
 \mathcal DE_j\ge\beta E_{j+1}
   -\bigl[A_\beta j^2+(A_\beta-3)j\bigr]E_j
   -2\langle B_j,R_j\rangle,
 \qquad A_\beta=\frac4{1-\beta}-2.
\end{equation}
In particular,
\begin{equation}\label{eq:barrier-upper-energy}
 \mathcal DE_j\ge-(2j^2-j)E_j-2\langle B_j,R_j\rangle.
\end{equation}
We will use \eqref{eq:barrier-adjustable-energy} with $\beta=\alpha$
at degree $k$, and \eqref{eq:barrier-upper-energy} at degree $k+1$. We retain the inner products with $R_j$ to bound the combined
contribution of the two terminal cumulant terms using stationarity
with respect to the test function.

\subsection{Extremal Quotient and the Test Function Variation}

To set up the induction step, fix $k\ge2$ and suppose that positive
constants $h_1=1,h_2,\ldots,h_k$ satisfy
\begin{equation}\label{eq:barrier-induction-hypothesis}
 \|B_j^\mu(f)\|^2\le h_j\|f\|_{2,\mu}^2
 \qquad(1\le j\le k)
\end{equation}
for every isotropic log-concave law $\mu$ and every
$f\in L^2(\mu)$. The initial bounds
\eqref{eq:barrier-linear-seed} and
\eqref{eq:barrier-quadratic-seed} justify the choices
$h_1=1$ and $h_2=8$.
Next, we choose $\phi$ positive, decreasing, concave, and $C^2$ on a
neighborhood of $[0,1]$, with $\phi(0)=1$.
The constants $h_j$ and the function $\phi$ remain fixed throughout
the maximization. Using the definitions of $g$ and $E_j$ in \eqref{eq:ejdef}, we consider the homogeneous quotient 
\begin{align}\label{eq:barrier-extremal-quotient}
 \mathcal F(\mu,f)
 =\frac{E_{k+1}}{g\phi(E_k/(h_kg))}
\end{align}
over $\mu\in\mathsf C_n$ and nonzero $f\in L^2(\mu)$.
For an isotropic law, $\Sigma=I$ and
$E_j=\|B_j^\mu(f)\|^2$. Thus
\eqref{eq:barrier-induction-hypothesis} gives
$0\le E_k/(h_kg)\le1$, so the denominator in
\eqref{eq:barrier-extremal-quotient} is positive.

We first show that, when taking the supremum of the quotient \eqref{eq:barrier-extremal-quotient}, it suffices to consider nonzero polynomial test functions of degree at most $k+1$. Let $\mathcal P_{\le k+1}$ denote the scalar polynomials of degree
at most $k+1$, regarded as a subspace of $L^2(\mu)$, and let
\[
 \Pi_{\le k+1}^\mu:L^2(\mu)\to\mathcal P_{\le k+1}
\]
be the orthogonal projection. Since every component of $A_j^\mu$
is a polynomial of degree at most $j$, orthogonality gives
\begin{equation}\label{eq:barrier-polynomial-projection}
 B_j^\mu\Pi_{\le k+1}^\mu=B_j^\mu
 \qquad(1\le j\le k+1).
\end{equation}
Moreover, $\|\Pi_{\le k+1}^\mu f\|_{2,\mu}\le\|f\|_{2,\mu}$.
Therefore projection preserves $E_k$ and $E_{k+1}$ while
decreasing $g$.

For fixed $y=E_k$, differentiation of the denominator with
respect to $g$ gives
\begin{equation}\label{eq:barrier-denominator-monotonicity}
 \frac{d}{dg}\left[g\phi\!\left(\frac{y}{h_kg}\right)\right]
 =\phi(u)-u\phi'(u)>0,
 \qquad u=\frac{y}{h_kg}.
\end{equation}
Both the original and projected test functions satisfy
\eqref{eq:barrier-induction-hypothesis}, so their scaled energies
lie in $[0,1]$. Applying
\eqref{eq:barrier-denominator-monotonicity} to
\eqref{eq:barrier-extremal-quotient} therefore shows that projection
cannot decrease $\mathcal F$.
If the projection vanishes, \eqref{eq:barrier-polynomial-projection}
gives $E_{k+1}=0$, so that test function does not affect the
supremum. Since the quotient \eqref{eq:barrier-extremal-quotient} is invariant under rescaling $f$, we may further restrict to $f\in\mathcal P_{\le k+1}$ with $g=\|f\|_{2,\mu}^2=1$.

We now show that the maximum is attained after this reduction.
Let $N=\binom{n+k+1}{k+1}$ and let
$p:\R^n\to\R^N$ list all monomials of degree at most $k+1$.
The moment Gram matrix
\[
 \mathsf G_{k+1}(\mu)=\int p(x)p(x)^T\,d\mu(x)
\]
is positive definite. Indeed, an isotropic log-concave law has a
density positive on a nonempty open set, and a nonzero polynomial
cannot vanish there almost everywhere.
By Lemma~\ref{lem:static-compactness}, the entries of
$\mathsf G_{k+1}(\mu)$ vary continuously on the weakly compact
class $\mathsf C_n$. Its least eigenvalue therefore has a
positive lower bound on $\mathsf C_n$.

Write $f(x)=c^Tp(x)$, so that $\|f\|_{2,\mu}^2=c^T\mathsf G_{k+1}(\mu)c.$ The uniform positive lower bound\footnote{This lower bound may depend on the fixed dimension and degree. It is used only to establish the existence of a maximizing pair and does not enter the coefficient estimates.} on the least eigenvalue of $\mathsf G_{k+1}(\mu)$ therefore bounds $|c|$ whenever $\|f\|_{2,\mu}=1$. By Lemma~\ref{lem:static-compactness}, the moments of $\mu$ through order $2(k+1)$ are continuous under weak convergence in $\mathsf C_n$. Thus the entries of $\mathsf G_{k+1}(\mu)$ vary continuously, and the condition $c^T\mathsf G_{k+1}(\mu)c=1$ is preserved under limits. Together with compactness of $\mathsf C_n$, this proves compactness of the set of pairs $(\mu,c)$ satisfying this normalization.

On this set, the energies in
\eqref{eq:barrier-extremal-quotient} are continuous, and the denominator of the quotient \eqref{eq:barrier-extremal-quotient} satisfies $\phi(E_k/h_k)\ge\min_{u\in[0,1]}\phi(u)>0$. Thus the maximum is attained at a pair $(\mu_*,f_*)$, with
\[
 f_*\in\mathcal P_{\le k+1},\qquad
 \|f_*\|_{2,\mu_*}^2=1.
\] 

We first vary the measure while keeping $f_*$ fixed in the original
coordinates. Next, we apply \eqref{eq:barrier-static-tilt} with base law
$\mu=\mu_*$, so that $\mu_{0,0}=\mu_*$.
If $V_*$ is a convex potential for $\mu_*$, the tilted potential,
up to an additive normalizing constant, is
\[
 V_{\theta,s}(x)
 =V_*(x)-\theta\cdot x+\frac{s}{2}|x|^2.
\]
The potential $V_*$ may take the value $+\infty$ outside its convex
support. The linear term is affine and the quadratic term is convex
for $s\ge0$, so these perturbations preserve log-concavity.

To compare $\mathcal F(\mu_{\theta,s},f_*)$ with the maximum
$\mathcal F(\mu_*,f_*)$, we first center and whiten
$\mu_{\theta,s}$ and compose $f_*$ with the inverse affine map. We now verify that this affine normalization preserves the value of the quotient \eqref{eq:barrier-extremal-quotient}, that is, $\mathcal F(T_\#\mu,f\circ T^{-1})=\mathcal F(\mu,f).$  For an invertible affine map $T(x)=Ax+b$, set $f_T=f\circ T^{-1}$. A change of variables gives
\[
 M_{T_\#\mu}(z)=e^{z\cdot b}M_\mu(A^Tz),
 \qquad
 \frac{e^{z\cdot T(x)}}{M_{T_\#\mu}(z)}
 =\frac{e^{(A^Tz)\cdot x}}{M_\mu(A^Tz)}
\]
for $z$ near zero. Comparing coefficients in
\eqref{full-Appell-definition} and integrating against
$f\,d\mu$ gives
\begin{equation}\label{eq:barrier-affine-transformation}
 \begin{aligned}
 B_j^{T_\#\mu}(f_T)&=A^{\otimes j}B_j^\mu(f),\\
 \Cov(T_\#\mu)&=A\Sigma A^T,\qquad
 \|f_T\|_{2,T_\#\mu}^2=\|f\|_{2,\mu}^2.
 \end{aligned}
\end{equation}
Here, the covariance and norm identities follow directly
by changing variables.

Substituting \eqref{eq:barrier-affine-transformation} into
\eqref{eq:ejdef} and using $A^T(A\Sigma A^T)^{-1}A=\Sigma^{-1}$ shows that every $E_j$ and $g$ is preserved.
Equation~\eqref{eq:barrier-extremal-quotient} then shows that
$\mathcal F$ is preserved as well. In other words, centering and whitening therefore produce an admissible isotropic
log-concave law with the same quotient value.

The transformed test
function remains a polynomial of degree at most $k+1$.
Consequently, writing
$\mathcal F((\theta,s),f_*)=\mathcal F(\mu_{\theta,s},f_*)$
for the covariance-normalized quotient, maximality gives
\begin{equation}\label{eq:barrier-tilt-maximality}
 \mathcal F((\theta,s),f_*)\le\mathcal F((0,0),f_*)
 =\mathcal F(\mu_*,f_*).
\end{equation}
for $\theta$ near zero and sufficiently small $s\ge0$.
The local exponential-moment bound in
Lemma~\ref{lem:static-compactness} justifies the parameter
derivatives used below.

Applying the necessary conditions for a maximum to
\eqref{eq:barrier-tilt-maximality} gives
\begin{equation}\label{eq:barrier-maximum-conditions}
 \nabla_\theta\mathcal F((0,0),f_*)=0,\qquad
 D_\theta^2\mathcal F((0,0),f_*)\preceq0,\qquad
 \partial_s^+\mathcal F((0,0),f_*)\le0.
\end{equation}
The derivative in $s$ is one-sided because $s\ge0$.
Combining \eqref{eq:barrier-maximum-conditions} with the definition
of $\mathcal D$ in \eqref{eq:doperator} yields
\[
 \mathcal D\bigl[(\theta,s)\mapsto
             \mathcal F((\theta,s),f_*)\bigr]\le0.
\]

In the following computation, all derivatives in $\theta$ and $s$ are evaluated at
$(\theta,s)=(0,0)$, with $f_*$ held fixed. To differentiate the numerator and denominator of the quotient
\eqref{eq:barrier-extremal-quotient}, introduce the notation, 
\begin{equation}\label{eq:barrier-scalar-notation}
 \begin{aligned}
 x&=E_{k+1}, & y&=E_k, & h&=h_k,\\
 u&=\frac{y}{hg}, & d&=g\phi(u), &
 q(u)&=-\frac{\phi'(u)}{\phi(u)}.
 \end{aligned}
\end{equation}
These quantities vary along the perturbations. At the maximizing
pair, $g=1$, so $u=y/h$ and $d=\phi(u)$. Applying \eqref{eq:barrier-expectation-identity} to the fixed polynomial $f_*^2$ gives $\mathcal Dg=0$.

Differentiating $y=hgu$ and using
\eqref{eq:barrier-product-rule} with $\mathcal Dg=0$,
we obtain
\[
 \nabla_\theta u
 =\frac{\nabla_\theta y-hu\nabla_\theta g}{hg},
 \qquad
 g\mathcal Du
 +\langle\nabla_\theta g,\nabla_\theta u\rangle
 =\frac{\mathcal Dy}{h}.
\]
Applying \eqref{eq:barrier-product-rule} to
$d=g\phi(u)$ and the chain rule to $\phi(u)$ therefore gives
\begin{align}\label{eq:barrier-perspective-chain-rule}
 \mathcal Dd
 &=\phi'(u)\bigl(
      g\mathcal Du
      +\langle\nabla_\theta g,\nabla_\theta u\rangle
    \bigr)
   +\frac{g\phi''(u)}2|\nabla_\theta u|^2
   \notag\\
 &=\frac{\phi'(u)}h\mathcal Dy
   +\frac{\phi''(u)}{2h^2g}
      |\nabla_\theta y-hu\nabla_\theta g|^2
   \notag\\
 &\le\frac{\phi'(u)}h\mathcal Dy,
\end{align}
where the inequality follows from $\phi''\le0$.

Next, applying the product rule \eqref{eq:barrier-product-rule} to $x=\mathcal F((\theta,s),f_*)\,d$ and using
$\nabla_\theta\mathcal F((0,0),f_*)=0$ from
\eqref{eq:barrier-maximum-conditions} gives
\[
 d\,\mathcal D\bigl[(\theta,s)\mapsto
             \mathcal F((\theta,s),f_*)\bigr]
 =\mathcal Dx-\mathcal F((0,0),f_*)\,\mathcal Dd.
\]
Since $d>0$ and
$\mathcal F((0,0),f_*)=x/\phi(u)$, combining
\eqref{eq:barrier-maximum-conditions} and
\eqref{eq:barrier-perspective-chain-rule} with this identity yields
\begin{align}\label{eq:barrier-maximum-inequality}
 0\ge d\,\mathcal D\bigl[(\theta,s)\mapsto
                  \mathcal F((\theta,s),f_*)\bigr]=\mathcal Dx-\mathcal F((0,0),f_*)\,\mathcal Dd\ge\mathcal Dx+\frac{xq(u)}h\mathcal Dy.
\end{align}

We next keep $\mu_*$ fixed and vary the test function at the same
maximizing pair. For a fixed vector $b\in\R^n$, let
$\ell_b(z)=b\cdot z$.
The Appell generating function
\eqref{full-Appell-definition} gives
\begin{equation}\label{eq:barrier-linear-cumulant-identity}
 B_a^\mu(\ell_b)
 =\left.D_\theta^a
       \bigl[b\cdot\nabla_\theta\log M_\mu(\theta)\bigr]
   \right|_{\theta=0}
 =\kappa_{a+1}^\mu(\cdot,b).
\end{equation}
At $(\mu_*,f_*)$, abbreviate $B_j=B_j^{\mu_*}(f_*),$ and  $\kappa_j=\kappa_j^{\mu_*}.$ Set $b=B_1$ and $\ell=\ell_b$.
Since $\mu_*$ is isotropic, we have
\[
 \langle f_*,\ell\rangle_{2,\mu_*}
 =b\cdot\E_{\mu_*}[f_*X]
 =|B_1|^2.
\]

We keep $b$ and $\ell$ fixed and set $f_t=f_*+t\ell$.
For sufficiently small $t$, this is a nonzero polynomial of degree
at most $k+1$, so it is admissible in the maximization.
Maximality hence gives
\begin{equation}\label{eq:barrier-test-function-stationarity}
 \left.\frac{d}{dt}\mathcal F(\mu_*,f_t)\right|_{t=0}=0.
\end{equation}
In the following calculation, a dot denotes differentiation
with respect to $t$ at $t=0$.

The covariance remains $I$ throughout this variation.
Using the linearity of $B_a^{\mu_*}$ and applying
\eqref{eq:barrier-linear-cumulant-identity} with $b=B_1$ gives
\begin{equation}\label{eq:barrier-test-function-derivatives}
 \begin{aligned}
 \dot x
 &=2\langle B_{k+1},
             \kappa_{k+2}(\cdot,B_1)\rangle,\\
 \dot y
 &=2\langle B_k,
             \kappa_{k+1}(\cdot,B_1)\rangle,\\
 \dot g
 &=2\langle f_*,\ell\rangle_{2,\mu_*}
 =2|B_1|^2.
 \end{aligned}
\end{equation}
Differentiating $d=g\phi(y/(hg))$ from
\eqref{eq:barrier-scalar-notation} gives
\begin{equation}\label{eq:barrier-test-denominator-derivative}
 \dot d
 =(\phi(u)-u\phi'(u))\dot g+\frac{\phi'(u)}h\dot y.
\end{equation}
Applying the quotient rule to
\eqref{eq:barrier-extremal-quotient}, using
\eqref{eq:barrier-test-function-stationarity}, and substituting
\eqref{eq:barrier-test-denominator-derivative} yields
\begin{equation}\label{eq:barrier-test-quotient-derivative}
 \dot x+\frac{xq(u)}h\dot y
 =x(1+uq(u))\dot g.
\end{equation}
Here we used $g=1$ at $t=0$ and the definition of $q$ in
\eqref{eq:barrier-scalar-notation}. Substituting \eqref{eq:barrier-test-function-derivatives} into
\eqref{eq:barrier-test-quotient-derivative} and dividing by two gives
\begin{align}\label{eq:barrier-terminal-cancellation}
 \langle B_{k+1},\kappa_{k+2}(\cdot,B_1)\rangle
  +\frac{xq(u)}h
       \langle B_k,\kappa_{k+1}(\cdot,B_1)\rangle=x(1+uq(u))|B_1|^2
 \le x(1+uq(u)),
\end{align}
where the last inequality follows by applying
\eqref{eq:barrier-linear-seed} to the unit test function $f_*$,
which gives $|B_1|^2\le1$, and by using the fact that $q(u)\ge0$.

By \eqref{eq:barrier-remainder-definition}, the terms
$\kappa_{k+2}(\cdot,B_1)$ and $\kappa_{k+1}(\cdot,B_1)$ occur in
$R_{k+1}$ and $R_k$, respectively.
When \eqref{eq:barrier-upper-energy} at degree $k+1$ and
\eqref{eq:barrier-adjustable-energy} at degree $k$ are substituted
into \eqref{eq:barrier-maximum-inequality}, these terms contribute
$-2$ times the left-hand side of
\eqref{eq:barrier-terminal-cancellation}.
Thus \eqref{eq:barrier-terminal-cancellation} controls their
combined contribution without a separate bound on
$\kappa_{k+2}(\cdot,B_1)$.

We next estimate the remaining terms in
\eqref{eq:barrier-remainder-definition}. For
$j\in\{k,k+1\}$, write
\[
 R_j=\kappa_{j+1}(\cdot,B_1)+R_j^\circ,
\]
where $R_j^\circ$ contains the summands with subset sizes
$2\le m\le j-1$. In particular, $R_2^\circ=0$.

For $1\le m\le k$ and $v\in\R^n$, apply
\eqref{eq:barrier-induction-hypothesis} to the linear function
$\ell_v(z)=v\cdot z$. Using
\eqref{eq:barrier-linear-cumulant-identity} and isotropy of
$\mu_*$ gives
\begin{equation}\label{eq:barrier-cumulant-contraction}
 \|\kappa_{m+1}(\cdot,v)\|_{\HS}^2
 =\|B_m^{\mu_*}(\ell_v)\|_{\HS}^2
 \le h_m\|\ell_v\|_{2,\mu_*}^2
 =h_m|v|^2.
\end{equation}
For $r\ge1$, the contraction of $\kappa_{m+1}$ with $B_r$
in one index is the tensor
$C\in\mathsf E_m\otimes\mathsf E_{r-1}$ defined by
\[
 C_{I K}=\sum_{i=1}^n(\kappa_{m+1})_{I i}(B_r)_{i K},
 \qquad |I|=m,\quad |K|=r-1.
\]
Here $I$ and $K$ are ordered tuples, and the tensor product
carries the product Hilbert--Schmidt norm. For each fixed $K$,
let $v_K\in\R^n$ have components $(v_K)_i=(B_r)_{iK}$.
Applying \eqref{eq:barrier-cumulant-contraction} to each $v_K$
and summing over $K$ yields
\[
 \|C\|_{\HS}^2
 =\sum_K\|\kappa_{m+1}(\cdot,v_K)\|_{\HS}^2
 \le h_m\sum_K|v_K|^2
 =h_m\|B_r\|_{\HS}^2.
\]
Thus the contraction has norm at most
$\sqrt{h_m}\|B_r\|_{\HS}$, with no dimension-dependent factor.

Each summand of $R_j^\circ$ is such a contraction with
$r=j-m+1$, up to a permutation of tensor indices.
Since $j\in\{k,k+1\}$ and $2\le m\le j-1$, both $m$ and $r$
are at most $k$. Applying
\eqref{eq:barrier-induction-hypothesis} to the unit test
function $f_*$ therefore bounds each summand by
$\sqrt{h_mh_{j-m+1}}$.

\subsection{A One-step Induction Criterion}
\label{subsec:one-step-coefficient}

At a pair maximizing the quotient
\eqref{eq:barrier-extremal-quotient}, we have obtained
\eqref{eq:barrier-maximum-inequality} by varying the measure and
\eqref{eq:barrier-terminal-cancellation} by varying the test
function. We have also bounded the remaining contractions using
the coefficient bounds through degree $k$. In this subsection,
we combine these estimates to give an explicit sufficient
condition for proving a proposed bound at degree $k+1$.
The condition retains the individual bounds at the preceding
degrees.

\begin{proposition}[One-step coefficient induction]
\label{prop:one-step-coefficient}
Fix $k\ge2$ and set $j=k+1$. Suppose that positive constants
$h_1=1,h_2,\ldots,h_k$ satisfy
\[
 \|B_m^\mu(f)\|_{\HS}^2
 \le h_m\|f\|_{2,\mu}^2
 \qquad(1\le m\le k)
\]
for every isotropic log-concave law $\mu$ and every
$f\in L^2(\mu)$. Let $h_j>0$ be a proposed bound at degree $j$,
and define
\[
 d_\ell
 =\sum_{m=2}^{\ell-1}
    \binom{\ell}{m}\sqrt{h_mh_{\ell-m+1}},
 \qquad \ell\in\{k,j\},
\]
with an empty sum equal to zero. These constants involve only
$h_2,\ldots,h_k$. Pick $0<\alpha<1$ and a positive, decreasing, concave $C^2$
function $\phi$ on a neighborhood of $[0,1]$, with $\phi(0)=1$.
Write
\[
 A=\frac4{1-\alpha}-2,\qquad
 q(u)=-\frac{\phi'(u)}{\phi(u)}.
\]
If, for every $u\in[0,1]$,
\begin{align}
 q(u)\left[
  \alpha\frac{h_j}{h_k}\phi(u)
  -\bigl(Ak^2+(A-3)k+2\bigr)u
  -\frac{2d_k}{\sqrt{h_k}}\sqrt u
 \right]
 >2j^2-j+2+\frac{2d_j}{\sqrt{h_j\phi(u)}},
 \label{eq:one-step-coefficient-condition}
\end{align}
then $\|B_j^\mu(f)\|_{\HS}^2
 \le h_j\|f\|_{2,\mu}^2$ for every isotropic log-concave law $\mu$ and every
$f\in L^2(\mu)$.
\end{proposition}

\begin{proof}
Fix a dimension $n$ and suppose that the proposed bound fails
for some $\mu\in\mathsf C_n$ and nonzero $f\in L^2(\mu)$.
Since $0<\phi(u)\le1$ on $[0,1]$, the quotient
\eqref{eq:barrier-extremal-quotient} then has supremum greater
than $h_j$.

The reduction to polynomial test functions and the compactness
argument preceding \eqref{eq:barrier-tilt-maximality} yield a
maximizing pair $(\mu_*,f_*)$ with
\[
 f_*\in\mathcal P_{\le j},\qquad
 \|f_*\|_{2,\mu_*}^2=1.
\]
Write $B_m=B_m^{\mu_*}(f_*)$ and
$\kappa_m=\kappa_m^{\mu_*}$. Evaluate the energies at this pair,
so that $g=1$. Using \eqref{eq:ejdef}, set
\[
 x=E_j=\|B_j\|_{\HS}^2,\qquad
 y=E_k=\|B_k\|_{\HS}^2=h_ku,\qquad
 w=\frac{xq(u)}{h_k}.
\]
The induction hypothesis gives $0\le u\le1$, and $w\ge0$
because $\phi$ is positive and decreasing. Moreover,
\[
 \mathcal F(\mu_*,f_*)=\frac{x}{\phi(u)}>h_j,
 \qquad\text{so}\qquad x>h_j\phi(u)>0.
\]
Both \eqref{eq:barrier-maximum-inequality} and
\eqref{eq:barrier-terminal-cancellation} hold at this same pair.

For $\ell\in\{k,j\}$, separate the term containing
$\kappa_{\ell+1}$ in \eqref{eq:barrier-remainder-definition}:
\[
 R_\ell=\kappa_{\ell+1}(\cdot,B_1)+R_\ell^\circ.
\]
Applying \eqref{eq:barrier-cumulant-contraction} and
\eqref{eq:barrier-induction-hypothesis} to each remaining
summand bounds its norm by
$\sqrt{h_mh_{\ell-m+1}}$. There are
$\binom{\ell}{m}$ subsets of size $m$, so the triangle inequality
gives
\begin{equation}\label{eq:one-step-remainder-bound}
 \|R_\ell^\circ\|_{\HS}
 \le\sum_{m=2}^{\ell-1}
       \binom{\ell}{m}\sqrt{h_mh_{\ell-m+1}}
 =d_\ell,
 \qquad \ell\in\{k,j\}.
\end{equation}
Combining \eqref{eq:barrier-terminal-cancellation} with
\eqref{eq:one-step-remainder-bound} and Cauchy--Schwarz yields
\begin{equation}\label{eq:one-step-cumulant-bound}
 \langle B_j,R_j\rangle
 +w\langle B_k,R_k\rangle
 \le x(1+uq(u))+\sqrt{x}\,d_j+w\sqrt{y}\,d_k.
\end{equation}

Now, we apply \eqref{eq:barrier-upper-energy} at degree $j$ and
\eqref{eq:barrier-adjustable-energy} at degree $k$ with
$\beta=\alpha$, for which $A_\alpha=A$.
Substituting these inequalities into
\eqref{eq:barrier-maximum-inequality}, and then using
\eqref{eq:one-step-cumulant-bound}, gives
\begin{equation}\label{eq:one-step-combined-inequality}
 \begin{aligned}
 0
 &\ge \mathcal Dx+w\mathcal Dy\\
 &\ge w\alpha x-(2j^2-j)x
      -w\bigl(Ak^2+(A-3)k\bigr)y\\
 &\qquad
      -2x(1+uq(u))-2\sqrt{x}\,d_j-2w\sqrt{y}\,d_k.
 \end{aligned}
\end{equation}
Dividing \eqref{eq:one-step-combined-inequality} by $x>0$
and substituting $w=xq(u)/h_k$ and $y=h_ku$ yields
\begin{equation}\label{eq:one-step-scalar-inequality}
 \begin{aligned}
 q(u)\left[
  \alpha\frac{x}{h_k}
  -\bigl(Ak^2+(A-3)k+2\bigr)u
  -\frac{2d_k}{\sqrt{h_k}}\sqrt u
 \right]
 -(2j^2-j+2)-\frac{2d_j}{\sqrt{x}}
 \le0.
 \end{aligned}
\end{equation}

Condition \eqref{eq:one-step-coefficient-condition} states
that the left-hand side of
\eqref{eq:one-step-scalar-inequality} is strictly positive
at $x=h_j\phi(u)$. It also implies $q(u)>0$, since
$q(u)\ge0$ and the right-hand side of
\eqref{eq:one-step-coefficient-condition} is positive.
For fixed $u$, the left-hand side of
\eqref{eq:one-step-scalar-inequality} is strictly increasing
in $x>0$. Since $x>h_j\phi(u)$ at the maximizing pair,
it must therefore be positive there as well.
This contradicts \eqref{eq:one-step-scalar-inequality}.

The proposed bound holds in the fixed dimension $n$.
Since $n$ was arbitrary, it holds in every dimension with
the same constant $h_j$.
\end{proof}

\subsection{Coefficient Bounds from Two Polynomial Barriers}
\label{sec:coefficient-certificates}

Thus far, we have established the linear and
quadratic coefficient bounds and reduced each subsequent induction
step to Proposition~\ref{prop:one-step-coefficient}. Joint maximization
over the measure and the test function controls the weighted sum of the two terminal cumulant contractions in \eqref{eq:barrier-terminal-cancellation}, while the remaining contractions are controlled by
the bounds already proved. To establish a proposed bound at the next
degree, it therefore suffices to choose a positive, decreasing,
concave denominator $\phi$ satisfying
\eqref{eq:one-step-coefficient-condition} throughout $[0,1]$.
We now construct such denominators and verify the criterion explicitly.

We define each denominator as the inverse of a polynomial on a
strictly decreasing, concave branch. The inverse is again decreasing
and concave, and the verification reduces to polynomial identities
and elementary inequalities. A quadratic polynomial gives the
degree-three bound. For the higher degrees, we estimate the sums
$d_\ell$ of lower-degree terms, treating the degree-two bound separately
from the proposed geometric bounds for the normalized coefficient norms. A single cubic polynomial then
verifies the criterion at every degree $d\ge4$, yielding the geometric
coefficient bound stated below.

\begin{theorem}\label{thm:geometric-coefficients}
For every dimension $n\ge1$, the constants
$c_d=\sup_{\mu\in\mathsf C_n}c_d(\mu)$ satisfy
\begin{equation}\label{eq:final-coefficient-tail}
 c_1=1,\qquad c_2^2\le2,\qquad
 c_d^2\le B K^{d-3}\quad(d\ge3),
 \qquad\text{where} \qquad B=\frac{121}{4},\quad K=\frac{223}{10}.
\end{equation}
\end{theorem}

\begin{proof}
The linear and quadratic bounds are available in
\eqref{eq:barrier-linear-seed} and
\eqref{eq:barrier-quadratic-seed}. We first establish the cubic bound
and then apply Proposition~\ref{prop:one-step-coefficient} inductively.

For $k=2$ and $j=3$, the constants in
Proposition~\ref{prop:one-step-coefficient} are
$d_2=0$ and $d_3=\binom{3}{2}\sqrt{8\cdot8}=24$.
Substituting these values, together with $h_2=8$ and $h_3=36B$,
into \eqref{eq:one-step-coefficient-condition} gives
\begin{equation}\label{eq:short-cubic-criterion}
 q(u)\left[
  \frac{9}{2}\alpha B\phi(u)-(6A-4)u
 \right]
 >17+\frac{8}{\sqrt{B\phi(u)}}.
\end{equation}
Pick $\alpha=7/13$, so that $A=20/3$ and $6A-4=36$, and define
\[
 U_3(v)=\frac7{117}(1-v)(65v+1).
\]
Its derivatives and endpoint values are
\[
 U_3'(v)=\frac{448-910v}{117},\qquad
 U_3''(v)=-\frac{910}{117},\qquad U_3(1)=0,
\]
\[
 U_3(1/2)-1=\frac1{468}>0,\qquad
 U_3'(1/2)=-\frac7{117}<0.
\]
Thus $U_3$ is strictly decreasing and concave on $[1/2,1]$, and its
range contains $[0,1]$. Its inverse $\phi$ on $[0,1]$ satisfies
\[
 \phi(0)=1,\qquad \phi>1/2,\qquad
 \phi'\!<0,\qquad
 \phi''(u)=-\frac{U_3''(\phi(u))}{U_3'(\phi(u))^3}<0.
\]
The derivative of $U_3$ is nonzero at the preimages of both endpoints;
therefore the same inverse extends smoothly to a neighborhood of
$[0,1]$.
Multiplying the displayed expression for $U_3'$ by $18v+9/16$ and
expanding $U_3$ gives
\[
 \left(18v+\frac9{16}\right)U_3'(v)
 =36U_3(v)-\frac{7623}{104}v
 =36U_3(v)-\frac92\alpha Bv.
\]
Writing $u=U_3(v)$ and using $\phi'(u)=1/U_3'(v)$ yields
\begin{equation}\label{eq:short-cubic-inverse-identity}
 -\phi'(u)\left[\frac92\alpha B\phi(u)-36u\right]
 =18\phi(u)+\frac9{16}.
\end{equation}
For $v>0$, the inequality
$\sqrt v\le(2/3)v+3/8$ follows from
$(\sqrt v-3/4)^2\ge0$. Since $8/\sqrt B=16/11<3/2$,
\[
 17v+\frac8{\sqrt B}\sqrt v
 <17v+\frac32\left(\frac23v+\frac38\right)
 =18v+\frac9{16}.
\]
By using this strict inequality in
\eqref{eq:short-cubic-inverse-identity} and dividing by $\phi(u)>0$, we obtain \eqref{eq:short-cubic-criterion}.
Proposition~\ref{prop:one-step-coefficient} now proves
$\|B_3(f)\|^2\le36B\|f\|_2^2$, or equivalently $c_3^2\le B$.
\vspace{0.1in}

For the subsequent induction define
\[
 C=\frac BK=\frac{605}{446},\qquad
 h_1=1,\quad h_2=8,\quad
 h_d=CK^{d-2}(d!)^2\quad(d\ge3).
\]
Then $h_3=36B$, so the cubic estimate supplies the initial case. Set
\[
 r=\sqrt{\frac CK}=\frac{55}{223}<\frac14.
\]
For the sums $d_\ell$ of Proposition~\ref{prop:one-step-coefficient},
we only need the bound
\begin{equation}\label{eq:short-one-exception-remainder}
 d_\ell\le\frac r2\,\ell^2\sqrt{h_\ell}
 \qquad(\ell\ge3).
\end{equation} 
We first treat $\ell=3$. Since $d_3=24$ and
$r\sqrt{h_3}=6C$, we have
\[
 \frac{d_3}{r\sqrt{h_3}}
 =\frac4C<\frac92,
\]
where we used $C>1$. This proves
\eqref{eq:short-one-exception-remainder} at $\ell=3$. Now let $\ell\ge4$. For $3\le m\le\ell-2$,
both degrees $m$ and $\ell-m+1$ are at least three.
Substituting $h_d=CK^{d-2}(d!)^2$ and $r=\sqrt{C/K}$
gives
\[
 \frac{\binom{\ell}{m}
       \sqrt{h_mh_{\ell-m+1}}}
      {r\sqrt{h_\ell}}
 =\ell-m+1.
\]
At degree two, the same formula would give $h_2=4C$,
whereas the actual value is $h_2=8=4C(2/C)$.
The two distinct summands with $m=2$ and $m=\ell-1$
each contain exactly one factor $h_2$ under the square
root. After division by $r\sqrt{h_\ell}$, they therefore
equal $(\ell-1)\sqrt{2/C}$ and $2\sqrt{2/C}$,
respectively. Summing over $2\le m\le\ell-1$, we obtain
\[
 \frac{d_\ell}{r\sqrt{h_\ell}}
 =\frac{\ell(\ell-1)}2-1
  +(\ell+1)\left(\sqrt{\frac2C}-1\right)
 <\frac{\ell^2-1}{2}
 <\frac{\ell^2}{2},
 \qquad \ell\ge4,
\]
where $\sqrt{2/C}-1<1/2$ follows from $C>1$.
Together with the case $\ell=3$, this proves
\eqref{eq:short-one-exception-remainder}
for every $\ell\ge3$.

At an induction step with $j=k+1$, the sums $d_k$ and
$d_j$ involve only the constants $h_2,\ldots,h_k$, so
\eqref{eq:short-one-exception-remainder} does not assume
the proposed coefficient bound at degree $j$. We now construct a single denominator $\phi$, obtained by
inverting a decreasing branch of a cubic polynomial, to verify
\eqref{eq:one-step-coefficient-condition} for every $k\ge3$. Choose
\[
 \alpha=\frac{17}{33},\qquad A=\frac{25}{4},
 \qquad P=\frac{446}{165},\qquad Q=\frac{281}{8415},
\]
and set
\[
 U(z)=Pz+Q-(P+Q)\left(\frac{17z+1}{18}\right)^3.
\]
Writing $w=(17z+1)/18$, we have
\[
 U'(z)=P-\frac{17}{6}(P+Q)w^2,\qquad
 U''(z)=-\frac{289}{54}(P+Q)w.
\]
In particular $U''<0$ for $z\ge4/7$. The identities
$P=4K/33$ and $Q=(P-1)/51$ give
\begin{equation}\label{eq:short-tail-inverse-identity}
 \left(\frac{17}{8}z+\frac18\right)U'(z)
 =\frac{51}{8}U(z)+\frac18-\alpha Kz.
\end{equation}
Indeed the cubic terms on the two sides are both
$-(51/8)(P+Q)w^3$, the coefficients of $z$ agree because
$(17/4)P=\alpha K$, and the constants agree because
$51Q=P-1$.
The rational endpoint evaluations are
\[
 U(1)=0,\qquad
 U(4/7)-1=\frac{32117}{36673560}>0,\qquad
 U'(4/7)=-\frac{20953}{476280}<0.
\]
Thus $U$ is strictly decreasing on $[4/7,1]$ and has an inverse
$\phi$ on $[0,1]$, with
\[
 \phi(0)=1,\qquad \phi>4/7,\qquad
 \phi'<0,\qquad \phi''=-\frac{U''(\phi)}{U'(\phi)^3}<0.
\]
As with $U_3$, the derivative is nonzero at both relevant endpoints,
so this inverse is smooth on a neighborhood of $[0,1]$.
The elementary inequality $\sqrt t\le(t+1)/2$ and $r<1/4$ imply
\[
 r\sqrt u\le\frac18(u+1),\qquad
 2z+r\sqrt z<\frac{17}{8}z+\frac18\quad(z>0).
\]
For $u=U(z)\in[0,1]$, the sign $U'(z)<0$ and
\eqref{eq:short-tail-inverse-identity} therefore give
\begin{align*}
 \alpha Kz-Au-r\sqrt u+(2z+r\sqrt z)U'(z) >\alpha Kz-\frac{51}{8}u-\frac18
       +\left(\frac{17}{8}z+\frac18\right)U'(z)=0.
\end{align*}
Since $q(u)=-1/(zU'(z))>0$, dividing by $-zU'(z)>0$, we obtain the barrier inequality
\begin{equation}\label{eq:short-bulk-margin}
 q(u)[\alpha K\phi(u)-Au-r\sqrt u]
       >2+\frac r{\sqrt{\phi(u)}}
 \qquad(0\le u\le1).
\end{equation}

We now use \eqref{eq:short-bulk-margin} to verify
\eqref{eq:one-step-coefficient-condition} at every induction
step $k\ge3$.  Suppose the coefficient bounds have been established through degree
$k\ge3$, and put $j=k+1$. Then $h_j/h_k=Kj^2$.
Divide the criterion in Proposition~\ref{prop:one-step-coefficient}
by $j^2$. Since $A=25/4$, $j\ge4$, and
\eqref{eq:short-one-exception-remainder} applies to both sums,
\[
 \begin{gathered}
 \frac{Ak^2+(A-3)k+2}{j^2}<A,\qquad
 2-\frac1j+\frac2{j^2}<2,\\
 \frac{2d_k}{j^2\sqrt{h_k}}\le r\frac{k^2}{j^2}<r,\qquad
 \frac{2d_j}{j^2\sqrt{h_j}}\le r.
 \end{gathered}
\]
Because $q(u)>0$ and $u\ge0$, the divided left side is at least
$q(u)[\alpha K\phi(u)-Au-r\sqrt u]$, while the divided right side
is strictly less than $2+r/\sqrt{\phi(u)}$.
Thus \eqref{eq:short-bulk-margin} verifies the criterion at every
$k\ge3$.
Proposition~\ref{prop:one-step-coefficient} proves the next coefficient
bound. The the degree-three bound starts the induction, and therefore
$c_d^2\le CK^{d-2}=BK^{d-3}$ for every $d\ge3$.
\end{proof}

\section{Curvature Estimates for Compatible Tensor Fields}\label{sec:full-hodge}

Section~\ref{sec:static-coefficients} established geometric bounds
on the normalized Appell coefficient norms. These bounds control
repeated centered integration of constant tensors. To obtain a
Poincar\'e bound, we need to extend this control to arbitrary
compatible tensor fields.

This section proves the curvature estimate needed for that
extension using a weighted Helmholtz--Hodge decomposition.
We use the compatible spaces $\mathcal C_r,\mathcal G_r$,
the curl, and the weighted divergence defined in
Section~\ref{sec:compatible-fields}.
The weighted divergence $\operatorname{div}_VF$ of a smooth
compatible tensor field $F$ may have nonzero curl.
We therefore decompose it into a compatible component and an
orthogonal correction with the same curl.
Compatibility cancels all derivatives of $F$ from this curl,
leaving an expression involving only $F$ and $D^2V$, as shown
in \eqref{full-divergence-curl}.
To bound the squared $L^2(\mu)$ norm of the correction, we use
the quadratic form of the constrained Hodge Laplacian, adapting
the framework of
\citet[Theorem~3.1 and Sections~4.1--4.5]{Johnsen}
to symmetric tensor fields.

Throughout this section, $\mu$ is regular in the sense of
Definition~\ref{def:regular-law}, and we set $\mathsf K=D^2V$.
We use the compatible gradient $D_r$ defined in
\eqref{eq:compatible-gradient}.
The curvature estimate \eqref{full-hodge-bound} is proved for
the potential test fields $F\in\mathcal K_{r+1}$ defined in
\eqref{eq:kmdef}. It retains the term $a\|F\|_2^2$ with
coefficient $1$, independently of tensor order.

Note that the fields in $\mathcal K_{r+1}$
belong to $\operatorname{Dom}(D_r^*)$. On this test domain define
$S_r:\mathcal K_{r+1}\to\mathcal G_r$ by
\begin{equation}\label{full-adjoint-projection}
 S_rF=P_{\mathcal G_r}\operatorname{div}_V F
     =P_{\mathcal C_r}\operatorname{div}_V F.
\end{equation}
The projections agree because integration by parts gives
$\mathbb E\operatorname{div}_V F=0$, while constant tensors belong to
$\mathcal C_r$. For $U\in\operatorname{Dom}(D_r)$, we have $ \langle U,S_rF\rangle_{2,\mu}
   =\langle D_rU,F\rangle_{2,\mu}.$ Thus $S_r$ agrees with $D_r^*$ on $\mathcal K_{r+1}$, that is,
\[
 S_rF=D_r^*F
 \qquad\text{for every }F\in\mathcal K_{r+1},
\] 
and the projected divergence is identified with the adjoint on the
stated test domain. The correction
$(I-P_{\mathcal C_r})\operatorname{div}_V F$ lies in
$\mathcal C_r^\perp\subset\mathcal X_r$.


\begin{lemma}[Compatible-tensor curvature estimate]\label{full-hodge}
For every $m\ge0$ and $F\in\mathcal K_{m+1}$,
\begin{equation}\label{full-hodge-bound}
 \|S_mF\|_2^2\ge\|\nabla F\|_2^2+a\|F\|_2^2.
\end{equation}
The coefficient of the curvature term does not depend on $m$.
\end{lemma}

\begin{proof}
By the definition of $\mathcal K_{m+1}$ in \eqref{eq:kmdef}, there exists $\psi\in C_c^\infty(\R^n)$ such that
$F=\nabla^{m+1}\psi$. We set $Y=\operatorname{div}_V F$. The weighted Helmholtz--Hodge decomposition is the orthogonal splitting
\begin{align*}
 Y=P_{\mathcal C_m}Y+(I-P_{\mathcal C_m})Y
 \quad\text{in }L^2(\mu;\mathsf E_m).
\end{align*}
The first term is compatible, and the second lies in
$\mathcal C_m^\perp$. For $m\ge1$, the second term
has the same curl as $Y$. 

On smooth compactly supported fields, the coordinate
derivative $\partial_i$ and the weighted adjoint
$\partial_j^*=-\partial_j+V_j$ satisfy
\[
 [\partial_i,\partial_j^*]
 =\partial_i\partial_j^*-\partial_j^*\partial_i
 =V_{ij},
\]
where, recall that $V_j=\partial_jV$ and
$V_{ij}=\partial_i\partial_jV$ acts by multiplication. For each fixed ordered tuple $I$ of length $m$, applying the
adjoint relation twice and using
$\partial_j\partial_i^*=\partial_i^*\partial_j+V_{ij}$ gives
\begin{align*}
 \sum_{i,j}
 \langle\partial_i^*F_{Ii},\partial_j^*F_{Ij}\rangle_{2,\mu}
 &=\sum_{i,j}\mathbb E[\partial_jF_{Ii}\,\partial_iF_{Ij}]
   +\sum_{i,j}\mathbb E[V_{ij}F_{Ii}F_{Ij}].
\end{align*}
Since $\nabla F$ is symmetric, summation over $I$ yields
\begin{equation}\label{full-divergence-energy}
 \|P_{\mathcal C_m}Y\|_2^2
 =\|\nabla F\|_2^2+B_{\mathsf{K}}(F)-\|(I-P_{\mathcal C_m})Y\|_2^2,
 \qquad\text{where}\qquad
 B_{\mathsf{K}}(F)=\sum_{I,i,j}\mathbb E[V_{ij}F_{Ii}F_{Ij}].
\end{equation}
Note that $B_{\mathsf K}:\mathcal X_{m+1}\to[0,\infty)$ is a
bounded quadratic form, and $\nabla F\in\mathcal X_{m+2}$ for
$F\in\mathcal K_{m+1}$. For $m=0$ the projection is the identity and $B_{\mathsf{K}}(F)\ge a\|F\|_2^2$, which proves the assertion. Assume $m\ge1$.
Differentiating $Y$ and using compatibility gives
\begin{equation}\label{full-divergence-curl}
 \partial_kY_{lI}-\partial_lY_{kI}
   =\sum_j(V_{kj}F_{lIj}-V_{lj}F_{kIj}),\qquad |I|=m-1.
\end{equation}
Compatibility makes the terms containing second derivatives of
$F$, and those containing first derivatives multiplied by $V_j$,
symmetric in $k,l$. These terms therefore cancel in the curl,
leaving only the Hessian terms in \eqref{full-divergence-curl}.
For example, when $m=1$ and $D^2V$ is diagonal at the point,
the $kl$ component of the curl is
$(V_{kk}-V_{ll})F_{kl}$, exhibiting the differences between
Hessian eigenvalues. We use the constrained complex introduced
below to bound the squared $L^2(\mu)$ norm of the orthogonal
correction by $B_{\mathsf K}(F)-a\|F\|_2^2$, thereby retaining
the term $a\|F\|_2^2$ in \eqref{full-divergence-energy}.

We now construct a differential complex that expresses the
compatibility condition for symmetric tensor fields. For symmetric degree $q\ge0$ and exterior degree $0\le p\le n$, let $\mathsf V_{q,p}=\mathsf E_q\otimes\Lambda^p\R^n$. A multi-index $\alpha=(\alpha_1,\ldots,\alpha_n)$
is a tuple of nonnegative integers. We write
$|\alpha|=\sum_{i=1}^n\alpha_i$ and
$\alpha!=\prod_{i=1}^n\alpha_i!$. On $\mathsf E_q$, use the orthonormal multiplicity
basis $e_\alpha$, $|\alpha|=q$, whose ordered components equal
$\sqrt{\alpha!/q!}$ on index tuples containing each \(i\) exactly \(\alpha_i\) times and vanish
otherwise. Define
\[
 a_i e_\alpha=\sqrt{\alpha_i}\,e_{\alpha-e_i},\qquad
 a_i^*e_\alpha=\sqrt{\alpha_i+1}\,e_{\alpha+e_i}.
\]
The first expression is zero if $\alpha_i=0$. These operators satisfy
$[a_i,a_j^*]=\delta_{ij}I$ and $\sum_i a_i^*a_i=qI$. For $q\ge1$, contraction with the coordinate vector $e_i$
is given by
\[
 T\mathbin{\lrcorner}e_i=\frac{a_iT}{\sqrt q},
 \qquad T\in\mathsf E_q.
\]
On $\Lambda^p\R^n$ recall that $\varepsilon_i$ denotes the wedge multiplication by $e_i$ as in~\eqref{eq:wedgemult}. Let $\iota_i=\varepsilon_i^*$, so
$\varepsilon_i\iota_j+\iota_j\varepsilon_i=\delta_{ij}I$.
On $\mathsf V_{q,p}=\mathsf E_q\otimes\Lambda^p\R^n$,
the operators $a_i,a_i^*$ act on the symmetric factor, while
$\varepsilon_i,\iota_i$ act on the exterior factor.
Each leaves the other factor unchanged.
Using the same symbols for these operators on $\mathsf V_{q,p}$,
their source and target spaces are
\[
 \begin{aligned}
 a_i&:\mathsf V_{q,p}\to\mathsf V_{q-1,p},
 &a_i^*&:\mathsf V_{q,p}\to\mathsf V_{q+1,p},\\
 \varepsilon_i&:\mathsf V_{q,p}\to\mathsf V_{q,p+1},
 &\iota_i&:\mathsf V_{q,p}\to\mathsf V_{q,p-1}.
 \end{aligned}
\]
The star on $a_i^*$ pairs adjacent symmetric degrees; for example,
$a_i^*:\mathsf E_q\to\mathsf E_{q+1}$ is the adjoint of
$a_i:\mathsf E_{q+1}\to\mathsf E_q$. Set
\begin{align}
\begin{aligned}\label{eq:bbstardef}
 b&=\sum_i\varepsilon_i a_i:
          \mathsf V_{q,p}\to\mathsf V_{q-1,p+1},\\
 b^*&=\sum_i a_i^*\iota_i:
          \mathsf V_{q,p}\to\mathsf V_{q+1,p-1}.
\end{aligned}
\end{align}
The commutation relations give
\begin{equation}\label{full-algebraic-complex}
 b^2=0,\qquad bb^*+b^*b=(q+p)I
 \quad\text{on }\operatorname{Sym}^q\mathbb R^n\otimes\Lambda^p\mathbb R^n.
\end{equation}
We remark that the second identity in \eqref{full-algebraic-complex} is the classical
polynomial de~Rham--Koszul homotopy identity; see, for example \citet[Theorem~3.1, Equation~(3.9)]{ArnoldFalkWinther}.
Under the identification $e_\alpha\leftrightarrow y^\alpha/\sqrt{\alpha!}$,
$b$ is exterior differentiation in the auxiliary variables $y$, and
$b^*$ is contraction with the radial vector field. To verify the second identity in \eqref{full-algebraic-complex}
directly in our normalization, commute $a_i$ past $a_j^*$
in $bb^*$ and use
$\iota_j\varepsilon_i=\delta_{ij}I-\varepsilon_i\iota_j$
in $b^*b$. The mixed terms cancel, leaving
$\sum_i a_i^*a_i+\sum_i\varepsilon_i\iota_i=(q+p)I$.  

Since the curl in \eqref{full-divergence-curl} has
$m-1$ symmetric indices and two exterior indices,
we now fix $q=m-1$, and put
\[
 \mathcal F_p=\ker(b|_{\mathsf V_{m-1,p}}),\qquad
 P_p=\frac{bb^*}{m+p-1}:\mathsf V_{m-1,p}\to\mathcal F_p,
 \qquad p\ge1.
\]
Here $P_p$ is the orthogonal projection onto $\mathcal F_p$: $bb^*$ has range in
$\ker b$ and equals $(m+p-1)I$ there by
\eqref{full-algebraic-complex}. The map
\[
 I_m: \mathsf E_m\longrightarrow \mathcal F_1,
 \qquad I_mT=\frac{bT}{\sqrt m},
\]
is an isometry onto $\mathcal F_1$. Its isometry property follows from
$b^*b=mI$ on exterior degree zero. For $\omega\in \mathcal F_1$,
$bb^*\omega=m\omega$, proving surjectivity. The inverse is $I_m^*=b^*/\sqrt m$, also on fields by the pointwise
extension convention of Section~\ref{sec:notation}.

For fields taking values in the full tensor spaces
$\mathsf V_{q,p}$, exterior differentiation and its weighted
adjoint have the following source and target spaces:
\[
 \begin{aligned}
 d&:\operatorname{Dom}(d)\subset L^2(\mu;\mathsf V_{q,p})
                   \to L^2(\mu;\mathsf V_{q,p+1}),
 &d&=\sum_i\varepsilon_i\partial_i,\\
 d^*&:\operatorname{Dom}(d^*)\subset L^2(\mu;\mathsf V_{q,p+1})
                   \to L^2(\mu;\mathsf V_{q,p}),
 &d^*&=\sum_i\iota_i\partial_i^*.
 \end{aligned}
\]
Each domain consists of the fields for which the displayed
distributional expression belongs to the indicated target $L^2$
space. The identities $bd+db=0$ and $d^2=0$ show that $d$
preserves the constraint. Its restrictions are
\[
 \begin{split}
 d_p&:\operatorname{Dom}(d_p)\subset L^2(\mu;\mathcal F_p)
                         \longrightarrow L^2(\mu;\mathcal F_{p+1}),\\
 \operatorname{Dom}(d_p)
   &=\{u\in L^2(\mu;\mathcal F_p):
                    du\in L^2(\mu;\mathcal F_{p+1})\}.
 \end{split}
\]
The operators $d_p$ form a differential complex,
meaning that $d_pu\in\operatorname{Dom}(d_{p+1})$ and
$d_{p+1}d_pu=0$ for every $u\in\operatorname{Dom}(d_p)$. These are closed, densely defined operators, and
$I_m(\mathcal C_m)=\ker d_1$.
Lemma~\ref{full-complex-domains} identifies their Hilbert adjoints:
\[
 \begin{split}
 d_p^*&:\operatorname{Dom}(d_p^*)\subset L^2(\mu;\mathcal F_{p+1})
                         \longrightarrow L^2(\mu;\mathcal F_p),\\
 \operatorname{Dom}(d_p^*)
   &=\{\omega\in L^2(\mu;\mathcal F_{p+1}):
                        P_pd^*\omega\in L^2(\mu;\mathcal F_p)\},
 \qquad d_p^*\omega=P_pd^*\omega.
 \end{split}
\]
The expression in this domain condition is distributional: it is
the projected expression that must be square integrable. Now, consider the nonnegative quadratic form
\[
 \ell(\omega)
 =\|d_2\omega\|_2^2+\|d_1^*\omega\|_2^2
 =\|d\omega\|_2^2+\|P_1d^*\omega\|_2^2,
\]
with domain $\operatorname{Dom}(\ell)
 =\operatorname{Dom}(d_2)\cap\operatorname{Dom}(d_1^*)
 \subset L^2(\mu;\mathcal F_2).$ The form is closed because $d_2$ and $d_1^*$ are closed.
Lemma~\ref{full-complex-domains} shows that every
$\omega\in\operatorname{Dom}(\ell)$ admits a sequence
$\omega_j\in C_c^\infty(\R^n;\mathcal F_2)$ such that $\|\omega_j-\omega\|_2^2+\ell(\omega_j-\omega)\longrightarrow0.$ Thus the same sequence approximates $\omega$ in $L^2(\mu)$
and its images under both $d_2$ and $d_1^*$ in their respective
$L^2(\mu)$ spaces.

Let $\Delta^H_{\mathcal F}$ denote the constrained Hodge Laplacian,
the nonnegative self-adjoint operator on
$L^2(\mu;\mathcal F_2)$ associated with $\ell$.
On smooth compactly supported fields, it acts as
\[
 \Delta^H_{\mathcal F}\omega
 =d_2^*d_2\omega+d_1d_1^*\omega.
\]
For $\omega,\eta\in\operatorname{Dom}(\ell)$, the associated
bilinear form is $\ell(\omega,\eta)
 =\langle d_2\omega,d_2\eta\rangle_{2,\mu}
  +\langle d_1^*\omega,d_1^*\eta\rangle_{2,\mu}.$ A field $\omega$ belongs to $\operatorname{Dom}(\Delta^H_{\mathcal F})$
precisely when $\omega\in\operatorname{Dom}(\ell)$ and there exists
$h\in L^2(\mu;\mathcal F_2)$ such that
\[
 \ell(\omega,\eta)=\langle h,\eta\rangle_{2,\mu}
 \qquad\text{for every }\eta\in\operatorname{Dom}(\ell).
\]
In this case, $\Delta^H_{\mathcal F}\omega=h$.
The argument below uses only the quadratic form $\ell$ and
does not require an inverse of $\Delta^H_{\mathcal F}$.

To bound the squared $L^2(\mu)$ norm of the orthogonal
correction in \eqref{full-divergence-energy}, we first derive
a lower bound for the quadratic form $\ell$. Initially let $\omega\in C_c^\infty(\mathbb R^n;\mathcal F_2)$. Define
\[
 \mathsf{K}_{\rm ext}=\sum_{i,j}V_{ij}\varepsilon_i\iota_j,\qquad
 \mathsf{K}_{\rm sym}=\sum_{i,j}V_{ij}a_j^*a_i,\qquad
 c=\sum_i a_i\partial_i^*,\quad c^*=\sum_i a_i^*\partial_i.
\]
For each $x$, $\mathsf K_{\rm ext}(x)$ and
$\mathsf K_{\rm sym}(x)$ are linear maps from \(\mathsf V_{q,p}\) to itself and they act by bounded multiplication on $L^2(\mu;\mathsf V_{q,p})$.
For the following calculations, we use $c$ and $c^*$ on
smooth, compactly supported fields with values in
$\mathsf V_{q,p}$:
\[
 \begin{aligned}
 c&:C_c^\infty(\R^n;\mathsf V_{q,p})
                  \to C_c^\infty(\R^n;\mathsf V_{q-1,p}),\\
 c^*&:C_c^\infty(\R^n;\mathsf V_{q,p})
                  \to C_c^\infty(\R^n;\mathsf V_{q+1,p}).
 \end{aligned}
\]
Across adjacent symmetric degrees they are formal adjoints for the
$L^2(\mu)$ pairing. The commutator identities below hold for all fields in
$C_c^\infty(\R^n;\mathsf V_{q,p})$.

Expanding $dd^*+d^*d$ gives
\[
 dd^*+d^*d
 =\sum_{i,j}(\varepsilon_i\iota_j+\iota_j\varepsilon_i)
                    \partial_j^*\partial_i
   +\sum_{i,j}V_{ij}\varepsilon_i\iota_j
 =\sum_i\partial_i^*\partial_i+\mathsf{K}_{\rm ext}.
\]
Also $bd^*+d^*b=c$. Since $b\omega=0$ and
$I-P_1=b^*b/m$, we have
\[
 \|(I-P_1)d^*\omega\|_2^2=\frac1m\|c\omega\|_2^2.
\]
To check the norm identity, note that $b^*b/m$ is an orthogonal projection, so its
squared norm on a vector $z$ equals
$\langle z,b^*bz\rangle/m=\|bz\|^2/m$.
Consequently
\begin{equation}\label{full-hook-energy}
 \ell(\omega)=\|\nabla\omega\|_2^2
       +\mathbb E\langle\omega,\mathsf{K}_{\rm ext}\omega\rangle
       -\frac1m\|c\omega\|_2^2.
\end{equation}
Similarly, expanding the commutator $cc^*-c^*c$ gives
\[
 cc^*-c^*c
 =\sum_i\partial_i^*\partial_i
   +\sum_{i,j}a_j^*a_i(\partial_i^*\partial_j-\partial_j\partial_i^*)
 =\sum_i\partial_i^*\partial_i-\mathsf{K}_{\rm sym}.
\]
The linear map
\[
 \bigoplus_{i=1}^n\mathsf V_{m-1,2}\longrightarrow\mathsf V_{m,2},
 \qquad (z_i)_{i=1}^n\longmapsto\sum_i a_i^*z_i,
\]
has squared norm $m$ whenever the exterior factor is nonzero
(and is the zero map otherwise), because its product with its adjoint is $\sum_i a_i^*a_i=mI$ on the target. Applied to the
components $\partial_i\omega$, this yields
$\|c^*\omega\|_2^2\le m\|\nabla\omega\|_2^2$. Thus, we have
\[
 \|c\omega\|_2^2
 =\|c^*\omega\|_2^2-\|\nabla\omega\|_2^2
                      +\mathbb E\langle\omega,\mathsf{K}_{\rm sym}\omega\rangle
 \le(m-1)\|\nabla\omega\|_2^2
                      +\mathbb E\langle\omega,\mathsf{K}_{\rm sym}\omega\rangle.
\]
Substituting the above display into \eqref{full-hook-energy} gives
\begin{equation}\label{full-hook-lower-bound}
 \ell(\omega)\ge\frac1m\|\nabla\omega\|_2^2
                    +\mathbb E\langle\omega,Q\omega\rangle,
 \qquad\text{where}\qquad Q=P_2(\mathsf{K}_{\rm ext}-m^{-1}\mathsf{K}_{\rm sym})P_2.
\end{equation}
Note that for each $x$, $Q(x)$ is a self-adjoint linear map from
$\mathcal F_2$ to itself. We also write $Q$ for its action
on fields, defined by $(Q\omega)(x)=Q(x)\omega(x).$ The uniform bound $\|\mathsf K(x)\|_{\mathrm{op}}\le M$,
from the regularity assumptions in
Subsection~\ref{sec:notation}, together with the definition
of $Q$ in \eqref{full-hook-lower-bound}, implies
$\sup_x\|Q(x)\|_{\mathrm{op}}<\infty$.
Consequently, $\|Q\omega\|_2
 \le \left(\sup_x\|Q(x)\|_{\mathrm{op}}\right)\|\omega\|_2.$ Thus $Q$ is a bounded operator on $L^2(\mu;\mathcal F_2)$.

We now complete a square at each point to compare the curl
in \eqref{full-divergence-curl} with the curvature terms in
\eqref{full-divergence-energy} and \eqref{full-hook-lower-bound}. At a fixed point, choose an orthonormal eigenbasis of
$\mathsf K$, with eigenvalues $\kappa_i\ge a$, and set $N=m+1$.
The following calculations are algebraic, so no derivatives
of this basis enter. 

Decompose $\mathsf V_{m-1,2}$ into
orthogonal blocks indexed by $\alpha$, where $\alpha_i$
counts the occurrences of index $i$ among the $m-1$
symmetric and two exterior indices. Thus $|\alpha|=N$.
Within each block, retain only indices with $\alpha_i>0$
and set $r_i=\sqrt{\alpha_i}$ and $S=\sum_i\alpha_i\kappa_i.$
An orthonormal basis of this block of $\mathsf V_{m-1,2}$ is
\[
 e_{\alpha-e_i-e_j}\otimes(e_i\wedge e_j),
 \qquad i<j,\quad \alpha_i,\alpha_j>0.
\]
In this block, the map $b$ from \eqref{eq:bbstardef} acts by wedge
multiplication with $r$: each nonzero term transfers an index
$i$ from the symmetric to the exterior factor with coefficient
$\sqrt{\alpha_i}$. In these coordinates,
$b^*=\iota_r:=\sum_i r_i\iota_i$.
For $\omega\in\mathcal F_2=\ker b$,
\eqref{full-algebraic-complex} gives
$r\wedge\iota_r\omega=N\omega$. Hence, we have
\[
 \omega=r\wedge v,\qquad
 v=N^{-1}\iota_r\omega\perp r,\qquad
 |\omega|^2=N|v|^2.
\]
Orthogonality follows from $\iota_r^2=0$.
The identity $\iota_r(r\wedge v)=Nv$ for $v\perp r$
proves uniqueness and the norm formula.
If only one $\alpha_i$ is positive, this block is zero
and $v=0$.

Within the same multiplicity block, index the coordinates by
the exterior pairs $i<j$. The operator
$\mathsf K_{\rm ext}-m^{-1}\mathsf K_{\rm sym}$ between the
projections $P_2$ in \eqref{full-hook-lower-bound} is diagonal
in these coordinates, with coefficient
$[N(\kappa_i+\kappa_j)-S]/m$ on the pair $ij$.
Expanding the squares and using $|r|^2=N$ and
$\sum_i\kappa_i r_i^2=S$ gives
\[
 \sum_{i<j}(\kappa_i+\kappa_j)(r_iv_j-r_jv_i)^2
 =S|v|^2+N\sum_i\kappa_iv_i^2
       -2(r\cdot v)\sum_i\kappa_ir_iv_i.
\]
The final term vanishes because $v\perp r$. Combining this equality with
$|r\wedge v|^2=N|v|^2$ proves
\begin{equation}\label{full-block-Q}
 \langle\omega,Q\omega\rangle=\frac{N^2}{m}\langle v,\mathsf{K}v\rangle
 \ge\frac{Na}{m}|\omega|^2.
\end{equation}
This proves that $Q(x)$ is positive
on $\mathcal F_2=\ker b$. The operator
$\mathsf K_{\rm ext}-m^{-1}\mathsf K_{\rm sym}$ on the full
space $\mathsf V_{m-1,2}$ need not be positive before
projection by $P_2$.

Let $f=F_\alpha$ be the coefficient of $F$ in the
normalized basis of $\mathsf E_N$, and let $t$ denote
the block of $dI_mY$ indexed by $\alpha$.
We first compute $dI_mY$ in fixed Euclidean coordinates.
Since
\[
 Y=\frac1{\sqrt N}\sum_j a_j\partial_j^*F,
 \qquad I_m=\frac b{\sqrt m},
\]
the definition of $b$ in \eqref{eq:bbstardef} gives
\begin{align*}
 dI_mY
 =\frac1{\sqrt{Nm}}
   \sum_{k,i,j}\varepsilon_k\varepsilon_i a_i a_j
       \partial_k\partial_j^*F=\frac1{\sqrt{Nm}}
   \sum_{k,i,j}\varepsilon_k\varepsilon_i a_i a_j
       \bigl(\partial_j^*\partial_kF+V_{kj}F\bigr)=\frac1{\sqrt{Nm}}
   \sum_{k,i,j}V_{kj}\varepsilon_k\varepsilon_i a_i a_jF.
\end{align*}
The second equality uses
$\partial_k\partial_j^*=\partial_j^*\partial_k+V_{kj}$.
For the third equality, $F=\nabla^N\psi$ gives
$a_i a_j\partial_kF=a_k a_j\partial_iF$.
Thus the terms containing $\partial_j^*\partial_kF$
cancel under interchange of $i$ and $k$, because
$\varepsilon_k\varepsilon_i=-\varepsilon_i\varepsilon_k$.

At the point under consideration, we now use the chosen
eigenbasis of $\mathsf K$. Since $V_{kj}=\kappa_k\delta_{kj}$,
the surviving sum becomes
\[
 dI_mY
 =\frac1{\sqrt{Nm}}
   \sum_{i<j}(\kappa_i-\kappa_j)
       \varepsilon_i\varepsilon_j a_i a_jF.
\]
Using
$a_i a_j e_\alpha
 =\sqrt{\alpha_i\alpha_j}\,e_{\alpha-e_i-e_j}$
for $i<j$, we obtain
\begin{align*}
 t_{ij}
 =\frac{(\kappa_i-\kappa_j)\sqrt{\alpha_i\alpha_j}}
        {\sqrt{Nm}}f
 =-\frac f{\sqrt{Nm}}
   \bigl(r_i(\mathsf K r)_j-r_j(\mathsf K r)_i\bigr),\quad\text{and consequently}\quad t=-\frac f{\sqrt{Nm}}\,r\wedge\mathsf K r.
\end{align*}

Since $F$ is symmetric, the quadratic form obtained by applying
$\mathsf K$ in one tensor index is independent of the chosen
index. Averaging over all $N$ indices therefore gives the
contribution $(S/N)f^2$ of this block to the integrand of
$B_{\mathsf K}(F)$ in \eqref{full-divergence-energy}.
Using \eqref{full-block-Q}, $v\perp r$, and
$\mathsf K\succeq aI$, we obtain
\begin{align*}
 \frac{S}{N}f^2-2\langle t,\omega\rangle
       +\langle\omega,Q\omega\rangle&=
 \left\langle
   \frac{f}{\sqrt N}r+\frac{N}{\sqrt m}v,\,
   \mathsf K\left(
     \frac{f}{\sqrt N}r+\frac{N}{\sqrt m}v
   \right)
 \right\rangle\ge
 a\left(f^2+\frac{N^2}{m}|v|^2\right)
 \ge af^2.
\end{align*}
Thus
$2\langle t,\omega\rangle-\langle\omega,Q\omega\rangle
\le (S/N-a)f^2$ in each block.

For $\omega\in\operatorname{Dom}(\ell)$,
Lemma~\ref{full-complex-domains} provides
$\omega_j\in C_c^\infty(\R^n;\mathcal F_2)$ such that
$\|\omega_j-\omega\|_2^2+\ell(\omega_j-\omega)\to0$.
Applying \eqref{full-hook-lower-bound} to
$\omega_j-\omega_k$ and using $Q\ge0$ shows that
$\nabla\omega_j$ is Cauchy in $L^2(\mu)$, i.e., $\|\nabla\omega_j-\nabla\omega_k\|_2\longrightarrow 0$
as $j,k\to\infty$. The $L^2(\mu)$ limit of $\nabla\omega_j$ is hence the weak gradient
$\nabla\omega$, since $\omega_j\to\omega$ in $L^2(\mu)$. Since $Q$ is bounded and $\ell(\omega_j)\to\ell(\omega)$,
passing to the limit proves \eqref{full-hook-lower-bound}
for every $\omega\in\operatorname{Dom}(\ell)$.

Finally, we apply \eqref{full-projection-variational} from
Lemma~\ref{full-complex-domains} to $u=I_mY$.
Since $I_m$ preserves norms and
$I_m(\mathcal C_m)=\ker d_1$, we obtain
\begin{equation}\label{eq:hodge-correction-variational}
 \|(I-P_{\mathcal C_m})Y\|_2^2
 =\sup_{\omega\in\operatorname{Dom}(\ell)}
   \left\{
    2\langle d_1I_mY,\omega\rangle_{2,\mu}
    -\ell(\omega)
   \right\}.
\end{equation}
Summing the completed-square bound over all multiplicity blocks
and integrating, and then using \eqref{full-hook-lower-bound},
(omitting its nonnegative gradient term) we obtain that for every
$\omega\in\operatorname{Dom}(\ell)$,
\[
 2\langle d_1I_mY,\omega\rangle_{2,\mu}-\ell(\omega)
 \le B_{\mathsf K}(F)-a\|F\|_2^2.
\]
Taking the supremum, we can conclude that
\[
 \|(I-P_{\mathcal C_m})Y\|_2^2
 \le B_{\mathsf K}(F)-a\|F\|_2^2.
\]
Substituting this bound into \eqref{full-divergence-energy}
and using \eqref{full-adjoint-projection}, we obtain
\eqref{full-hodge-bound}.
\end{proof}

\section{Centered Integration and Appell Polynomials}
\label{full-integration-section}

In this section, we connect the estimate for compatible tensor fields to the Appell coefficient bounds through centered integration, the inverse of the compatible gradient $D_r$ introduced in Section~\ref{sec:notation}. Bounding this inverse in the scalar case gives a Poincar\'e bound. The orthogonal splitting \eqref{full-compatible-split} separates the inputs to centered integration into constant tensors and mean-zero compatible fields. Integrating a constant tensor gives a linear tensor field, and repeated centered integration starting from constants produces factorially normalized Appell polynomials. The curvature estimate compares the integration operators at successive tensor orders, providing the comparison needed to extend the polynomial bounds to arbitrary compatible inputs.

Throughout this section, $\mu$ is regular in the sense of
Definition~\ref{def:regular-law} and isotropic. First note that, by Lemma~\ref{full-primitive}, the compatible gradient $D_r$
defined in \eqref{eq:compatible-gradient} is bijective, with
bounded inverse given by
\[
 \widetilde J_r=D_r^{-1}:
       \mathcal C_{r+1}\longrightarrow\mathcal G_r.
\]
For every $F\in\mathcal C_{r+1}$, this inverse satisfies
\begin{equation}\label{full-primitive-properties}
 \nabla\widetilde J_rF=F,\qquad
 \mathbb E_\mu\widetilde J_rF=0,\qquad
 \|\widetilde J_rF\|_2^2\le C_P(\mu)\|F\|_2^2.
\end{equation}
The finiteness of $C_P(\mu)$ follows independently from the
preliminary Brascamp--Lieb estimate
\eqref{full-initial-poincare}, which gives $C_P(\mu)\le a^{-1}$.

The range of $\widetilde J_r$ is
$\operatorname{Dom}(D_r)$, defined in
\eqref{eq:compatible-gradient}. Applying
\eqref{full-primitive-properties} and
\eqref{full-initial-poincare}, we obtain
\[
 \|\widetilde J_r\|_{\mathcal C_{r+1}\to\mathcal G_r}
 \le\sqrt{C_P(\mu)}\le a^{-1/2}.
\]
For $r=0$, $\widetilde J_0$ maps each compatible vector field
to its unique mean-zero scalar primitive. Conversely, every
$f\in W^{1,2}(\mu)\cap L^2_0(\mu)$ satisfies
$\widetilde J_0(\nabla f)=f$. The definition of the scalar
Poincar\'e constant therefore gives
\[
 \|\widetilde J_0\|^2
 =\sup_{\substack{f\in W^{1,2}(\mu)\cap L^2_0(\mu)\\f\ne0}}
       \frac{\|f\|_{2,\mu}^2}{\|\nabla f\|_{2,\mu}^2}
 =C_P(\mu).
\]

Using the orthogonal decomposition
\eqref{full-compatible-split}, we restrict $\widetilde J_r$
to mean-zero compatible fields and constant tensors:
\begin{equation}\label{full-J-L-definition}
 \begin{aligned}
 J_r&=\widetilde J_r|_{\mathcal G_{r+1}}:
       \mathcal G_{r+1}\longrightarrow\mathcal G_r,\\
 L_r&=\widetilde J_r|_{\mathsf E_{r+1}}:
       \mathsf E_{r+1}\longrightarrow\mathcal G_r.
 \end{aligned}
\end{equation}
For $T\in\mathsf E_{r+1}$ and $F_0\in\mathcal G_{r+1}$,
linearity gives
\[
 \widetilde J_r(T+F_0)=L_rT+J_rF_0.
\]
Note that while inputs $T$ and $F_0$ are orthogonal in $L^2(\mu)$, their primitives $L_rT$ and $J_rF_0$ need not be.

For $T\in\mathsf E_{r+1}$, the contraction
$x\mapsto T\mathbin{\lrcorner}x$ defined by
\eqref{eq:contraction} has gradient $T$ and is compatible
by symmetry of $T$. Its mean vanishes because
$\mathbb E_\mu X=0$. Uniqueness in
Lemma~\ref{full-primitive} therefore gives
\begin{equation}\label{full-linear-primitive}
 (L_rT)(x)=T\mathbin{\lrcorner}x.
\end{equation}
For an ordered index tuple $I$ of length $r$,
\eqref{full-linear-primitive} reads as $(L_rT)_I(x)=\sum_{j=1}^nT_{Ij}x_j.$ By using isotropy in this formula, we obtain
\begin{equation}\label{full-L-bound}
 \|L_rT\|_2^2=\sum_I\sum_{j,k=1}^n
       T_{Ij}T_{Ik}\,\mathbb E_\mu[X_jX_k]=\sum_I\sum_{j=1}^nT_{Ij}^2
 =\|T\|_{\HS}^2.
\end{equation}
Note that this calculation includes $r=0$, when $I$ is the empty tuple.
Combining \eqref{full-L-bound} with
\eqref{full-J-L-definition} and
\eqref{full-primitive-properties} gives
\begin{equation}\label{full-rankwise-bounds}
 \|L_r\|_{\mathsf E_{r+1}\to\mathcal G_r}=1,\qquad
 \|J_r\|_{\mathcal G_{r+1}\to\mathcal G_r}^2
 \le C_P(\mu),
 \qquad r\ge0.
\end{equation}

Since these operators are bounded, their adjoints are bounded maps
\[
 \widetilde J_r^*:\mathcal G_r\to\mathcal C_{r+1},\qquad
 J_r^*:\mathcal G_r\to\mathcal G_{r+1},\qquad
 L_r^*:\mathcal G_r\to\mathsf E_{r+1}.
\]
For $u\in\mathcal G_r$, the adjoint identity applied to
\eqref{full-linear-primitive} gives $L_r^*u=\Sym\,\mathbb E_\mu[u(X)\otimes X].$ Taking adjoints in \eqref{full-J-L-definition} and using
the orthogonal decomposition
\eqref{full-compatible-split}, we also obtain $\widetilde J_r^*u=L_r^*u+J_r^*u.$ For $F\in\mathcal K_{r+1}$ and $U\in\mathcal C_{r+1}$,
integration by parts, \eqref{full-adjoint-projection},
and \eqref{full-primitive-properties} give
\begin{align*}
 \langle\widetilde J_r^*S_rF,U\rangle_{2,\mu}=\langle S_rF,\widetilde J_rU\rangle_{2,\mu}=\langle F,\nabla\widetilde J_rU\rangle_{2,\mu} =\langle F,U\rangle_{2,\mu}.
\end{align*}
Thus $\widetilde J_r^*S_rF=F$. Separating the mean-zero
and constant components using
\eqref{full-compatible-split}, we obtain
\[
 J_r^*S_rF=F-\mathbb E_\mu F,\qquad
 L_r^*S_rF=\mathbb E_\mu F
 \qquad(F\in\mathcal K_{r+1}).
\]

\subsection{Variational Formulas and Curvature}

We start by defining
\[
 \begin{aligned}
 H_r&=D_r^*D_r:
       \operatorname{Dom}(H_r)\subset\mathcal G_r
       \longrightarrow\mathcal G_r,\\
 \operatorname{Dom}(H_r)
   &=\{U\in\operatorname{Dom}(D_r):
                  D_rU\in\operatorname{Dom}(D_r^*)\}.
 \end{aligned}
\]
Since $D_r$ is closed and densely defined by
Lemma~\ref{full-primitive}, the representation theorems for
closed forms \cite[Chapter~VI, Theorems~2.1, 2.6, and~2.23]{Kato}
identify $H_r$ as the nonnegative self-adjoint operator
associated with the closed form
\[
 h_r(U,W)=\langle\nabla U,\nabla W\rangle_{2,\mu},
 \qquad U,W\in\operatorname{Dom}(D_r).
\]
In particular, we have that 
\[
 \begin{aligned}
 \operatorname{Dom}(H_r^{1/2})
   &=\operatorname{Dom}(D_r),\qquad \text{where}\qquad
 \|H_r^{1/2}U\|_{2,\mu}^2
   &=\|\nabla U\|_{2,\mu}^2
   \qquad(U\in\operatorname{Dom}(D_r)).
 \end{aligned}
\]
Applying the scalar Poincar\'e inequality to each component of
$U\in\operatorname{Dom}(D_r)$ gives $ H_r\succeq C_P(\mu)^{-1}I.
$ Consequently, $H_r^{-1}:\mathcal G_r\to\mathcal G_r$ is bounded,
with $\|H_r^{-1}\|\le C_P(\mu)$ and range
$\operatorname{Dom}(H_r)$. For $r=0$, note that $D_0$ is the scalar gradient. Hence
$H_0=D_0^*D_0$ is the restriction of the weighted Laplacian
$\mathcal L_V$, defined in \eqref{eq:weighted-laplacian},
to the mean-zero space $L^2_0(\mu)$.

To convert the lower bound in \eqref{full-hodge-bound}
into an upper bound for integration, we use the following
standard variational identity for strictly positive self-adjoint
operators. It follows by completing the square in the
quadratic-form representation as in
\citet[Chapter~VI, Theorem~2.23]{Kato}: Let
$B:\operatorname{Dom}(B)\subset\mathcal U\to\mathcal U$
be self-adjoint on a Hilbert space $\mathcal U$, with
$B\succeq bI$ for some $b>0$. Then, for every $v\in\mathcal U$,
\begin{equation}\label{full-inverse-variational}
 \langle v,B^{-1}v\rangle_{\mathcal U}
 =\sup_{U\in\operatorname{Dom}(B^{1/2})}
       \left\{2\langle v,U\rangle_{\mathcal U}
                    -\|B^{1/2}U\|_{\mathcal U}^2\right\}.
\end{equation}
For $U\in\operatorname{Dom}(B^{1/2})$, completing the square gives
\begin{align*}
 \langle v,B^{-1}v\rangle_{\mathcal U}
 -2\langle v,U\rangle_{\mathcal U}
 +\|B^{1/2}U\|_{\mathcal U}^2=\|B^{1/2}U-B^{-1/2}v\|_{\mathcal U}^2\ge0.
 \end{align*}
Since $B^{-1}v\in\operatorname{Dom}(B)
\subset\operatorname{Dom}(B^{1/2})$, equality is attained at
$U=B^{-1}v$, proving \eqref{full-inverse-variational}.

We now apply \eqref{full-inverse-variational} with $B=H_r$,
$\mathcal U=\mathcal G_r$, and for $v\in\mathcal G_r$. Since $\operatorname{Dom}(H_r^{1/2})=\operatorname{Dom}(D_r)$,
Lemma~\ref{full-primitive} gives the unique representation
\[
 U=\widetilde J_rF,\qquad
 F=D_rU\in\mathcal C_{r+1},
 \qquad U\in\operatorname{Dom}(H_r^{1/2}).
\]
Therefore, we have
\begin{align*}
 \langle v,H_r^{-1}v\rangle
 =\sup_{F\in\mathcal C_{r+1}}
       \{2\langle\widetilde J_r^*v,F\rangle-\|F\|_2^2\}=\|\widetilde J_r^*v\|_2^2.
\end{align*}
The orthogonal decomposition of the inputs of $\widetilde J_r$ now provides
\begin{equation}\label{full-inverse-identities}
 H_r^{-1}=\widetilde J_r\widetilde J_r^*
          =J_rJ_r^*+L_rL_r^*.
\end{equation}

To proceed with deriving the curvature comparison, we need to show that
$S_r(\mathcal K_{r+1})$ is dense in $\mathcal G_r$ for every
$r\ge0$. Here $S_r$ is the projected weighted divergence
defined in \eqref{full-adjoint-projection}, and
$\mathcal K_{r+1}$ is the space of compactly supported
potential fields defined in \eqref{eq:kmdef}.
General range and projection results for closed differential
operators are given in \citet[Theorem~3.1]{Johnsen}.
Since weighted divergence is unbounded, denseness of
$\mathcal K_{r+1}$ in $\mathcal C_{r+1}$ alone does not
establish the required denseness of its projected divergences.
In the following, we prove that every field in $\mathcal G_r$ orthogonal to
$S_r(\mathcal K_{r+1})$ vanishes.

\begin{lemma}[Dense Projected Divergences]
\label{full-divergence-range}
For every $r\ge0$, the subspace
\[
 S_r(\mathcal K_{r+1})
 =\{S_r(\nabla^{r+1}\psi):
                  \psi\in C_c^\infty(\mathbb R^n)\}
\]
is dense in $\mathcal G_r$.
\end{lemma}

\begin{proof}
We fix $r\ge0$ and suppose that $u\in\mathcal G_r$ is
orthogonal to $S_r(\mathcal K_{r+1})$. For $r\ge1$,
repeated application of Lemma~\ref{full-primitive} gives
a scalar function $\phi\in W^{r,2}(\mu)$ such that
$u=\nabla^r\phi$. For $r=0$, we set $\phi=u$.
Writing $q=r+1$, we therefore have
\[
 u=\nabla^{q-1}\phi,\qquad
 \phi\in W^{q-1,2}(\mu).
\]
We write $w=Z^{-1}e^{-V}$ for the density of $\mu$.
Using \eqref{full-adjoint-projection} and integrating by parts
in the sense of distributions, the assumed orthogonality gives
\begin{equation}\label{eq:dense-divergence-weak-equation}
 \int\langle\nabla^q\phi,\nabla^q\psi\rangle\,d\mu=0
 \qquad\text{for every }\psi\in C_c^\infty(\mathbb R^n).
\end{equation}
Initially, the integral denotes the distributional pairing
of $\nabla^q\phi$ with the smooth compactly supported field
$w\nabla^q\psi$.

Expanding \eqref{eq:dense-divergence-weak-equation}
in distributions gives $ \sum_I\partial_I(w\partial_I\phi)=0,
$ where $I$ runs over ordered tuples of length $q$.
The term containing $2q$ derivatives of $\phi$ is
$w\Delta^q\phi$. Every remaining term has a smooth coefficient
and contains at most $2q-1$ derivatives of $\phi$.
Since $w$ is bounded away from zero on compact sets,
$\phi\in W^{q-1,2}(\mu)$ implies
$\phi\in H^{q-1}_{\rm loc}$. The remaining terms therefore
belong to $H^{-q}_{\rm loc}$. Dividing by the smooth
strictly positive function $w$, we obtain $ \Delta^q\phi\in H^{-q}_{\rm loc}.$

We now use \eqref{eq:dense-divergence-weak-equation}
to improve the local regularity from
$\phi\in H^{q-1}_{\rm loc}$ to $\phi\in H^q_{\rm loc}$. For $\zeta\in C_c^\infty(\mathbb R^n)$, the product rule
shows that
$\Delta^q(\zeta\phi)-\zeta\Delta^q\phi$
has compactly supported coefficients and contains at most
$2q-1$ derivatives of $\phi$. Consequently,
$\Delta^q(\zeta\phi)\in H^{-q}(\mathbb R^n)$.
Using the Fourier definition of the Sobolev norms in
Section~\ref{sec:notation}, we obtain
\[
 \int_{\mathbb R^n}
 \frac{|\xi|^{4q}}{(1+|\xi|^2)^q}
 |\widehat{\zeta\phi}(\xi)|^2\,d\xi<\infty.
\]
For $|\xi|\ge1$, the weight is at least
$2^{-2q}(1+|\xi|^2)^q$.
For $|\xi|\le1$, the required integrability follows from
$\zeta\phi\in H^{q-1}(\mathbb R^n)$.
Thus $\zeta\phi\in H^q(\mathbb R^n)$ for every such $\zeta$,
which proves $\phi\in H^q_{\rm loc}$.

We choose $\eta\in C_c^\infty(\mathbb R^n)$ with
$0\le\eta\le1$ and $\eta=1$ on $B_1$, and set
$\eta_R(x)=\eta(x/R)$ for $R\ge1$.
By the local $H^q$ regularity just proved, approximation
with smooth compactly supported functions permits the test
$\psi=\eta_R^{2q}\phi$ in
\eqref{eq:dense-divergence-weak-equation}.
The product rule gives
\begin{equation}\label{eq:dense-divergence-cutoff}
 \begin{aligned}
 \nabla^q(\eta_R^{2q}\phi)
   &=\eta_R^{2q}\nabla^q\phi+R_R,\\
 \|R_R(x)\|_{\HS}
   &\le C_q\eta_R(x)^q
       \sum_{j=1}^qR^{-j}
                \|\nabla^{q-j}\phi(x)\|_{\HS}.
 \end{aligned}
\end{equation}
Indeed, each term in $R_R$ places $j\le q$ derivatives
on the product of $2q$ cutoff factors, leaving at least
$2q-j\ge q$ factors undifferentiated. The differentiated
factors contribute $R^{-j}$.
Here $C_q$ depends on $q$ and the chosen cutoff,
but not on $R$.

Now, define
\[
 A_R=\|\eta_R^q\nabla^q\phi\|_2,\qquad
 b_R=C_q\sum_{j=1}^qR^{-j}
                      \|\nabla^{q-j}\phi\|_2.
\]
Since $\phi\in W^{q-1,2}(\mu)$, we have $b_R\to0$.
The estimate in \eqref{eq:dense-divergence-cutoff} gives
\[
 \left\|\frac{R_R}{\eta_R^q}\right\|_2\le b_R,
\]
where the quotient is defined to be zero wherever $\eta_R=0$.
Substituting $\psi=\eta_R^{2q}\phi$ into
\eqref{eq:dense-divergence-weak-equation} and using
\eqref{eq:dense-divergence-cutoff}, we obtain
\[
 A_R^2
 =-\left\langle
       \eta_R^q\nabla^q\phi,\frac{R_R}{\eta_R^q}
   \right\rangle_{2,\mu}
 \le A_Rb_R
\]
by Cauchy--Schwarz. Hence $A_R\le b_R$.
Since $\eta_R\to1$ pointwise, Fatou's lemma gives
\[
 \int\|\nabla^q\phi\|_{\HS}^2\,d\mu
 \le\liminf_{R\to\infty}A_R^2=0.
\]
It hence follows that $\nabla u=\nabla^q\phi=0$, so $u$ is
a constant tensor. Since $u\in\mathcal G_r$ has zero mean,
$u=0$. The orthogonal complement of
$S_r(\mathcal K_{r+1})$ in $\mathcal G_r$ is therefore
$\{0\}$, proving the required denseness.
\end{proof}

We now proceed to use the curvature estimate \eqref{full-hodge-bound}
to prove the required comparison for the integration operators. To this end, we fix
$v\in\mathcal G_{r+1}$ and derive an upper bound for
$\|J_rv\|_2^2$. By
Lemma~\ref{full-divergence-range},
$S_r(\mathcal K_{r+1})$ is dense in $\mathcal G_r$.
Completing the square in $\mathcal G_r$, we obtain
\begin{align*}
 \|J_rv\|_2^2
 =\sup_{F\in\mathcal K_{r+1}}
       \left\{2\langle J_rv,S_rF\rangle_{2,\mu}
                         -\|S_rF\|_2^2\right\}=\sup_{F\in\mathcal K_{r+1}}
       \left\{2\langle v,F\rangle_{2,\mu}
                         -\|S_rF\|_2^2\right\},
\end{align*}
where the second equality follows from integration by parts,
\eqref{full-adjoint-projection}, and the identity
$\nabla J_rv=v$ given by
\eqref{full-J-L-definition} and
\eqref{full-primitive-properties}.
We next apply \eqref{full-hodge-bound} on
$\mathcal K_{r+1}$ and enlarge the supremum to obtain
\[
 \|J_rv\|_2^2
 \le
 \sup_{\substack{F\in\mathcal C_{r+1}\\
                 F\in W^{1,2}(\mu;\mathsf E_{r+1})}}
       \left\{2\langle v,F\rangle_{2,\mu}
              -\|\nabla F\|_2^2-a\|F\|_2^2\right\}.
\]

By \eqref{full-compatible-split}, every field in this
supremum has the unique decomposition
\[
 F=T+F_0,\qquad
 T=\mathbb E_\mu F\in\mathsf E_{r+1},\qquad
 F_0\in\operatorname{Dom}(D_{r+1}).
\]
Since $\mathbb E_\mu v=0$ and $\nabla T=0$, the constant
tensor $T$ contributes only $-a\|T\|_{\HS}^2$,
which is maximized at $T=0$.
Applying \eqref{full-inverse-variational} with
$B=H_{r+1}+aI$ on $\mathcal G_{r+1}$ therefore gives
\[
 \|J_rv\|_2^2
 \le\langle v,(H_{r+1}+aI)^{-1}v\rangle_{2,\mu}.
\]
Since $H_{r+1}^{-1}$ is nonnegative,
$I+aH_{r+1}^{-1}\succeq I$ has a bounded inverse.
On $\operatorname{Dom}(H_{r+1})$, we have
\[
 H_{r+1}+aI
 =(I+aH_{r+1}^{-1})H_{r+1}.
\]
Let $g_a(t)=t/(1+at)$ for $t\ge0$. For a bounded nonnegative self-adjoint operator $T$, we have
$g_a(T)=T(I+aT)^{-1}$. Taking inverses of the above display, we hence obtain
\[
 (H_{r+1}+aI)^{-1}
 =H_{r+1}^{-1}(I+aH_{r+1}^{-1})^{-1}
 =g_a(H_{r+1}^{-1}).
\]
We now substitute \eqref{full-inverse-identities} at
index $r+1$. Since the preceding bound holds for every
$v\in\mathcal G_{r+1}$, we conclude that
\begin{equation}\label{full-integration-comparison}
 J_r^*J_r
 \preceq
 g_a\bigl(J_{r+1}J_{r+1}^*
              +L_{r+1}L_{r+1}^*\bigr).
\end{equation}
Both sides are bounded positive operators on
$\mathcal G_{r+1}$, and the inequality is understood
with respect to $\langle\cdot,\cdot\rangle_{2,\mu}$. 

We use \eqref{full-hodge-bound} only to bound
$\|S_rF\|_2^2$ from below for $F\in\mathcal K_{r+1}$.
Lemma~\ref{full-divergence-range} and
\eqref{full-inverse-variational} then give
\eqref{full-integration-comparison} for every input
$v\in\mathcal G_{r+1}$.

\subsection{Appell Polynomials as Repeated Primitives}

To use the curvature comparison
\eqref{full-integration-comparison} in the final operator
norm argument, we also need bounds on repeated centered
integration starting from constant tensors.
We therefore study the compositions
$J_r\cdots J_{r+d-2}L_{r+d-1}$ for $r\ge0$ and $d\ge1$.
We identify their outputs with contractions of factorially
normalized Appell tensors, allowing us to apply
Theorem~\ref{thm:geometric-coefficients} to bound their norms.

Recall the definitions of Appell tensors $A_d^\mu$, scalar polynomials $P_d^\mu[T]$,
and normalized coefficient norms $c_d(\mu)$ from
Subsection~\ref{sec:notation}. Their centering and derivative
relations are classical
\citep[Section~2.3, Equations~(9)--(10)]{AnshelevichAppell}.
Taking expectation in \eqref{full-Appell-definition} gives
$\mathbb E_\mu A_d^\mu=0$ for $d\ge1$.
Since $\mathbb E_\mu X=0$, the degree-one coefficient is
$A_1^\mu(x)=x$.
For $d\ge1$, $T\in\mathsf E_d$, and $v\in\mathbb R^n$,
differentiating \eqref{full-Appell-definition} in the spatial
direction $v$ and comparing coefficients gives
\begin{equation}\label{full-Appell-derivative}
 \partial_v P_d^\mu[T]
 =dP_{d-1}^\mu[T(v,\cdot,\ldots,\cdot)].
\end{equation}
Since log-concave laws have moments of every order, all
polynomial fields used below and their derivatives belong
to $L^2(\mu)$.

For $r\ge0$, $d\ge1$, and $T\in\mathsf E_{r+d}$,
the contraction defined in \eqref{eq:contraction} makes
$T\mathbin{\lrcorner}A_d^\mu$ an $\mathsf E_r$-valued
polynomial field. We show that
\begin{equation}\label{full-Appell-primitives}
 J_rJ_{r+1}\cdots J_{r+d-2}L_{r+d-1}T
 =\frac{T\mathbin{\lrcorner}A_d^\mu}{d!}.
\end{equation}
We prove \eqref{full-Appell-primitives} by induction on $d$,
with the assertion made for every $r\ge0$ at each degree.
For $d=1$, the product of $J$'s is empty, and the identity
reduces to \eqref{full-linear-primitive} because
$A_1^\mu(x)=x$.

For an ordered tuple $I$ of length $r$, we define
$T_I\in\mathsf E_d$ by $ (T_I)_{j_1\cdots j_d}=T_{I j_1\cdots j_d}.$ For $d\ge2$, applying \eqref{full-Appell-derivative}
to $T_I$ with $v=e_\ell$ gives
\[
 \partial_\ell
 \left(\frac{\langle T_I,A_d^\mu\rangle_{\HS}}{d!}\right)
 =\frac1{(d-1)!}
   \sum_{j_1,\ldots,j_{d-1}=1}^n
   T_{I\ell j_1\cdots j_{d-1}}
   (A_{d-1}^\mu)_{j_1\cdots j_{d-1}}.
\]
The right-hand side is symmetric in the free indices
$I$ and $\ell$. Thus the gradient of
$T\mathbin{\lrcorner}A_d^\mu/d!$ is fully symmetric,
so this field is compatible.
The field and its gradient both have zero mean, using
$\mathbb E_\mu A_d^\mu=0$ and
$\mathbb E_\mu A_{d-1}^\mu=0$.
The induction hypothesis at degree $d-1$ and tensor order
$r+1$ identifies the gradient as
\[
 \nabla\left(\frac{T\mathbin{\lrcorner}A_d^\mu}{d!}\right)
 =J_{r+1}\cdots J_{r+d-2}L_{r+d-1}T.
\]
Uniqueness of the mean-zero primitive in
Lemma~\ref{full-primitive} then proves
\eqref{full-Appell-primitives}.

Applying the definition of $c_d(\mu)$ to each $T_I$
and summing over the ordered tuples $I$, we obtain
\begin{align*}
 \left\|\frac{T\mathbin{\lrcorner}A_d^\mu}{d!}\right\|_2^2=\sum_I
   \left\|\frac{\langle T_I,A_d^\mu\rangle_{\HS}}{d!}
   \right\|_2^2\le c_d(\mu)^2\sum_I\|T_I\|_{\HS}^2
 =c_d(\mu)^2\|T\|_{\HS}^2.
\end{align*}
The last equality follows from the ordered-index convention
for the Hilbert--Schmidt norm.
Combining this estimate with \eqref{full-Appell-primitives}
gives
\begin{equation}\label{full-Appell-block-bound}
 \|J_r\cdots J_{r+d-2}L_{r+d-1}\|
       _{\mathsf E_{r+d}\to\mathcal G_r}
 \le c_d(\mu).
\end{equation}
Since $c_d(\mu)\le c_d$, Theorem~\ref{thm:geometric-coefficients}
and \eqref{full-Appell-block-bound} give geometric bounds
on these compositions, uniformly in the output tensor order $r$
and the dimension $n$.

\subsection{Integration Across Tensor Orders}

The curvature comparison \eqref{full-integration-comparison}
and the Appell bounds \eqref{full-Appell-block-bound} apply
at each tensor order. We now combine them into estimates
for operators $J$ and $L$ on the Hilbert direct sums.
We also show that $C_P(\mu)\le1+\|J\|^2$, reducing the
Poincar\'e bound to an estimate for $\|J\|$.

We combine the operators $J_r$ and $L_r$ using the Hilbert
direct sums
\begin{equation}\label{full-direct-sum-spaces}
 \mathcal H=\bigoplus_{r\ge0}\mathcal G_r,\qquad
 \mathcal E=\bigoplus_{r\ge0}\mathsf E_{r+1}.
\end{equation}
These consist of sequences $F=(F_r)_{r\ge0}$ with
$F_r\in\mathcal G_r$ and $U=(U_r)_{r\ge0}$ with
$U_r\in\mathsf E_{r+1}$, respectively, satisfying
\[
 \|F\|_{\mathcal H}^2
   =\sum_{r\ge0}\|F_r\|_{2,\mu}^2<\infty,\qquad
 \|U\|_{\mathcal E}^2
   =\sum_{r\ge0}\|U_r\|_{\HS}^2<\infty.
\]
On these spaces, we define
$J:\mathcal H\to\mathcal H$ and
$L:\mathcal E\to\mathcal H$ by
\begin{equation}\label{full-direct-sum-definition}
 (JF)_r=J_rF_{r+1},\qquad
 (LU)_r=L_rU_r,\qquad r\ge0.
\end{equation}
Thus $J$ lowers the tensor order by one and vanishes on
$\mathcal G_0$. Orthogonality of the direct-sum components
and \eqref{full-rankwise-bounds} give
\[
 \|J\|_{\mathcal H\to\mathcal H}^2
 =\sup_{r\ge0}\|J_r\|^2\le C_P(\mu),\qquad
 \|L\|_{\mathcal E\to\mathcal H}
 =\sup_{r\ge0}\|L_r\|=1.
\]
In particular, both operators are bounded.
Taking adjoints in \eqref{full-direct-sum-definition},
we obtain
\[
 (J^*F)_0=0,\qquad
 (J^*F)_{r+1}=J_r^*F_r,\qquad
 (L^*F)_r=L_r^*F_r,
 \qquad F\in\mathcal H.
\]

On $\mathcal G_r$, the operators $JJ^*$ and $LL^*$ act as
$J_rJ_r^*$ and $L_rL_r^*$, respectively.
Since $g_a(T)=T(I+aT)^{-1}$, the operator
$g_a(JJ^*+LL^*)$ also acts separately on each summand,
as $g_a(J_rJ_r^*+L_rL_r^*)$ on $\mathcal G_r$.
For $r\ge1$, $J^*J$ acts on $\mathcal G_r$ as
$J_{r-1}^*J_{r-1}$. We therefore apply
\eqref{full-integration-comparison} with index $r-1$.
On $\mathcal G_0$, $J^*J$ is zero and
$g_a(JJ^*+LL^*)$ is nonnegative.
Summing the resulting inequalities over the orthogonal
summands gives
\begin{equation}\label{full-direct-sum-comparison}
 J^*J\preceq g_a(JJ^*+LL^*).
\end{equation}

Next, for $d\ge1$, \eqref{full-direct-sum-definition} gives
\[
 (J^{d-1}LU)_r
 =J_r\cdots J_{r+d-2}L_{r+d-1}U_{r+d-1}.
\]
The source spaces $\mathsf E_{r+d}$ are pairwise orthogonal
in $\mathcal E$, and the corresponding target spaces
$\mathcal G_r$ are pairwise orthogonal in $\mathcal H$.
Consequently,
\begin{equation}\label{full-direct-sum-observations}
 \begin{aligned}
 \|J^{d-1}L\|_{\mathcal E\to\mathcal H}
 &=\sup_{r\ge0}
       \|J_r\cdots J_{r+d-2}L_{r+d-1}\|\\
 &=c_d(\mu),\qquad d\ge1.
 \end{aligned}
\end{equation}
For the upper bound, we use \eqref{full-Appell-block-bound}.
For the reverse bound, we take $r=0$ in
\eqref{full-Appell-primitives}. This gives the map
$T\mapsto P_d^\mu[T]/d!$ from $\mathsf E_d$ to
$\mathcal G_0$, whose norm is exactly $c_d(\mu)$ by
its definition in Subsection~\ref{sec:notation}.

Combining \eqref{full-direct-sum-observations} with
Theorem~\ref{thm:geometric-coefficients}, we obtain
\begin{align}\label{eq:jlbound}
 \|JL\|^2\le2,\qquad
 \|J^mL\|^2
 \le\frac{121}{4}
       \left(\frac{223}{10}\right)^{m-2}
 \qquad(m\ge2).
\end{align}
Since $(J^mL)^*=L^*(J^*)^m$ and an operator and its
adjoint have the same norm,
\eqref{full-direct-sum-observations} also gives
\[
 \|L^*(J^*)^mF\|_{\mathcal E}
 \le c_{m+1}(\mu)\|F\|_{\mathcal H},
 \qquad m\ge0,\quad F\in\mathcal H.
\]
This is the coefficient estimate used for the adjoint
iterates $(J^*)^mF$ in the final operator argument.

Finally, using the scalar inverse-gradient identity
$\|\widetilde J_0\|^2=C_P(\mu)$ and
\eqref{full-inverse-identities}, we obtain
\begin{equation}\label{full-Poincare-from-integration}
 C_P(\mu)=\|H_0^{-1}\|
 =\|J_0J_0^*+L_0L_0^*\|\le\|J_0\|^2+\|L_0\|^2
 \le1+\|J\|^2.
\end{equation}
Here $\|L_0\|=1$ by \eqref{full-rankwise-bounds}, and
$\|J_0\|\le\|J\|$ follows from
\eqref{full-direct-sum-definition}.


\section{Proof of Theorem~\ref{thm:main}}
\label{sec:completion}

We now combine the curvature comparison
\eqref{full-direct-sum-comparison} with the bounds on
$\|J^mL\|$ supplied by \eqref{full-direct-sum-observations}
and Theorem~\ref{thm:geometric-coefficients}.
These estimates yield $\|J\|^2\le24$ for regular isotropic
laws. Equation~\eqref{full-Poincare-from-integration}
then gives $C_P(\mu)\le25$, and approximation extends
the bound to all isotropic log-concave laws.

\begin{proof}[Proof of Theorem~\ref{thm:main}]
We first assume that $\mu$ is regular and isotropic, with
$D^2V\succeq aI$ for some $a>0$.
We use the bounded operators $J$ and $L$ defined in
\eqref{full-direct-sum-definition}, which satisfy
\eqref{full-direct-sum-comparison}.
Their preliminary boundedness follows from
\eqref{full-initial-poincare} and
\eqref{full-rankwise-bounds}.

We now prove that $\|J\|^2\le24$.
Suppose instead that $\|J\|^2>24$.
Since $\|J^*\|=\|J\|$, we may choose $f\in\mathcal H$
with $\|f\|_{\mathcal H}=1$ and $\|J^*f\|_{\mathcal H}^2>24$.
For $m\ge0$, we define
\[
 v_m=(J^*)^mf,\qquad
 u_m=\|v_m\|_{\mathcal H}^2,\qquad
 p_m=\|L^*v_m\|_{\mathcal E}^2.
\]
Thus $u_0=1$ and $u_1>24$.
For $m\ge1$, we write $r_m=u_m/u_{m-1}$ whenever
$u_{m-1}>0$. We will prove that $u_m>0$ for every $m$.

To use \eqref{full-direct-sum-comparison}, we first derive
a scalar inequality from \eqref{full-inverse-variational}.
We set $T=JJ^*+LL^*$.
For a unit vector $w\in\mathcal H$, we apply
\eqref{full-inverse-variational} with $B=I+aT$ and
restrict the test vectors to $\lambda w$, obtaining
 \begin{align}
 \langle w,(I+aT)^{-1}w\rangle_{\mathcal H}\ge\sup_{\lambda\in\mathbb R}
   \left\{2\lambda
     -\lambda^2\bigl(1+a\langle w,Tw\rangle_{\mathcal H}\bigr)
   \right\}=\frac{1}{1+a\langle w,Tw\rangle_{\mathcal H}}.
 \end{align}
Using $g_a(T)=a^{-1}[I-(I+aT)^{-1}]$, we therefore have
\begin{equation}\label{eq:curvature-vector-bound}
 \langle w,g_a(T)w\rangle_{\mathcal H}=\frac1a\left[
   1-\langle w,(I+aT)^{-1}w\rangle_{\mathcal H}
   \right]\le g_a\bigl(\langle w,Tw\rangle_{\mathcal H}\bigr).
\end{equation}

If $m\ge1$ and $u_m>0$, then $u_{m-1}>0$.
The adjoint relation gives
\[
 u_m=\langle v_{m-1},Jv_m\rangle_{\mathcal H}.
\]
Applying Cauchy--Schwarz,
\eqref{full-direct-sum-comparison}, and
\eqref{eq:curvature-vector-bound} with
$w=v_m/\sqrt{u_m}$, we obtain
\[
 r_m
 \le\frac{\|Jv_m\|_{\mathcal H}^2}{u_m}
 \le g_a\left(\frac{u_{m+1}+p_m}{u_m}\right).
\]
Here
$\langle v_m,Tv_m\rangle_{\mathcal H}=u_{m+1}+p_m$
by the definitions of $T$, $u_m$, and $p_m$.
In particular, $ar_m<1$.
Since $g_a$ is increasing and its inverse is
$g_a^{-1}(t)=t/(1-at)$ for $0\le t<a^{-1}$, we obtain
\begin{equation}\label{eq:curvature-ratio-recurrence}
 r_{m+1}
 \ge\frac{r_m}{1-ar_m}-\frac{p_m}{u_m}
 \ge r_m+ar_m^2-\frac{p_m}{u_m}.
\end{equation}
Moreover, as $p_m=\|(J^mL)^*f\|_{\mathcal E}^2,$ \eqref{eq:jlbound} gives
\[
 p_1\le2,\qquad
 p_m\le BK^{m-2}\qquad(m\ge2).
\]

Since $r_1=u_1$, applying
\eqref{eq:curvature-ratio-recurrence} at $m=1$
and omitting its nonnegative curvature term gives
\[
 r_2
 \ge u_1-\frac2{u_1}
 >24-\frac2{24}
 =K+\frac{97}{60}>K.
\]
The strict inequality uses $u_1>24$ and the monotonicity
of $x-2/x$ for $x>0$. In particular, $u_2>0$.

For $m\ge2$ with $u_m>0$, we define
\[
 \varepsilon_m=\frac{BK^{m-2}}{u_m},\qquad
 E_m=(r_m-K)^2-\frac{2BK^{m-2}}{u_{m-1}}
     =(r_m-K)^2-2r_m\varepsilon_m.
\]
The preceding estimates imply
\[
 E_2
 =(r_2-K)^2-\frac{2B}{u_1}
 >\left(\frac{97}{60}\right)^2-\frac{2B}{24}
 =\frac{167}{1800}>0.
\]

We prove by induction that
$u_m>0$, $r_m>K$, and $E_m\ge E_2$ for every $m\ge2$.
The case $m=2$ has been established.
Suppose these properties hold at $m$.
The definition of $E_m$ gives
\[
 \varepsilon_m
 <\frac{(r_m-K)^2}{2r_m}
 <r_m-K.
\]
Using \eqref{eq:curvature-ratio-recurrence} and
$p_m/u_m\le\varepsilon_m$, we obtain
\[
 r_{m+1}\ge r_m-\varepsilon_m>K,
\]
and hence $u_{m+1}>0$.
Since $BK^{m-1}/u_m=K\varepsilon_m$, we also have
\begin{align*}
 E_{m+1}=(r_{m+1}-K)^2-2K\varepsilon_m\ge(r_m-K-\varepsilon_m)^2-2K\varepsilon_m=E_m+\varepsilon_m^2
 \ge E_2.
\end{align*}
This completes the induction.

Since $E_m\le(r_m-K)^2$, we may set $t=K+\sqrt{E_2}>K$ and conclude that $r_m\ge t$ for every $m\ge2$.
Consequently,
\[
 u_m=u_1\prod_{j=2}^m r_j
 \ge u_1t^{m-1},
\]
and therefore
\[
 \varepsilon_m
 =\frac{BK^{m-2}}{u_m}
 \le\frac{B}{u_1t}
       \left(\frac Kt\right)^{m-2}
 \longrightarrow0.
\]

We now retain the positive curvature term in
\eqref{eq:curvature-ratio-recurrence}.
Since $r_m\ge t$ and $p_m/u_m\le\varepsilon_m$,
for all sufficiently large $m$ we have
\[
 r_{m+1}
 \ge r_m+at^2-\varepsilon_m
 \ge r_m+\frac{at^2}{2}.
\]
Thus the ratios $r_m$ become unbounded.
This contradicts boundedness of $J$, because
\[
 r_m
 =\frac{\|J^*v_{m-1}\|_{\mathcal H}^2}
        {\|v_{m-1}\|_{\mathcal H}^2}
 \le\|J\|^2.
\]
Hence, we can conclude that $\|J\|^2\le24$.
Applying \eqref{full-Poincare-from-integration}, we obtain $C_P(\mu)\le1+\|J\|^2\le25.$ This bound holds for every regular isotropic law and is
independent of the auxiliary curvature constant $a$.

For a general isotropic log-concave law $\mu$,
Lemma~\ref{lem:static-regular-approximation} supplies
regular isotropic laws $\mu_j$ converging weakly to $\mu$.
Each $\mu_j$ satisfies $C_P(\mu_j)\le25$.
For a fixed $\varphi\in C_c^\infty(\mathbb R^n)$,
the functions $\varphi$, $\varphi^2$, and
$|\nabla\varphi|^2$ are bounded and continuous.
Passing to the limit in the Poincar\'e inequalities for
$\mu_j$ gives
\[
 \operatorname{Var}_\mu\varphi
 \le25\int|\nabla\varphi|^2\,d\mu.
\]
Lemma~\ref{finite-energy-extension} extends this inequality
to every locally Lipschitz function with finite gradient
energy and proves its square integrability.
Hence $C_P(\mu)\le25$, as claimed.
\end{proof}

\appendix

\section{Compactness of Isotropic Log-concave Laws}
\label{app:static-compactness}

In this section, we record the standard compactness facts used in
Section~\ref{sec:static-coefficients}. Compactness is stated by~\cite[Section~2, paragraph following Lemma~2.2]{KlartagCLT},
and continuity of moments is given by~\citet[Corollary~6]{MeckesMeckes}.
The uniform exponential-moment bound is an immediate consequence of
Borell's lemma in the form stated in
\cite[Lemma~2.1(i)]{KlartagCLT}.

\begin{lemma}[Compactness and Moment Continuity]
\label{lem:static-compactness}
Fix $n\ge1$. Then $\mathsf C_n$ is compact for weak convergence, and for every polynomial $p:\R^n\to\R$, the map $\mu\longmapsto\int p\,d\mu$ is continuous on $\mathsf C_n$. In particular, cumulants of every fixed
order are continuous on this class. Moreover, there exist constants
$c_n>0$ and $C_n<\infty$ such that
\[
 \sup_{\mu\in\mathsf C_n}\int e^{c_n|x|}\,d\mu(x)\le C_n.
\]
The constants may depend on $n$.
\end{lemma}

\section{Primitives and Operator Domains}\label{sec:full-domains}

In this section, we provide the primitive and operator-domain statements used in
Sections~\ref{sec:full-hodge} and~\ref{full-integration-section}.
Throughout this appendix $V\in C^\infty$, $aI\preceq D^2V\preceq MI$,
and $0<a\le M<\infty$. In particular, the scalar Poincar\'e constant is
finite. This follows independently of the compatible-tensor estimate
from the Brascamp--Lieb inequality
\cite[Equation~(1.3)]{carlen-cordero-lieb}:
\begin{equation}\label{full-initial-poincare}
 \operatorname{Var}_\mu f
 \le\int\langle\nabla f,(D^2V)^{-1}\nabla f\rangle\,d\mu
 \le a^{-1}\int|\nabla f|^2\,d\mu.
\end{equation}
The inequality extends from smooth functions to $W^{1,2}(\mu)$ by
cutoff and mollification: on each compact set the smooth positive
density and its reciprocal are bounded, and the cutoff derivative is
bounded by a constant times the inverse cutoff radius. We write
$C_P=C_P(\mu)$ for the best scalar constant, so $C_P\le a^{-1}$.

\subsection{Centered Primitives and Approximation by Smooth Potentials}

The existence of local Sobolev primitives is the standard Poincar\'e
lemma; see, for example,~\citet[Proposition~4.1(i)]{CostabelMcIntosh}.
For the cutoff and mollification method in weighted spaces, see, for example, \citet[Lemma~4.1]{Johnsen}. The following lemma adapts these
arguments to global centered primitives in the symmetric-tensor spaces,
using the preliminary scalar Poincar\'e inequality.
We provide the details for completeness, including the denseness of
compactly supported potential fields.

\begin{lemma}\label{full-primitive}
For every $r\ge0$ and $F\in\mathcal C_{r+1}$, there is a unique
$U\in\mathcal G_r\cap W^{1,2}(\mu;\mathsf E_r)$ such that
\[
 \nabla U=F,\qquad \|U\|_2\le\sqrt{C_P}\,\|F\|_2.
\]
Consequently $D_r$ is a closed, densely defined bijection and its inverse
$\widetilde J_r$ is bounded. For $s\ge1$, the test space $\mathcal K_s$ of
Section~\ref{sec:notation} is dense in $\mathcal C_s$.
\end{lemma}

\begin{proof}
First consider a locally square-integrable vector field $f$ satisfying
$\partial_if_j=\partial_jf_i$ in distributions. We convolve $f$ with a compactly supported smooth
approximate identity to obtain smooth vector fields
$f_\varepsilon$ satisfying
$\partial_i f_{\varepsilon,j}=\partial_j f_{\varepsilon,i}$
for all $i,j$. The functions
\[
 h_\varepsilon(x)=\int_0^1\langle f_\varepsilon(tx),x\rangle\,dt
\]
satisfy $\nabla h_\varepsilon=f_\varepsilon$. Indeed,
\[
 \partial_jh_\varepsilon(x)
 =\int_0^1\left(f_{\varepsilon,j}(tx)
        +t\sum_i x_i\partial_jf_{\varepsilon,i}(tx)\right)dt
 =\int_0^1\frac d{dt}\bigl(t f_{\varepsilon,j}(tx)\bigr)dt
 =f_{\varepsilon,j}(x).
\]
Now, we subtract the Lebesgue mean of $h_\varepsilon$ on the unit ball. For every
$R\ge1$, the usual Poincar\'e inequality on $B_R$ then gives
\[
 \|h_\varepsilon-h_\delta\|_{L^2(B_R)}
 \le C_R\|f_\varepsilon-f_\delta\|_{L^2(B_R)}.
\]
To verify the use of the mean on $B_1$ rather than $B_R$, put
$h=h_\varepsilon-h_\delta$ and let $h_R$ be its mean on $B_R$.
Since the mean of $h$ on $B_1$ is zero,
$|h_R|\le |B_1|^{-1/2}\|h-h_R\|_{L^2(B_R)}$. We apply the usual Poincar\'e inequality on $B_R$
to the mean-zero function $h-h_R$ and combine it with
the preceding bound on $h_R$ to obtain the displayed estimate.
Thus $h_\varepsilon$ converges in $W^{1,2}(B_R)$ for every $R$ to a
locally Sobolev primitive $h$ of $f$.

For every square-integrable real random variable $Z$ with median zero,
\begin{equation}\label{eq:median-second-moment}
 \mathbb E Z^2\le2\operatorname{Var}(Z).
\end{equation}
Indeed, if $\mathbb EZ\ge0$, Cauchy--Schwarz gives $(\mathbb EZ)^2\le\bigl(\mathbb E[Z\mathbf1_{\{Z>0\}}]\bigr)^2\le\mathbb P(Z>0)\,\mathbb EZ^2\le\tfrac12\mathbb EZ^2.$ For $\mathbb EZ<0$, we apply the same argument to $-Z$.
Using
$\operatorname{Var}(Z)=\mathbb EZ^2-(\mathbb EZ)^2$
then proves \eqref{eq:median-second-moment}.

Suppose that $f\in L^2(\mu;\R^n)$.
We subtract a median from $h$ and set
$h_N=\max\{-N,\min\{h,N\}\}$.
Each $h_N$ has median zero and belongs to $W^{1,2}(\mu)$,
since it is bounded and its weak gradient satisfies
$|\nabla h_N|\le|f|$ almost everywhere.
Applying \eqref{eq:median-second-moment} and the scalar
Poincar\'e inequality gives
\[
 \|h_N\|_2^2
 \le2\operatorname{Var}_\mu(h_N)
 \le2C_P\|\nabla h_N\|_2^2
 \le2C_P\|f\|_2^2.
\]
Fatou's lemma gives $h\in L^2(\mu)$.
Subtracting its mean and applying the scalar Poincar\'e
inequality once more yields
$\|h-\mathbb E_\mu h\|_2^2\le C_P\|f\|_2^2$.

We now apply the scalar construction to each vector field
$(F_{iI})_{i=1}^n$, where $I$ is an ordered tuple of length $r$,
and choose the primitive $U_I$ with $\mathbb E_\mu U_I=0$.
Permuting $I$ leaves the gradient of $U_I$ unchanged by
symmetry of $F$. Uniqueness of the mean-zero scalar primitive
therefore makes $U$ symmetric.
For $r\ge1$ and an ordered tuple $I'$ of length $r-1$,
\[
 \partial_iU_{jI'}
 =F_{ijI'}
 =F_{jiI'}
 =\partial_jU_{iI'},
\]
so $U$ is compatible.
Summing the scalar estimates over the ordered tuples gives
\[
 \|U\|_2^2
 =\sum_{|I|=r}\|U_I\|_2^2
 \le C_P\sum_{|I|=r}\sum_{i=1}^n\|F_{iI}\|_2^2
 =C_P\|F\|_2^2.
\]
Any two mean-zero primitives differ by a constant tensor.
Taking expectations shows that this constant is zero.

For $s\ge1$, repeating the construction gives
$F=\nabla^s\phi$ with $\phi\in W^{s,2}(\mu)$ and
$\mathbb E_\mu\nabla^j\phi=0$ for $0\le j<s$.
We choose smooth cutoffs $\chi_R$ with $0\le\chi_R\le1$,
equal to one on $B_R$, supported in $B_{2R}$, and satisfying
\[
 \|\nabla^j\chi_R(x)\|_{\HS}\le C_jR^{-j},
 \qquad 1\le j\le s,
\]
with constants independent of $R$.
For $0\le k\le s$, the product rule gives
\[
 \|\nabla^k(\chi_R\phi-\phi)\|_2
 \le\|(\chi_R-1)\nabla^k\phi\|_2
   +\sum_{j=1}^k\binom{k}{j}C_jR^{-j}
                     \|\nabla^{k-j}\phi\|_2.
\]
The first term tends to zero by dominated convergence.
Each term in the sum tends to zero because
$\phi\in W^{s,2}(\mu)$. Hence
$\chi_R\phi\to\phi$ in $W^{s,2}(\mu)$.

For fixed $R$, we mollify $\chi_R\phi$.
The mollifications have support in a fixed enlarged compact
set, where the weighted and unweighted Sobolev norms are
equivalent. We may therefore choose
$\psi_R\in C_c^\infty(\mathbb R^n)$ with
\[
 \|\psi_R-\chi_R\phi\|_{W^{s,2}(\mu)}<R^{-1}.
\]
Then $\psi_R\to\phi$ in $W^{s,2}(\mu)$, so
$\nabla^s\psi_R\to F$ in $L^2(\mu)$.
Since $\nabla^s\psi_R\in\mathcal K_s$ by \eqref{eq:kmdef},
this proves denseness of $\mathcal K_s$ in $\mathcal C_s$.
For $s=0$, the same cutoff and mollification argument gives
denseness of $C_c^\infty(\mathbb R^n)$ in $L^2(\mu)$.

The centered fields
$\nabla^r\psi-\mathbb E_\mu\nabla^r\psi$, with
$\psi\in C_c^\infty(\mathbb R^n)$, belong to
$\operatorname{Dom}(D_r)$ by \eqref{eq:compatible-gradient}.
The preceding approximation result shows that they are
dense in $\mathcal G_r$, with $\nabla^0\psi=\psi$ when $r=0$.
If $U_j\to U$ in $\mathcal G_r$ and $D_rU_j\to F$ in
$\mathcal C_{r+1}$, local distributional convergence gives
$\nabla U=F$. Thus $U\in\operatorname{Dom}(D_r)$ and
$D_rU=F$, proving that $D_r$ is closed. Finally, existence and uniqueness of the mean-zero primitive
show that $D_r$ is bijective. Taking square roots of the
inequality obtained above by summing the scalar estimates gives $ \|U\|_2\le\sqrt{C_P(\mu)}\,\|F\|_2.$ Since $U=D_r^{-1}F$, this also proves that the inverse is bounded.
\end{proof}

\subsection{The Constrained Complex}

Here, we prove the adjoint and approximation statements used in
Section~\ref{sec:full-hodge}. The approximation extends
\eqref{full-hook-lower-bound} from smooth compactly supported
fields to every field in $\operatorname{Dom}(\ell)$.
The projection identity below supplies the variational formula
used to bound the orthogonal correction in
\eqref{full-hodge-bound}.
Our approximation argument adapts
\citet[Lemmas~4.1 and~4.3]{Johnsen} to the fixed subspaces
$\mathcal F_p$. The projection formula uses the relation
between adjoint ranges and kernels underlying
\cite[Theorem~3.1]{Johnsen}, without assuming that the range
is closed.

\begin{lemma}\label{full-complex-domains}
Fix $m\ge1$, and use the subspaces $\mathcal F_p$,
projections $P_p$, and operators $d_p$ defined in
Section~\ref{sec:full-hodge}.
Each $d_p$ is closed and densely defined, and
$\operatorname{ran}d_p\subset\ker d_{p+1}$.
Its Hilbert adjoint satisfies
\[
 \begin{aligned}
 \operatorname{Dom}(d_p^*)
   &=\{\omega\in L^2(\mu;\mathcal F_{p+1}):
          P_pd^*\omega\in L^2(\mu;\mathcal F_p)\},\\
 d_p^*\omega&=P_pd^*\omega.
 \end{aligned}
\]
Here $P_pd^*\omega$ is interpreted distributionally. For the quadratic form $\ell$ defined in
Section~\ref{sec:full-hodge}, every
$\omega\in\operatorname{Dom}(\ell)$ admits a sequence
$\omega_j\in C_c^\infty(\R^n;\mathcal F_2)$ such that $\|\omega_j-\omega\|_2^2
 +\|d_2\omega_j-d_2\omega\|_2^2
 +\|d_1^*\omega_j-d_1^*\omega\|_2^2
 \longrightarrow0.$ For every $u\in\operatorname{Dom}(d_1)$,
\begin{equation}\label{full-projection-variational}
 \|(I-P_{\ker d_1})u\|_2^2
 =\sup_{\omega\in\operatorname{Dom}(\ell)}
   \left\{2\langle d_1u,\omega\rangle_{2,\mu}
                         -\ell(\omega)\right\}.
\end{equation}
\end{lemma}

\begin{proof}
Each $d_p$ is closed on its distributional domain.
Indeed, if $u_j\to u$ and $d_pu_j\to v$ in their respective
$L^2(\mu)$ spaces, local distributional convergence gives
$du=v$. Thus $u\in\operatorname{Dom}(d_p)$ and $d_pu=v$.
The domain is dense because it contains all smooth compactly
supported fields with values in $\mathcal F_p$.
For $u\in\operatorname{Dom}(d_p)$, the identity $d^2u=0$
gives $d_pu\in\operatorname{Dom}(d_{p+1})$ and
$d_{p+1}d_pu=0$.

We next prove approximation in $L^2(\mu)$ while controlling
the expressions $d_p$ and $P_pd^*$.
We choose smooth compactly supported scalar cutoffs $\chi_R$
with $0\le\chi_R\le1$, $\chi_R\to1$ pointwise, and
$|\nabla\chi_R|\le C/R$.
Multiplication by $\chi_R$ preserves the fixed subspaces.
The product rule gives
\[
 \begin{aligned}
 \|d_p(\chi_Ru)-\chi_Rd_pu\|_2
   &\le CR^{-1}\|u\|_2,\\
 \|P_pd^*(\chi_R\omega)-\chi_RP_pd^*\omega\|_2
   &\le CR^{-1}\|\omega\|_2.
 \end{aligned}
\]
These bounds and dominated convergence show that the cutoffs
approximate both the fields and their differential expressions
in $L^2(\mu)$.

For a field with fixed compact support, we convolve with a
smooth mollifier $\rho_\varepsilon$ supported in $B_\varepsilon$.
Convolution preserves the fixed subspaces and commutes with
the constant matrices $P_p$ and the constant coefficient
derivative terms.
The multiplication term in $P_pd^*$ is the smooth matrix
$\mathsf B(x)=\sum_iV_i(x)P_p\iota_i$.
Its convolution commutator is
\[
 \mathsf B(x)(\omega*\rho_\varepsilon)(x)
 -(\mathsf B\omega)*\rho_\varepsilon(x)
 =\int\rho_\varepsilon(y)
       [\mathsf B(x)-\mathsf B(x-y)]\omega(x-y)\,dy.
\]
Uniform continuity of $\mathsf B$ on an enlarged compact set
and Young's convolution inequality show that this commutator
tends to zero in $L^2(dx)$.
Weighted and unweighted norms are equivalent there, so the
convergence also holds in $L^2(\mu)$.
Together with convergence of mollification, this proves
approximation of each field and its differential expression.
Using the same cutoffs and mollifiers for $d_2$ and $P_1d^*$
gives simultaneous approximation for both expressions.

If $\omega\in\operatorname{Dom}(d_p^*)$, testing against
smooth compactly supported $\mathcal F_p$-valued fields
identifies $P_pd^*\omega$ with the Hilbert adjoint image.
Conversely, suppose
$\omega\in L^2(\mu;\mathcal F_{p+1})$ and
$P_pd^*\omega\in L^2(\mu;\mathcal F_p)$ distributionally.
Integration by parts gives $\langle d_pu,\omega\rangle_{2,\mu}
   =\langle u,P_pd^*\omega\rangle_{2,\mu}$ for smooth compactly supported $\mathcal F_p$-valued fields $u$.
The approximation just proved extends this equality to every
$u\in\operatorname{Dom}(d_p)$.
This proves the asserted adjoint domain and formula.
The simultaneous approximation for $d_2$ and $P_1d^*$ is
therefore the stated approximation for $d_2$ and $d_1^*$.

To prove \eqref{full-projection-variational}, we set
$A=d_1$, $B=d_2$, and let $P$ be the orthogonal projection
onto $\ker B$ in $L^2(\mu;\mathcal F_2)$.
Since $\operatorname{ran}A\subset\ker B$, for
$\omega\in\operatorname{Dom}(A^*)$ and
$v\in\operatorname{Dom}(A)$ we have $\langle Av,P\omega\rangle_{2,\mu}
 =\langle Av,\omega\rangle_{2,\mu}
 =\langle v,A^*\omega\rangle_{2,\mu}.$ The definition of the Hilbert adjoint therefore gives
\[
 P\omega\in\operatorname{Dom}(A^*)\cap\ker B,
 \qquad A^*P\omega=A^*\omega.
\]
Write $u_\perp=(I-P_{\ker A})u$.
For $\omega\in\operatorname{Dom}(\ell)$, we set
$z=A^*\omega$. Then $z\perp\ker A$, and completing the
square gives
 \begin{align}
 2\langle Au,\omega\rangle_{2,\mu}-\ell(\omega)=2\langle u_\perp,z\rangle_{2,\mu}
      -\|z\|_2^2-\|B\omega\|_2^2=\|u_\perp\|_2^2-\|u_\perp-z\|_2^2-\|B\omega\|_2^2\le\|u_\perp\|_2^2.
 \end{align}
This proves one inequality in
\eqref{full-projection-variational}. 

For the reverse inequality, we use
$\overline{\operatorname{ran}A^*}=(\ker A)^\perp$.
Indeed, a vector orthogonal to $\operatorname{ran}A^*$
belongs to $\ker A^{**}=\ker A$, and the reverse inclusion
follows from the adjoint identity.
We can therefore choose
$\eta_j\in\operatorname{Dom}(A^*)$ such that
$A^*\eta_j\to u_\perp$.
Setting $\omega_j=P\eta_j$ gives
$\omega_j\in\operatorname{Dom}(\ell)$,
$B\omega_j=0$, and $A^*\omega_j=A^*\eta_j$.
Consequently,
 \begin{align}
 2\langle Au,\omega_j\rangle_{2,\mu}-\ell(\omega_j)=2\langle u_\perp,A^*\eta_j\rangle_{2,\mu}
      -\|A^*\eta_j\|_2^2\longrightarrow\|u_\perp\|_2^2.
 \end{align}
This proves \eqref{full-projection-variational}.
\end{proof}
\section{Approximation and Test Functions}
\label{app:static-approximation}

In this section, the standard approximation properties used in our proof are stated. For approximation by regular log-concave laws, see, for example, \citet[Lemma~2.1 and its proof in Section~5]{KlartagLogarithmic}. For the sake of completeness, we give a proof of the isotropic approximation needed here
using Moreau regularization and compactly supported
mollification of the potential, followed by centering
and whitening.

\begin{lemma}\label{lem:static-regular-approximation}
Let $\mu$ be an isotropic log-concave probability law on $\R^n$.
There exist isotropic log-concave probability laws
$\mu_\varepsilon$, $0<\varepsilon\le1$, each regular in the
sense of Definition~\ref{def:regular-law}, such that
$\mu_\varepsilon$ converges weakly to $\mu$ as
$\varepsilon\downarrow0$.
\end{lemma}

\begin{proof}
Since $\operatorname{Cov}(\mu)=I$, the law $\mu$ is
full-dimensional. Recall that $ d\mu=Z^{-1}e^{-V(x)}\,dx,
$ where $V:\R^n\to(-\infty,\infty]$ is a convex potential
of $\mu$, taken to be lower semicontinuous, and
$0<Z<\infty$.

We first derive a linear lower bound for $V$.
Continuity on the interior of its domain gives a level $L$
for which $K_L=\{V\le L\}$ contains a ball.
For every $t\ge L$,
\[
 |K_t|\le e^t\int e^{-V(x)}\,dx=e^tZ<\infty.
\]
An unbounded convex set containing a ball has infinite volume,
as can be seen by forming cones with arbitrarily large height.
Thus $K_t$ is bounded, and consequently every sublevel set
of $V$ is bounded.
Lower semicontinuity on the compact set $K_L$ gives a global
minimizer $x_*$, with finite value $m=V(x_*)$.

We choose $R>0$ so that $K_{m+1}\subset B(x_*,R)$.
If $r=|x-x_*|\ge R$ and $V(x)<\infty$, convexity gives
\[
 m+1
 <V\!\left(x_*+\frac Rr(x-x_*)\right)
 \le m+\frac Rr\bigl(V(x)-m\bigr).
\]
Hence $V(x)>m+r/R$.
For $r<R$, we use $V(x)\ge m$.
These bounds also hold when $V(x)=\infty$, and together imply
\begin{equation}\label{eq:static-linear-tail}
 V(x)\ge c|x|-C\qquad(x\in\R^n)
\end{equation}
with $c=1/R$ and $C=1-m+c|x_*|$.

For $0<\varepsilon\le1$, we define
\[
 V^\varepsilon(x)
 =\inf_{y\in\R^n}
       \left\{V(y)+\frac{|x-y|^2}{2\varepsilon}\right\}.
\]
By \eqref{eq:static-linear-tail}, the function being minimized tends to \(+\infty\) as \(\lvert y\rvert\to\infty\). It is also lower semicontinuous and strictly
convex on its domain, so it has a unique minimizer
$p_\varepsilon(x)\in\operatorname{dom}V$.
In particular, $V^\varepsilon$ is finite everywhere.
We set
\[
 g_\varepsilon(x)=\frac{x-p_\varepsilon(x)}{\varepsilon}.
\]

We verify the differentiability and convexity estimates used
below. Put $p=p_\varepsilon(x)$.
For $V(z)<\infty$ and $0<t<1$, minimality at $p$ and convexity
of $V$ give
\[
 0\le t\bigl(V(z)-V(p)\bigr)
   -\frac t\varepsilon\langle x-p,z-p\rangle
   +\frac{t^2}{2\varepsilon}|z-p|^2.
\]
Dividing by $t$ and letting $t\downarrow0$, we obtain
\begin{equation}\label{eq:static-subgradient}
 V(z)\ge V(p)+\langle g_\varepsilon(x),z-p\rangle.
\end{equation}
This inequality also holds when $V(z)=\infty$.
Applying \eqref{eq:static-subgradient} at $x$ and $x'$,
with the respective choices
$z=p_\varepsilon(x')$ and $z=p_\varepsilon(x)$, gives $\langle g_\varepsilon(x)-g_\varepsilon(x'), p_\varepsilon(x)-p_\varepsilon(x')\rangle\ge0.$ Since $x-x'=p_\varepsilon(x)-p_\varepsilon(x')+\varepsilon\bigl(g_\varepsilon(x)-g_\varepsilon(x')\bigr),$ Cauchy--Schwarz yields
 \begin{align}
 \varepsilon|g_\varepsilon(x)-g_\varepsilon(x')|^2
 \le\langle g_\varepsilon(x)-g_\varepsilon(x'),x-x'\rangle\le|g_\varepsilon(x)-g_\varepsilon(x')|\,|x-x'|.
 \end{align}
Thus $g_\varepsilon$ is $\varepsilon^{-1}$-Lipschitz.

Using $p_\varepsilon(x)$ and $p_\varepsilon(x+h)$ as
competitors in the two minimization problems gives
\[
 \begin{aligned}
 V^\varepsilon(x+h)-V^\varepsilon(x)
 &\le\langle g_\varepsilon(x),h\rangle
                +\frac{|h|^2}{2\varepsilon},\\
 V^\varepsilon(x+h)-V^\varepsilon(x)
 &\ge\langle g_\varepsilon(x+h),h\rangle
                -\frac{|h|^2}{2\varepsilon}.
 \end{aligned}
\]
The Lipschitz bound for $g_\varepsilon$ therefore proves that
$V^\varepsilon$ is continuously differentiable, with
$\nabla V^\varepsilon=g_\varepsilon$.
Convexity follows by using
$(1-t)p_\varepsilon(x)+tp_\varepsilon(x')$ as a competitor
at $(1-t)x+tx'$ and applying convexity of $V$ and the squared
norm. Combining convexity with the first inequality above gives
\begin{equation}\label{eq:static-envelope-remainder}
 0\le V^\varepsilon(x+h)-V^\varepsilon(x)
       -\langle\nabla V^\varepsilon(x),h\rangle
 \le\frac{|h|^2}{2\varepsilon}.
\end{equation}

We choose a nonnegative, even function
$\chi\in C_c^\infty(\R^n)$, supported in the unit ball,
with $\int\chi=1$, and set
\[
 \chi_\varepsilon(y)=\varepsilon^{-n}\chi(y/\varepsilon),
 \qquad
 W_\varepsilon(x)
   =(V^\varepsilon*\chi_\varepsilon)(x)
      +\frac\varepsilon2|x|^2.
\]
The potential $W_\varepsilon$ is smooth.
Adding \eqref{eq:static-envelope-remainder} at $h$ and $-h$
gives
\[
 0\le V^\varepsilon(x+h)+V^\varepsilon(x-h)
          -2V^\varepsilon(x)
 \le\frac{|h|^2}{\varepsilon}.
\]
We convolve this inequality with $\chi_\varepsilon$, set
$h=tv$ for a unit vector $v$, divide by $t^2$, and let
$t\to0$. We obtain
\begin{equation}\label{eq:static-hessian-bounds}
 \varepsilon I\preceq D^2W_\varepsilon(x)
       \preceq(\varepsilon^{-1}+\varepsilon)I.
\end{equation}

We next verify polynomial growth of the derivatives.
Since $\nabla V^\varepsilon$ is Lipschitz,
\[
 |\nabla V^\varepsilon(x)|
 \le|\nabla V^\varepsilon(0)|+\varepsilon^{-1}|x|.
\]
Consequently, $\nabla W_\varepsilon$ has at most linear
growth and $W_\varepsilon$ has at most quadratic growth.
For a multi-index $\alpha$ with $|\alpha|\ge2$, integration
by parts gives
\[
 \int D^\alpha\chi_\varepsilon(y)\,dy=0,\qquad
 \int y_iD^\alpha\chi_\varepsilon(y)\,dy=0
 \quad(1\le i\le n).
\]
It follows that
\[
 \begin{aligned}
 D^\alpha(V^\varepsilon*\chi_\varepsilon)(x)
 =\int\bigl[
     V^\varepsilon(x-y)-V^\varepsilon(x)
     +\langle\nabla V^\varepsilon(x),y\rangle
   \bigr]D^\alpha\chi_\varepsilon(y)\,dy.
 \end{aligned}
\]
By \eqref{eq:static-envelope-remainder}, its absolute value
is at most
\[
 \frac1{2\varepsilon}
       \int |y|^2|D^\alpha\chi_\varepsilon(y)|\,dy<\infty,
\]
uniformly in $x$.
The added quadratic has constant second derivatives and
vanishing derivatives of higher order.
Thus every derivative of $W_\varepsilon$ has polynomial
growth.

We now prove convergence of the potentials.
The minimizing identity and $V\ge m$ give
\[
 \frac{|x-p_\varepsilon(x)|^2}{2\varepsilon}
 \le V^\varepsilon(x)-m.
\]
If $V(x)<\infty$, the competitor $y=x$ gives
$V^\varepsilon(x)\le V(x)$, so $p_\varepsilon(x)\to x$.
Lower semicontinuity then implies
\[
 V(x)\le\liminf_{\varepsilon\downarrow0}V^\varepsilon(x)
 \le\limsup_{\varepsilon\downarrow0}V^\varepsilon(x)
 \le V(x).
\]
If $V(x)=\infty$ and $V^\varepsilon(x)$ were bounded along
a sequence tending to zero, the same minimizing identity
would give $p_\varepsilon(x)\to x$ and bounded values of
$V(p_\varepsilon(x))$, contradicting lower semicontinuity.
Hence $V^\varepsilon(x)\to V(x)$ everywhere, with the value
$\infty$ allowed.

Evenness of $\chi_\varepsilon$ gives
$\int y\chi_\varepsilon(y)\,dy=0$, so Jensen's inequality
yields $(V^\varepsilon*\chi_\varepsilon)(x)\ge V^\varepsilon(x).$ For $x\in\operatorname{int}(\operatorname{dom}V)$ and
sufficiently small $\varepsilon$, we also have $(V^\varepsilon*\chi_\varepsilon)(x)
 \le\sup_{|y|\le\varepsilon}V(x-y).$ Continuity of $V$ on the interior of its domain therefore
gives $W_\varepsilon(x)\to V(x)$ there.
Where $V(x)=\infty$, the lower bound gives the same conclusion.
Since the boundary of the domain has measure zero,
\[
 e^{-W_\varepsilon(x)}\longrightarrow e^{-V(x)}
 \quad\text{for Lebesgue-almost every }x.
\]

To justify integration of these limits, we apply
\eqref{eq:static-linear-tail} and the triangle inequality:
 \begin{align}
 V^\varepsilon(x) \ge c|x|-C+
       \inf_{y\in\R^n}
       \left\{\frac{|x-y|^2}{2\varepsilon}-c|x-y|\right\}=c|x|-C-\frac{\varepsilon c^2}{2}.
 \end{align}
Since $W_\varepsilon\ge V^\varepsilon$ and
$0<\varepsilon\le1$, this gives $e^{-W_\varepsilon(x)}
 \le e^{C+c^2/2}e^{-c|x|}.$ The right side remains integrable after multiplication by
$1+|x|^2$.
We define
\[
 Z_\varepsilon=\int e^{-W_\varepsilon(x)}\,dx,\qquad
 d\widehat\mu_\varepsilon
   =Z_\varepsilon^{-1}e^{-W_\varepsilon(x)}\,dx.
\]
Dominated convergence gives $Z_\varepsilon\to Z$ and weak
convergence $\widehat\mu_\varepsilon\to\mu$.
It also gives convergence of the first and second moments.
Thus, writing $m_\varepsilon$ and $A_\varepsilon$ for the mean
and covariance of $\widehat\mu_\varepsilon$, we have $m_\varepsilon\to0$ and $A_\varepsilon\to I$. Each $A_\varepsilon$ is positive definite because
$\widehat\mu_\varepsilon$ has a positive density on $\R^n$.

We restore isotropy by defining $\mu_\varepsilon$ as the law of
\[
 Y_\varepsilon
   =A_\varepsilon^{-1/2}(X_\varepsilon-m_\varepsilon),
 \qquad X_\varepsilon\sim\widehat\mu_\varepsilon.
\]
This law has mean zero and covariance $I$.
Its potential, up to an additive normalizing constant, is $\widetilde W_\varepsilon(y)
   =W_\varepsilon(m_\varepsilon+A_\varepsilon^{1/2}y).$ The chain rule and \eqref{eq:static-hessian-bounds} give
\[
 \varepsilon\lambda_{\min}(A_\varepsilon)I
 \preceq D^2\widetilde W_\varepsilon(y)
 \preceq(\varepsilon^{-1}+\varepsilon)
           \lambda_{\max}(A_\varepsilon)I.
\]
Affine substitution preserves polynomial growth of the
derivatives. These bounds show that $\mu_\varepsilon$ is
log-concave and regular in the sense of
Definition~\ref{def:regular-law}.

Finally, $A_\varepsilon\to I$ implies
$A_\varepsilon^{-1/2}\to I$.
For every bounded continuous function $f$, we therefore have $f\!\left(A_\varepsilon^{-1/2}(x-m_\varepsilon)\right)
 \longrightarrow f(x)$ at every $x$. Using the common exponential bound and
$Z_\varepsilon\to Z>0$, dominated convergence gives
 \begin{align}
 \int f\,d\mu_\varepsilon=\frac1{Z_\varepsilon}
   \int f\!\left(A_\varepsilon^{-1/2}(x-m_\varepsilon)\right)
          e^{-W_\varepsilon(x)}\,dx\longrightarrow
   \frac1Z\int f(x)e^{-V(x)}\,dx
   =\int f\,d\mu.
 \end{align}
Hence $\mu_\varepsilon$ converges weakly to $\mu$.
\end{proof}

The extension to finite-energy functions uses the standard combination of
mollification, spatial cutoffs, and truncation. The truncation argument
removing the a priori assumption of square integrability is described in
the one-dimensional setting by~\citet[Section~2.1, discussion following Definition~2]{RoustantBartheIooss}.
We include the details for absolutely continuous measures on $\R^n$,
together with the Sobolev variant for smooth positive densities.

\begin{lemma}\label{lem:static-finite-energy}\label{finite-energy-extension}
Let $\mu$ be an absolutely continuous probability measure on $\R^n$, and let $0\le C<\infty$. Assume
\begin{equation}\label{eq:static-smooth-poincare}
 \int u^2\,d\mu-\left(\int u\,d\mu\right)^2
       \le C\int|\nabla u|^2\,d\mu
       \qquad\bigl(u\in C_c^\infty(\R^n)\bigr).
\end{equation}
If $f:\R^n\to\R$ is locally Lipschitz and $E:=\int|\nabla f|^2\,d\mu<\infty,$ then $f\in L^2(\mu)$ and
\begin{equation}\label{eq:static-full-poincare}
 \int\left|f-\int f\,d\mu\right|^2\,d\mu
       \le C\int|\nabla f|^2\,d\mu.
\end{equation}
Here local Lipschitz regularity is required on the ambient space $\R^n$. For a measure on a smaller affine hull, the same statement applies in coordinates on that affine hull.
If the density of $\mu$ is smooth and strictly positive on $\R^n$, the conclusion also holds for every $f\in W_{\mathrm{loc}}^{1,2}(\R^n)$ whose weak gradient has finite $L^2(\mu)$ norm.
\end{lemma}

\begin{proof}
We first extend \eqref{eq:static-smooth-poincare} to compactly supported Lipschitz functions. Let $u:\R^n\to\R$ have compact support and Lipschitz constant $L$. Use the nonnegative smooth kernel $\chi_\varepsilon$ from the proof of Lemma~\ref{lem:static-regular-approximation}, and put $u_\varepsilon=u*\chi_\varepsilon$. Then $u_\varepsilon\in C_c^\infty(\R^n)$ and $|u_\varepsilon(x)-u(x)| \le \int L|y|\chi_\varepsilon(y)\,dy\le L\varepsilon.$ A Lipschitz function has a weak gradient agreeing with its pointwise gradient almost everywhere, with $|\nabla u|\le L$. Differentiating convolution in the weak sense gives
\[
 \nabla u_\varepsilon=(\nabla u)*\chi_\varepsilon,
       \qquad |\nabla u_\varepsilon|\le L.
\]
The Lebesgue differentiation theorem implies $\nabla u_\varepsilon\to\nabla u$ at Lebesgue-almost every point. Absolute continuity of $\mu$ gives the same convergence $\mu$-almost everywhere. The bound $|\nabla u_\varepsilon|\le L$, and the fact that $\mu$ is a probability measure, permit dominated convergence in the energy integral. Uniform convergence of $u_\varepsilon$ gives convergence of its first and second moments. Applying \eqref{eq:static-smooth-poincare} to $u_\varepsilon$ and taking the limit proves it for $u$.

This argument does not require a lower bound for the density of $\mu$ near the boundary of its support. The function being convolved is defined on all of $\R^n$; its gradients are bounded independently of $\varepsilon$, and convergence is first established Lebesgue-almost everywhere.

Next, let $h:\R^n\to\R$ be bounded and locally Lipschitz, with $\nabla h\in L^2(\mu)$, and choose $\eta\in C_c^\infty(\R^n)$ with $0\le\eta\le1$, equal to one on the unit ball and zero outside the ball of radius two. Set $\eta_R(x)=\eta(x/R)$ and  $ h_R=\eta_R h.$ Since $h$ is locally Lipschitz, it is Lipschitz on $\overline{B(0,3R)}$. Multiplication by the smooth cutoff $\eta_R$, which vanishes outside $B(0,2R)$, therefore makes $h_R$ a compactly supported
Lipschitz function.

Dominated convergence gives $h_R\to h$ in $L^2(\mu)$. Moreover, $\nabla h_R-\nabla h=(\eta_R-1)\nabla h+h\nabla\eta_R.$ The first term tends to zero in $L^2(\mu)$ by dominated convergence. The second satisfies
\[
 \|h\nabla\eta_R\|_{L^2(\mu)}
       \le \frac{\|h\|_\infty\|\nabla\eta\|_\infty}{R}
       \longrightarrow0.
\]
We may therefore pass to the limit in the compact-support inequality for $h_R$, obtaining \eqref{eq:static-full-poincare} for every such bounded $h$.

It remains to prove integrability of a finite-energy function $f$. We choose a finite median $m$ of $f$, so that
\[
 \mu(f\le m)\ge\tfrac12,\qquad \mu(f\ge m)\ge\tfrac12.
\]
A finite median exists because $f$ is real-valued: it can be defined as the infimum of the real numbers $s$ for which $\mu(f\le s)\ge1/2$. For $N>0$, let $ f_N=\max\{-N,\min\{f-m,N\}\}.$ These functions are bounded and locally Lipschitz, and their gradients satisfy $|\nabla f_N|\le|\nabla f|$ almost everywhere. The gradient bound follows from the fact that the clipping map is $1$-Lipschitz: at every point where $f$ and $f_N$ are differentiable, take difference quotients in each direction and use this Lipschitz bound. Both functions are differentiable almost everywhere. Since $f_N$ has median zero under $\mu$, we apply
\eqref{eq:median-second-moment} and then
\eqref{eq:static-full-poincare}, which has already been
proved for bounded locally Lipschitz functions, to obtain
\[
 \int f_N^2\,d\mu
 \le2\operatorname{Var}_\mu(f_N)
 \le2C\int|\nabla f_N|^2\,d\mu
 \le2CE.
\]
Since $f_N\to f-m$ pointwise, Fatou's lemma gives $f-m\in L^2(\mu)$, hence $f\in L^2(\mu)$. Finally, $|f_N|\le|f-m|$ and $|f_N-(f-m)|\le|f-m|$, so dominated convergence gives $f_N\to f-m$ in $L^2(\mu)$. Their variances therefore converge to the variance of $f$. Taking the limit in their variance bound proves \eqref{eq:static-full-poincare}.

We finish with the Sobolev assertion. Write $d\mu=\rho\,dx$, with $\rho$ smooth and strictly positive. On each fixed compact set $K$, both $\rho$ and $\rho^{-1}$ are bounded, so the weighted and unweighted $L^2$ norms are equivalent there. If $u$ is compactly supported and belongs to $W_{\mathrm{loc}}^{1,2}(\R^n)$, its mollifications converge to $u$ in unweighted $W^{1,2}(\R^n)$ and have support in a fixed enlarged compact set. To see the convergence directly, apply
\[
 \|v*\chi_\varepsilon-v\|_{L^2(dx)}
 \le\int\chi_\varepsilon(y)
               \|v(\,\cdot-y)-v\|_{L^2(dx)}\,dy
       \longrightarrow0
\]
to $v=u$ and to each weak derivative of $u$, using continuity of translations in $L^2(dx)$. The bounds on $\rho$ then give convergence in the weighted function and gradient norms. This proves \eqref{eq:static-smooth-poincare} for compactly supported Sobolev functions. Bounded Sobolev functions are treated by the same cutoffs $h_R=\eta_Rh$ and the same two energy estimates as above. Finally, Sobolev clipping preserves local $W^{1,2}$ regularity and satisfies $|\nabla f_N|\le|\nabla f|$ almost everywhere. One can obtain this bound by replacing the clipping map by smooth maps with derivatives bounded by one: the chain rule for a smooth map follows by mollifying $f$ on compact sets, and weak lower semicontinuity of the gradient norm on each measurable subset gives the stated bound in the clipping limit. The median, Fatou, and $L^2$ convergence arguments above consequently apply without change and prove the Sobolev assertion.
\end{proof}

\end{document}